%% file: main.tex
\documentclass[onefignum,onetabnum]{siamart190516}

\input{ex_shared}

\ifpdf
\hypersetup{
  pdftitle={The geometry of the Painlev\'e paradox},
  pdfauthor={N. Cheesman, S.J. Hogan, K. Uldall Kristiansen}
}
\fi

\begin{document}

\maketitle

\begin{abstract}
Painlev\'e showed that there can be inconsistency and indeterminacy in solutions to the equations of motion of a $2D$ rigid body moving on a sufficiently rough surface.
The study of Painlev\'e paradoxes in $3D$ has received far less attention. In this paper, we highlight the pivotal role in the dynamics of the azimuthal angular velocity $\Psi$ by proving the existence of three critical values of $\Psi$, one of which occurs independently of any paradox. We show that the $2D$ problem is highly singular and uncover a rich geometry in the $3D$ problem which we use to explain recent numerical results.
\end{abstract}


\begin{AMS}
  { 37N15, 70E18, 74H20, 74H25
}\end{AMS}

\section{Introduction}
Painlev\'{e} \cite{Painleve1895, Painleve1905a, Painleve1905b} showed that there can be \textit{inconsistency} (non-existence) and \textit{indeterminacy}  (non-uniqueness) in solutions to the equations of motion of a rigid body moving along a rough surface, known as \textit{Painlev\'{e} paradoxes}, when the coefficient of dry (Coulomb) friction $\mu$ exceeds a certain critical value $\mu_\mathrm{P}$.
A paradox occurs because the moment on the rigid body caused by friction can be large enough to drive the body \textit{into} the rigid surface, which is impossible.
Painlev\'{e} \cite{Painleve1905a, Painleve1905b} studied planar ($2D$) motion of a slender rod slipping along a rough surface, where $\mu_\mathrm{P} = \frac{4}{3}$ if the rod is uniform.
Related problems in statics had already been found in 1872 by Jellet \cite{Jellet72}, but the discovery by Painlev\'e led to celebrated discussion about the limits of rigid body theory \cite{Lecornu1905a, Lecornu1905b, Sparre1905, Klein1909, Mises1909, Hamel1909, Prandtl1909,Delassus1920, Delassus1923, Beghin1923, Beghin1924}, especially as the critical value $\mu=\mu_\mathrm{P}$ was so large. 

Interest in Painlev\'e paradoxes has been revived more recently, where they have been shown to occur in many important engineering systems, where $\mu_\mathrm{P}$, which is a function of system parameters, can be considerably lower \cite{WilmsCohen1981, Lotstedt1981, Ivanov1986, Lesuanan1990, NeimarkFufayev1995, WilmsCohen1997,Stewart1997}.

The theoretical study of Painlev\'e paradoxes in $2D$ received a great boost with the work by G\'enot and Brogliato \cite{Genot1999}. 
They proved that the rod cannot reach an inconsistent state when slipping, unless the free acceleration at the rod tip vanishes at the same {\it critical point} in phase space. This work led to further research \cite{LeineBrogliatoNijmeijer2002, Ivanov03, Liu2007, ShenStronge2011, Or2012, Or2014, Varkonyi15, Varkonyi2017, Burns2017, Elkaranshawy17,NordVarChamp18}. Experimental evidence of a Painlev\'e paradox has been found in a robotic system \cite{ZhaoLiuMaChen2008}. 

In contrast, the study of Painlev\'e paradoxes in $3D$, where a rigid body moves over a rough surface, has received far less attention. The first paper to explicitly consider such a case appears to be that of Zhao {\emph{et al.}} \cite{Zhao2008}. These authors derived the governing equations using energy methods and confirmed that Painlev\'e paradoxes could occur. Shen \cite{Shen2015} considered a $3D$ elastic rod in rectilinear motion. This version of the problem allows for elastic waves in the rod, whilst still keeping the surface rigid. Champneys and V{\'{a}}rkonyi \cite{Champneys2016} carried out numerical computations of the equations derived by \cite{Zhao2008} and showed that it was possible to enter the inconsistent region from slipping, unlike the $2D$ case \cite{Genot1999}. 

So it is natural to ask the question: how does the planar $2D$ problem, with its critical point and the impossibility of reaching an inconsistent state through slipping, fit into the $3D$ case, where reaching inconsistency through slipping seems to be the default? This paper is organised as follows. In \cref{sec:painleve3D}, we derive the governing equations of a rigid body moving with one point in contact with a rough surface, using force and torque balances. 

In \cref{sec:constraint}, we show where a paradox can occur and, perhaps more importantly, how to avoid one. We highlight the pivotal role of the azimuthal angular velocity $\Psi$. It turns out that there are three critical values of $\Psi$. One critical value $\Psi_\mathrm{L}$ occurs independently of any paradox. For $|\Psi|>\Psi_\mathrm{L}$, the rod can lift off the surface in the absence of any motion in the polar direction of the rod. The other two critical values of $\Psi$ occur only in the presence of a paradox.

In $3D$, the vertical acceleration of the rod is identically zero on the \textit{G\'enot-Brogliato} (or \textit{GB}) manifold \cite{Varkonyi18}. In \cref{sec:slip}, we focus on the slipping dynamics near the GB manifold. In $3D$, the default motion from slipping is into the inconsistent region, with bifurcations in the dynamics occurring around the GB manifold. In \cref{sec:alltogether}, we bring together results from the previous sections to show how the kinematics of the formation of the GB manifold interacts with the dynamics to give a rich geometry of the Painlev\'e paradox in $3D$ and that the $2D$ problem is highly singular.

\section{The \texorpdfstring{Painlev\'{e}}{Painlevé} problem in \texorpdfstring{$3D$}{3D}}
\label{sec:painleve3D}
We consider a slender, rigid rod $AB$ of mass $m$ slipping on a planar rough horizontal surface, as shown in \cref{fig:me2}.

\begin{figure}[!htbp]
    \centering
    \begin{overpic}[tics=5]{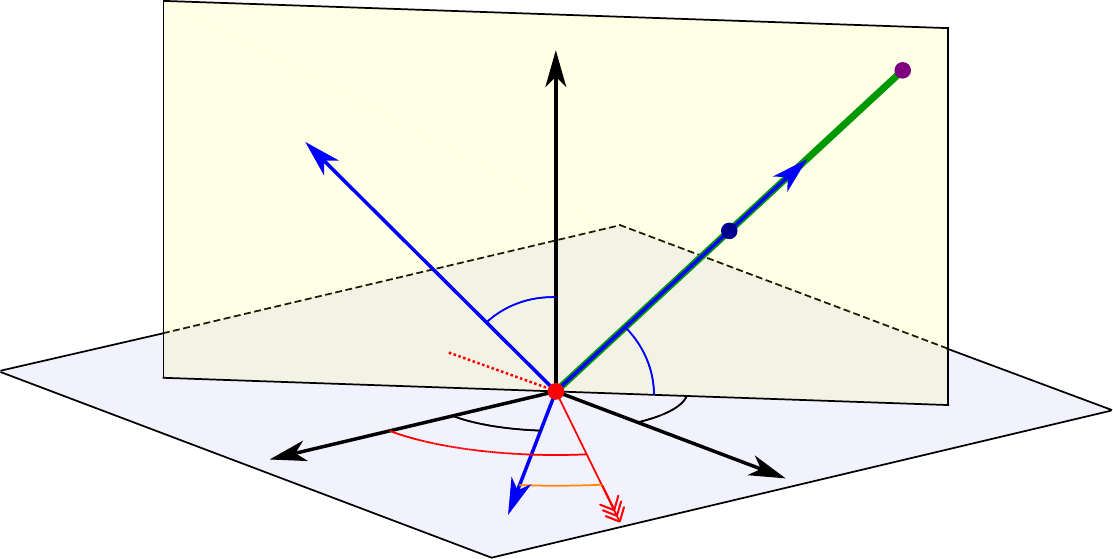}
        \put(25,11){$\vec{I}$}
        \put(68,10){$\vec{J}$}
        \put(46,44){$\vec{K}$}
        \put(44,3){\color{blue}{$\vec{i}$}}
        \put(70,36){\color{blue}{$\vec{j}$}}
        \put(27,32){\color{blue}{$\vec{k}$}}
        \put(40,7){\color{red}{$\beta$}}
        \put(49.5,4){\color[RGB]{255,134,0}{$\varphi$}}
        \put(61.5,12){$\psi$}
        \put(46,12){$\psi$}
        \put(59,17){\color{blue}{$\theta$}}
        \put(45,24){\color{blue}{$\theta$}}
        \put(57,5){\color{red}{$\eta$}}
        \put(26,18){\color{red}{$A:(x,y,0)$}}
        \put(67,28){\color[cmyk]{1,1,0,0.45}{$S:(X,Y,Z)$}}
        \put(76,44){\color[cmyk]{0,1,0,0.5}{$B$}}
\end{overpic}
    \caption{A slender rod $AB$ (shown in green) of mass $m$ moving on a planar rough surface at speed $\eta \ge 0$. Frame of reference $(\vec{I}, \vec{J}, \vec{K})$ is fixed to the surface and $(\vec{i}, \vec{j}, \vec{k})$ is fixed to the body. The distance between the point of contact $A$ and the centre of mass $S$ is given by $l$. \cite{Zhao2008}.\label{fig:me2}}
\end{figure}
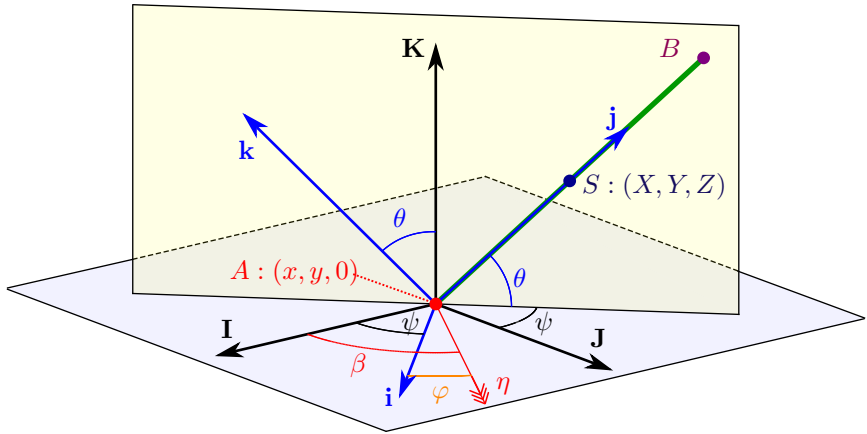

An inertial frame fixed to the surface has Cartesian axes $(\vec{I}, \vec{J}, \vec{K})$. 
The centre of mass of the rod $S$, which is at a distance $l$ from the tip $A$ in contact with the surface, has coordinates $(X, Y, Z)$ in the inertial frame. Axes fixed to the rod, and parallel to its principal axes, are denoted by $(\vec{i}, \vec{j}, \vec{k})$. Point $A$ has coordinates $(x,y,0)$ in the inertial frame.  

The rod is inclined at an angle $\theta$ to the horizontal, the \textit{polar angle} between the $\vec{k}$ axis and the $\vec{K}$ axis. The \textit{azimuthal angle} between the vertical plane containing the rod and the $\vec{J}$ axis is denoted by $\psi$. The rod moves on the rough surface with a variable speed $\eta$ and at an angle $\beta$ to the $\vec{I}$ axis - an angle $\varphi=\beta-\psi$ to the $\vec{i}$ axis. We refer to $\beta$ as the \textit{slip angle} (or \textit{heading}) and $\varphi$ as the \textit{relative slip angle} (or \textit{relative heading}). We exclude the cases of the vertical rod and horizontal rod. So we take $\theta\in(0,\pi/2)$, $\psi\in[-\pi,\pi)$, $\beta\in[-\pi,\pi)$ and $\varphi \in (-\pi,0]$. 

Point $A$ moves with velocity $(u, v, 0)$ where 
\begin{align}
    \label{eq:uv}
    \begin{split}
        \dot{x} = u & = \eta \cos \beta, \\
        \dot{y} = v & = \eta \sin \beta
    \end{split}
\end{align}
where a dot denotes differentiation with respect to time, and $\eta =\sqrt{u^2 + v^2}$.

\subsection{Equations of motion of the slender rod}
\label{subsec:eqmotion}
We derive the equations of motion of the slender rod as it moves in contact with the rough surface, using force and torque balances, in contrast to the energy method used in \cite{Zhao2008}. The moment of inertia tensor $\mat{I}$ of the rod with respect to its principal axes is given by
\begin{equation}
    \label{eq:inertia}
    \mat{I}=
    \begin{pmatrix}
        I_{0} & 0 & 0      \\
        0     & 0 & 0      \\
        0     & 0 & I_{0} 
    \end{pmatrix}
\end{equation}
where $I_0 = \frac{1}{3}ml^2$ for a uniform rod. The matrix $\mat{I}$ is singular due to the assumed slenderness of the rod.


The angular velocity $\vec{\Omega}$ of the rod is given by 
\begin{align}
	\vec{\Omega}&=\dot{\theta}\,\vec{i}+\dot{\psi}\,\vec{K}
\intertext{or}
	\vec{\Omega}&=\dot{\theta}\,\vec{i}+\dot{\psi}\sin{\theta}\,\vec{j}+\dot{\psi}\cos{\theta}\,\vec{k}.
    \label{eq:angvel}
\end{align}

The angular momentum in the rod frame is given by $\vec{H}=\mat{I}\vec{\omega}$, and its rate of change is given by $\dot{\vec{H}} = \mat{I} \, \dot{\vec{\Omega}} + \vec{\Omega} \times (\mat{I} \, \vec{\Omega})$. Using \cref{eq:inertia,eq:angvel}, we find
\begin{equation}\label{eq:moments}
    \dot{\vec{H}}=
    \begin{pmatrix}
        I_{0}(\ddot{\theta}+\dot{\psi}^2 \sin\theta\cos\theta  )\\
        0          \\
        I_{0}(\ddot{\psi}\cos\theta-2\dot{\theta}\dot{\psi}\sin\theta   )
    \end{pmatrix}
\end{equation}

Euler's rotation equations state that this rate of change of angular momentum $\dot{\vec{H}}$ is equal to the applied torque $\vec{G}$ of the forces acting on the rod. The force on the rod in the inertial frame is denoted by
$\vec{F}=(F_x, F_y, F_z)^\mathsf{T}$. Then $\vec{G} = \vec{r}\times (\mat{T} \, \vec{F})$, where $\vec{r}=(0,-l,0)^\mathsf{T}$ is the position vector of the contact point $A$ relative to $S$ in the body frame, and $\mat{T}$ is the rotation matrix from the inertial frame to the body frame given by
\begin{align}
    \mat{T}&=\left[\begin{pmatrix}
    \cos\psi            & -\sin\psi             & 0 \\
    \sin\psi            & \cos\psi              & 0 \\
    0                   & 0                     & 1 \\
\end{pmatrix}\begin{pmatrix}
    1                   & 0                     & 0 \\
    0                   & \cos\theta            & -\sin\theta \\
    0                   & \sin\theta            & \cos\theta \\
\end{pmatrix}\right]^\intercal\label{eq:rotmatrix}.
\end{align}
So we find
\begin{equation}
    \label{eq:torques}
    \vec{G}=\vec{r} \times (\mat{T} \, \vec{F})=
    \begin{pmatrix}
    -l(\sin\psi\sin\theta F_x -\cos\psi\sin\theta F_y +\cos\theta F_z)\\
    0\\
    l(\cos\psi F_x +\sin\psi F_y)
    \end{pmatrix},
\end{equation}
which is zero in the $\vec{j}$ component due to the slenderness of the rod.

Equating \cref{eq:moments} and \cref{eq:torques}, we find 
\begin{align}
    \label{eq:rotode}
    \begin{split}
        \ddot{\theta} & = -\frac{l}{I_0} \left (\sin\psi\sin\theta F_x - \cos\psi\sin\theta F_y+\cos\theta F_z \right ) - \dot{\psi}^2\sin\theta\cos\theta,\\
        \ddot{\psi} & =\frac{l}{I_0 \cos\theta} \left ( \cos\psi F_x +\sin\psi F_y \right ) + 2 \dot{\psi}\dot{\theta}\tan\theta.
    \end{split}
\end{align}
For the case of a uniform rod, \cref{eq:rotode} are identical to \cite[eq. (12)]{Zhao2008}, which were derived directly from the energy equations.

The rectilinear dynamics are determined as follows. Using Newton's second law in the inertial frame, we have for the centre of mass $S$ of the rod
\begin{equation}
    \label{eq:linode}
    m \ddot{X}=F_x,\qquad
    m \ddot{Y}=F_y,\qquad
    m \ddot{Z}=F_z-mg.
\end{equation}
We are interested in the dynamics of the rod tip $A$, whose coordinates are given by
\begin{equation}\label{eq:tipfromcentre}
    x=X+l\cos{\theta}\sin{\psi},\qquad
    y=Y-l\cos{\theta}\cos{\psi},\qquad
    z=Z-l\sin{\theta}.
\end{equation}
Differentiating \cref{eq:tipfromcentre} twice with respect to time  and substituting the results from \cref{eq:rotode,eq:linode} gives the following
\begin{align}
    \begin{split}\label{eq:linode2}
        \ddot{x}=&\frac{ml^2}{I_0} \left[  { \left(- \sin^2 \psi  \cos^{2} \theta  +\frac{I_{0}}{m l^2}+1 \right) }\frac{F_{x}}{m}
        +   { \left(\cos^2 \theta \sin  \psi    \cos  \psi   \right)} \frac{F_{y}}{m} 
        +   { \left(\sin \theta \cos  \theta    \sin  \psi   \right)} \frac{F_{z}}{m} \right. \\
        &\qquad\quad+ \left.\frac{I_0}{ml}\left( -\, \dot{\psi} ^{2}  \cos^{3}  \theta   \sin  \psi - \dot{\theta} ^{2}\cos\theta  \sin  
        \psi \right)\right],\\
        \ddot{y}=&\frac{ml^2}{I_0} \left[  { \left( \sin \psi  \cos \psi \cos^2 \theta   \right) }\frac{F_{x}}{m}
        +   { \left(- \cos^2 \psi  \cos^{2} \theta  +\frac{I_{0}}{m l^2}+1   \right)} \frac{F_{y}}{m} 
        -   { \left(\sin \theta \cos  \theta    \cos  \psi   \right)} \frac{F_{z}}{m} \right. \\
        &\qquad\quad+ \left.\frac{I_0}{ml}\left( \, \dot{\psi} ^{2}  \cos^{3}  \theta   \cos  \psi + \dot{\theta} ^{2} \cos\theta  \cos \psi \right)\right],\\
        \ddot{z}=&\frac{ml^2}{I_0} \left[  { \left( \sin \psi  \sin \theta \cos \theta   \right) }\frac{F_{x}}{m}
        -   { \left( \cos \theta  \cos \psi \sin \theta   \right)} \frac{F_{y}}{m} 
        +   { \left(\cos^2 \theta +\frac{I_0}{ml^2}  \right)} \frac{F_{z}}{m} \right. \\
        &\qquad\quad+ \left.\frac{I_0}{ml}\left( \, \dot{\psi} ^{2}  \cos^{2}  \theta   \sin  \theta + \dot{\theta} ^{2} \sin\theta - \frac{g}{l} \right)\right],\\
    \end{split}
\end{align}
which are identical to \cite[eq. (13)]{Zhao2008} for the case of a uniform rod.

In order to proceed, we must specify $F_x$, $F_y$ and $F_z$. We assume Coulomb friction, and hence when slipping
\begin{equation}\label{eq:friction}
    F_x=-\mu \cos \beta F_z,\qquad F_y=-\mu \sin \beta F_z,
\end{equation}
where $\mu$ is the coefficient of friction between the rod and the surface and $\beta$ is the slip angle (see \cref{fig:me2}). 

It is useful to write equations of motion \cref{eq:uv,eq:rotode,eq:linode2} in terms of $\eta$, the speed of the contact point $A$, and $\varphi=\beta-\psi$, the relative slip angle. In addition, following \cite{Hogan2017}, we adopt the scalings $(x,y,z)=l(\,\tilde{x},\tilde{y},\tilde{z})$, $(F_x,F_y,F_z)=mg(\tilde{F}_x,\tilde{F}_y,\tilde{F}_z)$,
$\alpha=ml^2/I_0$ and $t=\tilde{t}/\omega$, where $\omega^2=g/l$. Note that $\alpha=3$ corresponds to a uniform slender rod.

Dropping the tildes, \cref{eq:uv,eq:rotode,eq:linode2,eq:friction} can be written as the following set of first order ordinary differential equations
\begin{equation}\label{eq:polareqs_varphi}
\begin{aligned}
    \dot{x}&=u=\eta\cos(\psi+\varphi),\qquad&\dot{\eta}&=Q_1 F_z+A_1,\\
    \dot{y}&=v=\eta\sin(\psi+\varphi), &\eta\dot{\varphi}&=Q_2 F_z+A_2-\eta\Psi,\\
    \dot{z}&=w, & \dot{w}&=p F_z+b,\\
    \dot{\psi}&=\Psi, & \dot{\Psi}&=d_1 F_z+c_1,\\
    \dot{\theta}&=\Theta,  &\dot{\Theta}&=d_2 F_z+c_2,
\end{aligned}
\end{equation}
where
\begin{align}
b(\Psi,\Theta,\theta) & := \left( \Psi ^{2}  \cos^{2}  \theta  + \Theta ^{2} \right)\sin\theta - 1, \label{eq:b}\\
    p(\theta,\varphi; \mu, \alpha) & := 1+\alpha+\alpha\sin\theta\left(\mu\cos\theta\sin\varphi-\sin\theta\right). \label{eq:p}
\end{align} 

The term $b(\Psi,\Theta,\theta)$ in the $\dot{w}$ equation is the scaled vertical component of free acceleration of the rod tip $A$.
The \textit{dimensional} form of $b$ is given by $l\dot{\psi}^2 \sin \theta \cos^2 \theta + l\dot{\theta}^2\sin \theta - g$.
The first term is due to azimuthal motion, the second term due to polar motion and the third term is the acceleration due to gravity.
The term $p F_z$ is the vertical acceleration of the rod tip resulting from the interaction with the surface
. The coefficients $Q_i(\theta,\varphi; \mu, \alpha), A_i(\Psi,\Theta,\theta,\varphi), d_i(\theta,\varphi; \mu, \alpha), c_i(\Psi,\Theta,\theta) \; (i=1,2)$ are given in \cref{sec:appA}. 
\begin{remark}\label{rem:3d22d}
From \cref{eq:polareqs_varphi}, see also \cref{eq:polareqs_varphi_terms} in \cref{sec:appA}, $\{\Psi=0\}\cap \{\varphi=\pm\pi/2\}$ (the planar Painlev\'e problem) is an invariant manifold.
\end{remark}

\section{Constraint-based approach}
\label{sec:constraint}
In this section, we analyse the rigid body equations \cref{eq:polareqs_varphi}. In particular, we highlight where they break down, compare them with the $2D$ problem and demonstrate the role of the azimuthal angular velocity $\Psi$. Our focus is on the equation of vertical acceleration $\dot{w}=p F_z+b$.

\subsection{Breakdown of rigid body equations}
\label{subsec:break}
The Painlev\'{e} problem can be thought of \cite{Lotstedt1981, Genot1999} as the linear  complementarity problem (LCP)
\begin{equation}\label{eq:comp}
    0\leq F_z \perp z\geq 0.
\end{equation}

When the rod slips along the surface $z=0$, its tip must be in vertical equilibrium, that is, $w=\dot{w}=0$. In this case, from \cref{eq:polareqs_varphi}, we have
\begin{equation}\label{eq:FZ}
    F_z=-\frac{b(\Psi,\Theta,\theta)}{p(\theta,\varphi; \mu, \alpha)}.
\end{equation}

So the LCP \cref{eq:comp} with \cref{eq:FZ} has four modes, dependent on the signs of $b$ and $p$.
\begin{itemize}[leftmargin=0.95in]
    \item[$b<0, \; p>0$:] From \cref{eq:FZ}, we have $F_z>0$ and the rod is \textit{slipping} along the rough surface.
    \item[$b>0, \; p>0$:] Both components of the vertical acceleration $\dot{w}$ acting on the rod tip  are directed away from the surface and \textit{lift-off} occurs.
    \item[$b<0, \; p<0$:] Both the free acceleration $b$ and the vertical acceleration from contact forces $pF_z$  act vertically downwards. This is \textit{inconsistent} with our assumption of a rigid surface and we have \textit{non-existence} of solutions to \cref{eq:polareqs_varphi}.
    \item[$b>0, \; p<0$:] From \cref{eq:FZ}, we have $F_z>0$. But because the free acceleration $b$ is upwards, lift-off is still possible. So the motion of the rod is \textit{indeterminate} and we have \textit{non-uniqueness} of solutions to \cref{eq:polareqs_varphi}.
\end{itemize}

Hence we need $p<0$ to have a paradox. From \cref{eq:p}, it is straightforward to show that $p <0$ when
\begin{equation}\label{eq:mup}
    \mu>\mu_\mathrm{P}^*(\varphi;\alpha)=\frac{2\sqrt{\alpha+1}}{\alpha|\sin\varphi|}.
\end{equation}

Solving $p=0$ for $\theta$, we obtain $\theta=\theta_\pm$, where
\begin{equation}\label{eq:thetastar}
    \theta_{\pm}(\varphi;\mu,\alpha)=\arctan\left(-\frac{1}{2}\mu\alpha\sin\varphi\pm \frac{1}{2}\sqrt{\mu^2 \alpha^2 \sin^2 \varphi-4(\alpha+1)}\right).\\
\end{equation}
Similarly, solving $p=0$ for $\varphi$, we find $\varphi=\varphi_\pm$ where 
\begin{equation}\label{eq:varphistar}
    \varphi_{\mp}(\theta;\mu,\alpha):=-\frac{\pi}{2}\mp\arccos\left(\frac{1+\alpha\cos^2\theta}{\alpha\mu\sin\theta\cos\theta}\right).\\
\end{equation}

\begin{remark}
    When $\varphi=-\pi/2$, \cref{eq:mup,eq:thetastar} reduce to
    \begin{equation}\label{eq:mup2D}
    \mu_\mathrm{P}(\alpha):=\mu_\mathrm{P}^*(-\pi/2;\alpha)=\frac{2}{\alpha}\sqrt{\alpha+1}
\end{equation}
and
    \begin{equation}\label{eq:theta12}
    \begin{split}
        \theta_1(\mu,\alpha):=\theta_-\left(-\frac{\pi}{2};\mu,\alpha\right)= \arctan\left(\frac{1}{2}\mu\alpha- \frac{1}{2}\sqrt{\mu^2 \alpha^2-4(\alpha+1)}\right), \\
        \theta_2(\mu,\alpha):=\theta_+\left(-\frac{\pi}{2};\mu,\alpha\right)=\arctan\left(\frac{1}{2}\mu\alpha+ \frac{1}{2}\sqrt{\mu^2 \alpha^2-4(\alpha+1)}\right),
    \end{split}
\end{equation}
corresponding to \cite[eq. (2.13), (2.14)]{Hogan2017} respectively for the $2D$ Painlev\'{e} problem (\cref{rem:3d22d}). 
\end{remark}
\cref{fig:badsection} shows an example of regions in which $p<0$ for the $3D$ Painlev\'{e} problem.  
In \cref{fig:rod_orientations}, $p<0$ where the rod is within the blue cone and the state of the rod is either inconsistent or indeterminate. 
In \cref{fig:badregion},  we show where $p<0$ in the $(\theta, \varphi)$-plane. The bounds of the blue region in the $(\theta, \varphi)$-plane are important in the sequel. They can be determined as follows. From \cref{eq:p}, it can be shown by differentiation that when $p=0$,
\begin{equation}\label{eq:dvarphidtheta}
    \frac{d\varphi}{d\theta} = \frac{2 (\mu \cos 2\theta \sin \varphi - \sin 2\theta)}{\mu \sin 2\theta \cos \varphi}
\end{equation} 
The blue region in \cref{fig:badregion} is bounded in $\theta$ when $\diff{\theta}{\varphi}=0$, that is when $\varphi = -\frac{\pi}{2}$ since $\varphi\in(-\pi,0)$. The solutions to $p(\theta,-\pi/2)=0$ are $\theta=\theta_{1,2}(\mu,\alpha)$, given in \cref{eq:theta12}. The blue region is bounded in $\varphi$ when $\diff{\varphi}{\theta}=0$. This happens when $\sin \varphi = \frac{1}{\mu}\tan 2\theta$, that is, when $\theta=\theta_\mathrm{P}(\alpha)$ where 
\begin{equation}\label{eq:thetap}
    \theta_\mathrm{P}:=\arctan{\left(\sqrt{1+\alpha}\right)}. 
\end{equation}
Hence $\varphi \in (\varphi_1, \varphi_2)$ in \cref{fig:badregion} where
\begin{equation}\label{eq:varphi12}
    \varphi_{1,2}(\mu,\alpha):=\varphi_{\mp}(\theta_\mathrm{P}(\alpha);\mu,\alpha)=-\frac{\pi}{2}\mp \arccos{\left(\frac{2\sqrt{1+\alpha}}{\alpha\mu}\right)},
\end{equation}
and $p(\theta_\mathrm{P},\varphi_{1,2})=0$.

Therefore, the region where $p<0$ can be written as
\begin{align}
    \big\{(\theta,\varphi)|\varphi\in (\varphi_-(\theta),\varphi_+(\theta)),\theta\in(\theta_1,\theta_2)\big\},
    \intertext{or equivalently}
    \big\{(\theta,\varphi)|\theta\in(\theta_-(\varphi),\theta_+(\varphi)), \varphi\in (\varphi_1,\varphi_2)\big\},
\end{align}
for $\mu>\mu_\mathrm{P}(\alpha)$.
\begin{figure}[htbp]
	\centering
	\begin{subfigure}[t]{0.32\textwidth}
        \begin{overpic}[width=\textwidth]{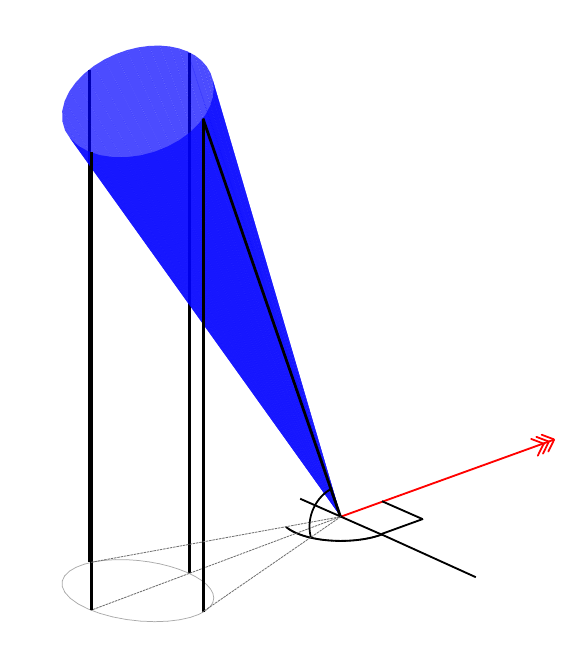}
        	\put(50,10){$-\varphi$}
        	\put(35,23){$\theta$}
        	\put(50,39){\color{red}{Direction of slip}}
        \end{overpic}
        \caption{}
        \label{fig:rod_orientations}
    \end{subfigure}
    \hfil
    \begin{subfigure}[t]{0.5\textwidth}
        \pgfplotstableset{col sep=comma}
\pgfplotsset{scaled x ticks=false}
\pgfplotsset{width=\textwidth,height=0.8\textwidth}
\begin{tikzpicture}
\begin{axis}[enlargelimits=false,axis on top,
    xlabel={$\theta$},
    ylabel={$\varphi$},
    xmin=0, xmax=1.57079632679,
    ymin=-3.14159265358979, ymax=0,
    xtick={0,0.78539816339,0.9702,1.2209,1.57079632679},
    xticklabels={$0$,$\frac{\pi}{4}$,$\theta_1$,$\theta_2$,$\frac{\pi}{2}$},
    ytick={-3.14159265358979,-1.8807,-1.57079632679,-1.2610,0},
    yticklabels={$-\pi$,$\varphi_1$,$-\frac{\pi}{2}$,$\varphi_2$,$0$},
    ]
\addplot graphics[xmin=0,xmax=1.57079633,ymin=-3.14159265258979,ymax=0]{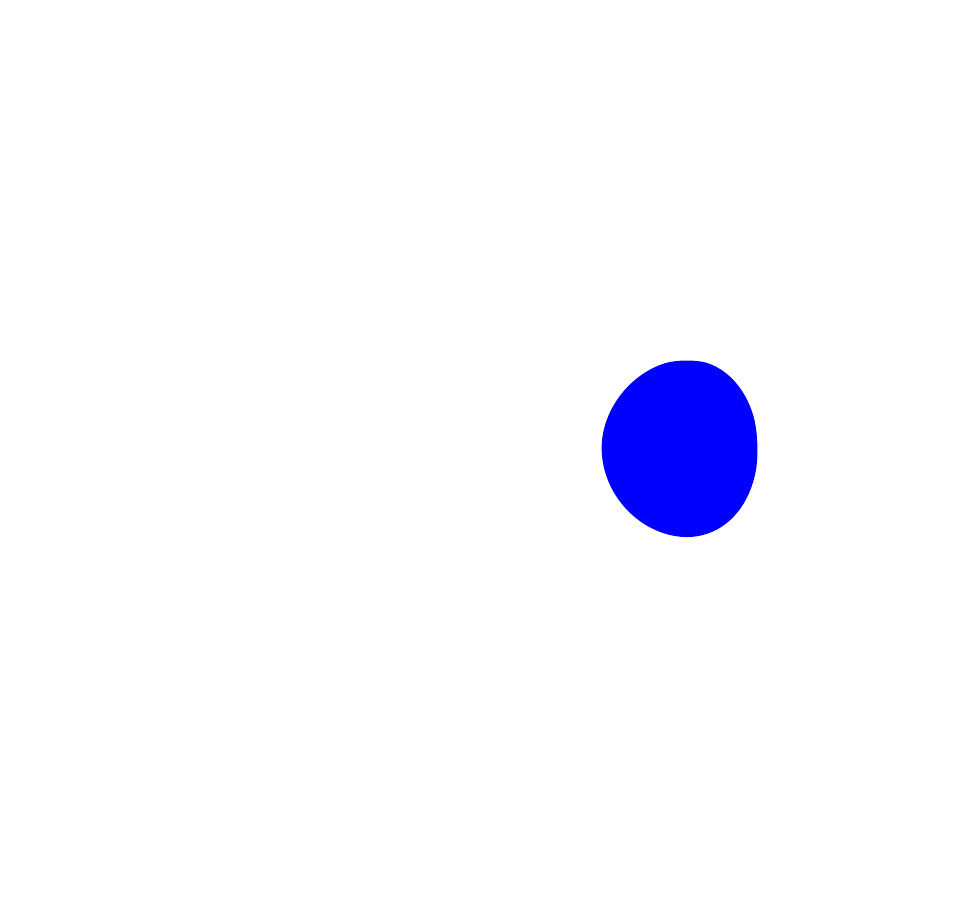};
\addplot[dashed,color=black,line width=0.5pt] coordinates{(0.9702,-3.14159265358979)(0.9702,0)};
\addplot[dashed,color=black,line width=0.5pt] coordinates{(1.2209,-3.14159265358979)(1.2209,0)};
\addplot[dashed,color=black,line width=0.5pt] coordinates{(0,-1.2610)(1.57079632679,-1.2610)};
\addplot[dashed,color=black,line width=0.5pt] coordinates{(0,-1.8807)(1.57079632679,-1.8807)};
\addplot[dashed,color=black,line width=0.5pt] coordinates{(0,-1.57079632679)(1.57079632679,-1.57079632679)};
\end{axis}
\end{tikzpicture}
        \caption{}
        \label{fig:badregion}
    \end{subfigure}
    \caption{Region, shown in blue, inside which $p<0$ for $\alpha=3, \mu=1.4$ where we have either inconsistency or indeterminacy. Here $\theta_1=0.9702$, $\theta_2=1.2209$ from \cref{eq:theta12} and $\varphi_1=-1.8807$, $\varphi_2=-1.2610$ when $\theta=\theta_\mathrm{P}=1.1071$ from \cref{eq:varphi12} and \cref{eq:thetap} respectively: (a) in physical space, (b) in the $(\theta, \varphi)$-plane.}
    \label{fig:badsection}
\end{figure}
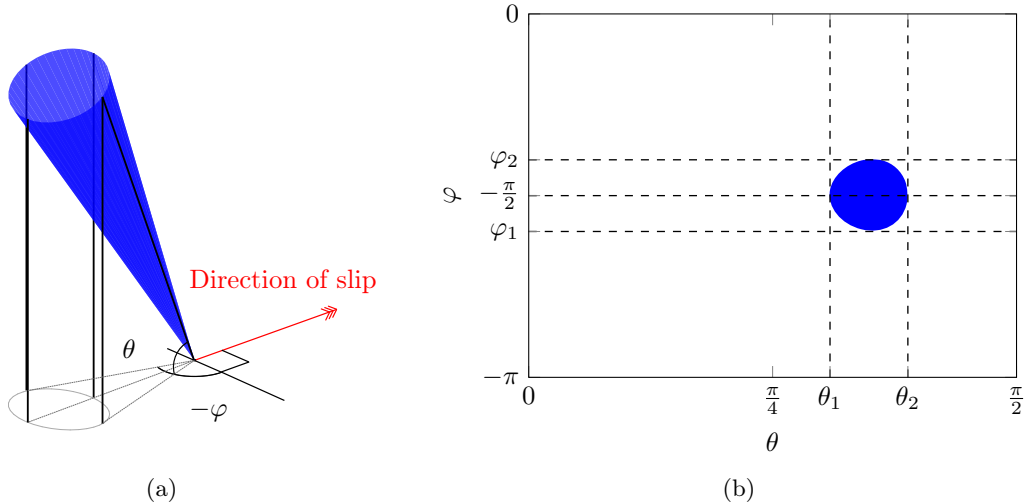

\subsection{Comparison with \texorpdfstring{$2D$}{2D} problem}
\label{subsec:comp}
We compare the $2D$ and $3D$ problems, taking \cite{Genot1999} $\alpha=3$, corresponding to a uniform slender rod and  $\mu=1.4>\mu_P(3)=\frac{4}{3}$ from \cref{eq:mup2D}. 
\cref{fig:p3s} shows, for $\Theta\equiv\dot{\theta}>0$, the codimension-1 surfaces where $b=0$ (in red) and $p=0$ (in blue). For the $2D$ problem (\cref{fig:pandb}), these surfaces are projected into $(\theta,\Theta)$-space, where they intersect at points $P^+$ and $Q^+$ (labelled $P^+_{c1}, P^+_{c2}$ respectively in \cite[fig. 2]{Genot1999}). For the $3D$ problem, these surfaces are shown for fixed $\Psi=0$ and projected into $(\theta,\varphi,\Theta)$-space (\cref{fig:p3.2}), where they intersect at a closed curve. This is the $\Psi=0$ section of the \textit{{G\'enot-Brogliato manifold}} (or \textit{GB manifold}) \cite{Varkonyi18}. In \cref{fig:GBs_bif3}, we will see how this section varies for $\Psi\neq0$. The $2D$ section at $\varphi=-\frac{\pi}{2}$ in \cref{fig:p3.2} gives \cref{fig:pandb}. 

The surfaces $b=0$ and $p=0$, symmetric about $\Theta=0$ and $\varphi=-\frac{\pi}{2}$, segment phase space into four different regions corresponding to the four modes (slipping, lift-off, inconsistent and indeterminate) of the LCP \cref{eq:comp}. In \cref{fig:pandb}, these regions are labelled and a paradox ($p<0$) occurs between the blue lines $p=0$. It is not possible to avoid a paradox whilst increasing $\theta$. In \cref{fig:p3.2}, a paradox occurs inside the blue cylinder $p=0$: \textit{indeterminate} above the surface $b=0$ and \textit{inconsistent} below. Outside the blue cylinder $p=0$, we have \textit{lift-off} above the surface $b=0$ and \textit{slipping} below. From \cref{eq:p}, the shape of $p=0$  is independent of $\Psi$ and hence in $3D$ it is always possible to avoid a paradox, if it occurs, by choosing the relative slip angle $\varphi$ such that either $\varphi \in (-\pi, \varphi_2)$ or $\varphi \in (\varphi_1, 0)$, see also \cref{fig:badregion}.
\begin{figure}[htbp]
    \centering
    \begin{subfigure}[t]{0.41\textwidth}
        \begin{overpic}[width=\textwidth]{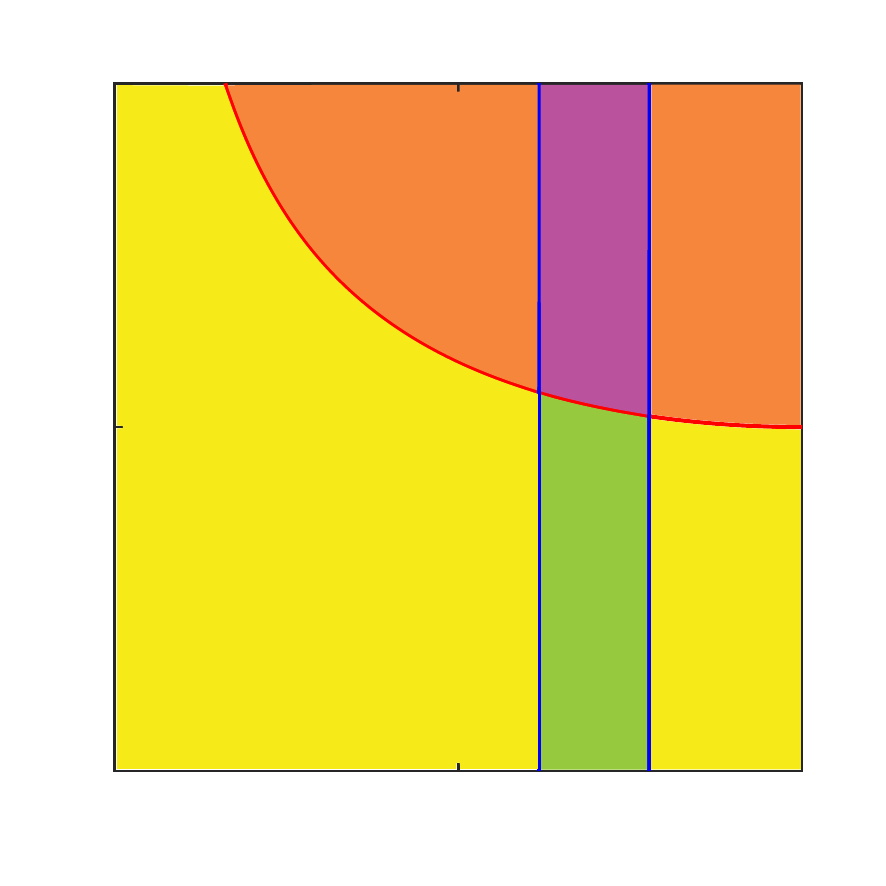}
            \put(12,6){0}
            \put(50,6){$\frac{\pi}{4}$}
            \put(60,6){$\theta_1$}
            \put(73,6){$\theta_2$}
            \put(88,6){$\frac{\pi}{2}$}
            \put(80,0){$\theta$}
            \put(0,60){$\Theta$}
            \put(8,13){0}
            \put(8,50){1}
            \put(8,89){2}
            \put(51,50){$P^+$}
            \put(76,46){$Q^+$}
            \put(59.5,54.3){$\bullet$}
            \put(71.9,52){$\bullet$}
            \put(20,97){\textcolor{myred}{$b=0$}}
            \put(23,91){\textcolor{myred}{\textbf{\textbackslash}}}
            \put(60,97){\textcolor{myblue}{$p=0$}}
            \put(71,91){\textcolor{myblue}{\textbf{\textbackslash}}}
            \put(61,91){\textcolor{myblue}{\textbf{/}}}
            \put(30,40){Slipping}
            \put(40,75){Lift-off}
       \put(68,30){--------------}
            \put(95,30){Inconsistent}
            \put(68,65){--------------}
            \put(95,65){Indeterminate}
        \end{overpic}
        \caption{$2D$ motion, from \cite{Hogan2017}}
        \label{fig:pandb}
    \end{subfigure}
    \hfill
    \begin{subfigure}[t]{0.41\textwidth}
        \begin{overpic}[width=\textwidth]{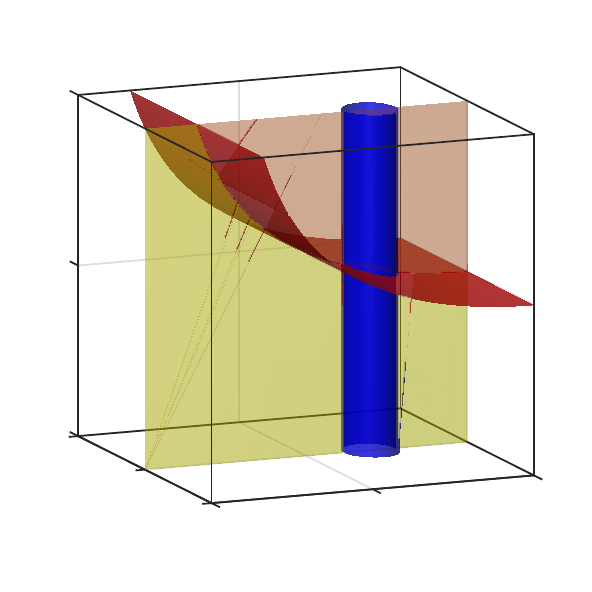}
            \put(38,9){0}
            \put(64,10){$\frac{\pi}{4}$}
            \put(92,13){$\frac{\pi}{2}$}
            \put(75,0){$\theta$}
            \put(0,65){$\Theta$}        
            \put(8,55){1}
            \put(8,84){2}
            \put(10,10){$\varphi$}
            \put(13,17){$-\frac{\pi}{2}$}
            \put(24,12){$-{\pi}$}
            \put(8,25){$0$}
            \put(23,89){\textcolor{red}{$b=0$}}
               \put(29,84){\textcolor{red}{\textbf{/}}}
                \put(55,93){\textcolor{blue}{$p=0$}}
            \put(61,86){\textcolor{blue}{\textbf{/}}}

        \end{overpic}
        \caption{$3D$ motion, for $\Psi=0$}
        \label{fig:p3.2}
    \end{subfigure}
    \caption{Kinematics for $\Theta \equiv \dot{\theta} >0$: (a) $2D$ motion, reproduced from \cite{Hogan2017}: the curves $b=0$ and $p=0$ intersect at $P^+$ and $Q^+$. They separate phase space into four regions, with  different rigid body dynamics: \textit{slipping} ($b<0$, $p>0$, in yellow), \textit{lift-off} ($b>0$, $p>0$, in orange), \textit{inconsistent} ($b<0$, $p<0$, in lime) and \textit{indeterminate} ($b>0$, $p<0$, in purple). Due to the symmetry about $\Theta=0$, there are similarly $P^-$ and $Q^-$ for $\Theta<0$. (b) $3D$ motion with $\Psi\equiv\dot{\psi}=0$: the curves $b=0$ and $p=0$ intersect in the G\'enot-Brogliato (or GB) manifold \cite{Varkonyi18}. A paradox occurs inside the blue cylinder: indeterminate above the surface $b=0$ (in red), inconsistent below. Outside the blue cylinder, we have lift-off above the surface $b=0$ (in red), slipping below. The cross section at $\varphi=-\pi/2$ gives (a). In both figures, we set $\alpha=3$, $\mu=1.4$.}
    \label{fig:p3s} 
\end{figure}
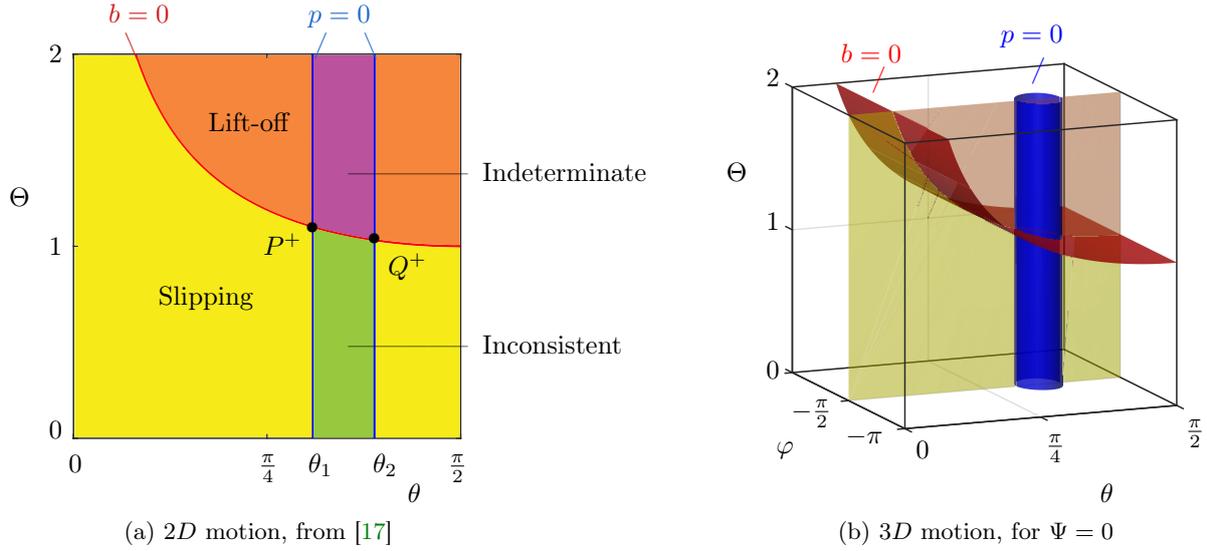
A figure similar to \cref{fig:p3s} can be drawn for $\Theta<0$ where, for the $2D$ problem, the lines $p=0$ and $b=0$ intersect at $P^-$ and $Q^-$ (labelled $P^-_{c1}, P^-_{c2}$ respectively in \cite{Genot1999}) . But since there is no effect on the kinematics, it is not shown here. However as we shall see in \cref{sec:slip}, the geometry of the dynamics in $\Theta>0$ is very different from that in $\Theta<0$.

\subsection{The role of \texorpdfstring{$\Psi$}{Psi}}
\label{subsec:roleofPsi}
 Now we investigate the role of the azimuthal angular velocity $\Psi \equiv \dot{\psi}$. We shall see that when $|\Psi|$ is large enough, lift-off can occur, independently of a paradox (for $p>0$), even when $\Theta=\dot{\theta}=0$. At even larger values of $\Psi$, indeterminacy can occur when $\Theta=0$ and the inconsistent region can even disappear.

In numerical plots for $\alpha=3$ and $\mu=1.4$, between  $\Psi=\pm1.56$ (\cref{fig:p3.3}) and  $\Psi=\pm1.7$ (\cref{fig:p3.4}), we see that the area of lift-off (above $b=0$, outside the cylinder $p=0$) has deformed so much that it is possible for the rod to lift off the rough surface even when $\Theta=0$. The rotation about the vertical $\bf{K}$ axis in the inertial frame (see \cref{fig:me2}) is sufficient to generate an overall upwards acceleration of the rod tip. 
A further increase to $\Psi=\pm 2.2$ (\cref{fig:p3.5}) sees the possibility of indeterminacy (above $b=0$, inside the cylinder $p=0$) when $\Theta=0$. Then between \cref{fig:p3.5} and \cref{fig:p3.6}, the inconsistent region (below $b=0$, inside the cylinder $p=0$) ceases to exist. 
\begin{figure}[htbp]
    \centering
    \newlength{\templength}
    \setlength{\templength}{0.43\textwidth}
    \begin{subfigure}[t]{\templength}
        \begin{overpic}[width=\textwidth]{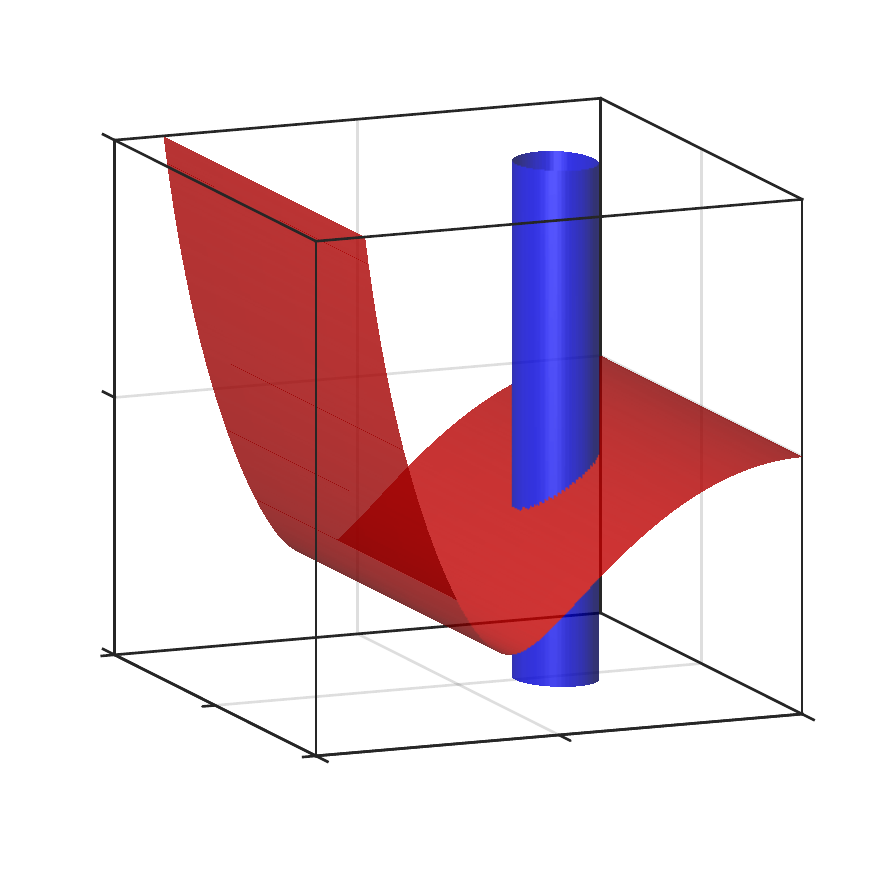}
            \put(38,9){0}
            \put(64,10){$\frac{\pi}{4}$}
            \put(92,13){$\frac{\pi}{2}$}
            \put(75,0){$\theta$}
            \put(0,65){$\Theta$}        
            \put(8,55){1}
            \put(8,84){2}
            \put(10,10){$\varphi$}
            \put(13,17){$-\frac{\pi}{2}$}
            \put(24,12){$-{\pi}$}
            \put(8,25){$0$}
        \end{overpic}
        \caption{$\Psi=\pm1.56$}
        \label{fig:p3.3}
    \end{subfigure}
    \hfil
    \begin{subfigure}[t]{\templength}
        \begin{overpic}[width=\textwidth]{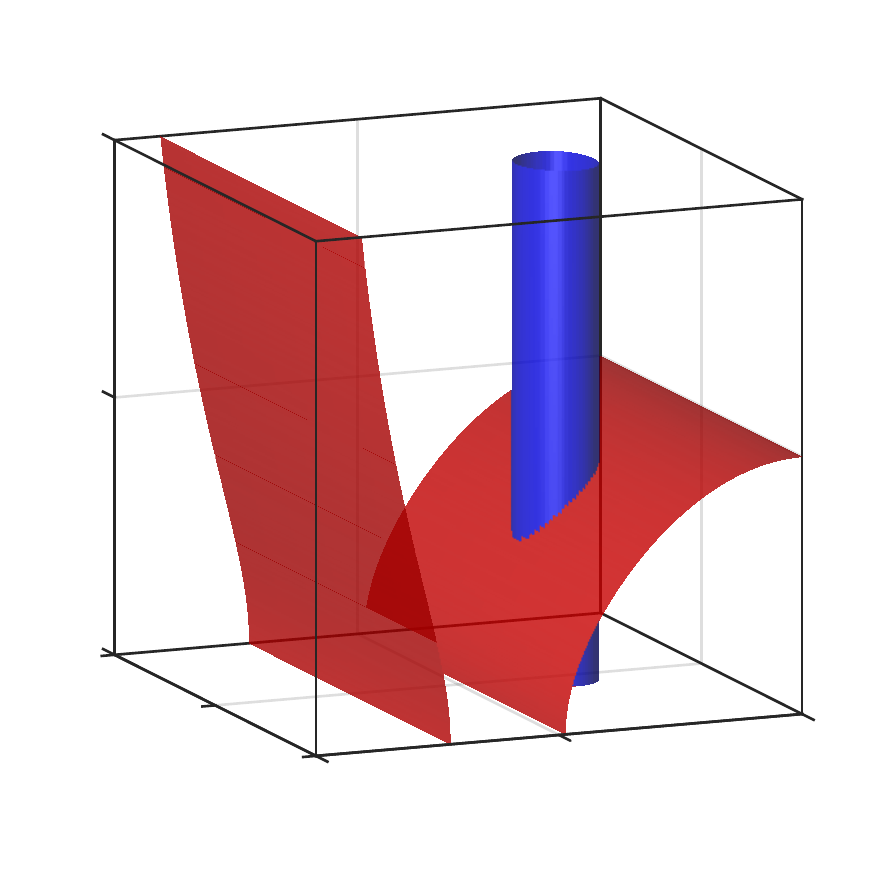}
            \put(38,9){0}
            \put(64,10){$\frac{\pi}{4}$}
            \put(92,13){$\frac{\pi}{2}$}
            \put(75,0){$\theta$}
            \put(0,65){$\Theta$}        
            \put(8,55){1}
            \put(8,84){2}
            \put(10,10){$\varphi$}
            \put(13,17){$-\frac{\pi}{2}$}
            \put(24,12){$-{\pi}$}
            \put(8,25){$0$}
        \end{overpic}
        \caption{$\Psi=\pm1.7$}
        \label{fig:p3.4}
    \end{subfigure}
    \begin{subfigure}[t]{\templength}
        \begin{overpic}[width=\textwidth]{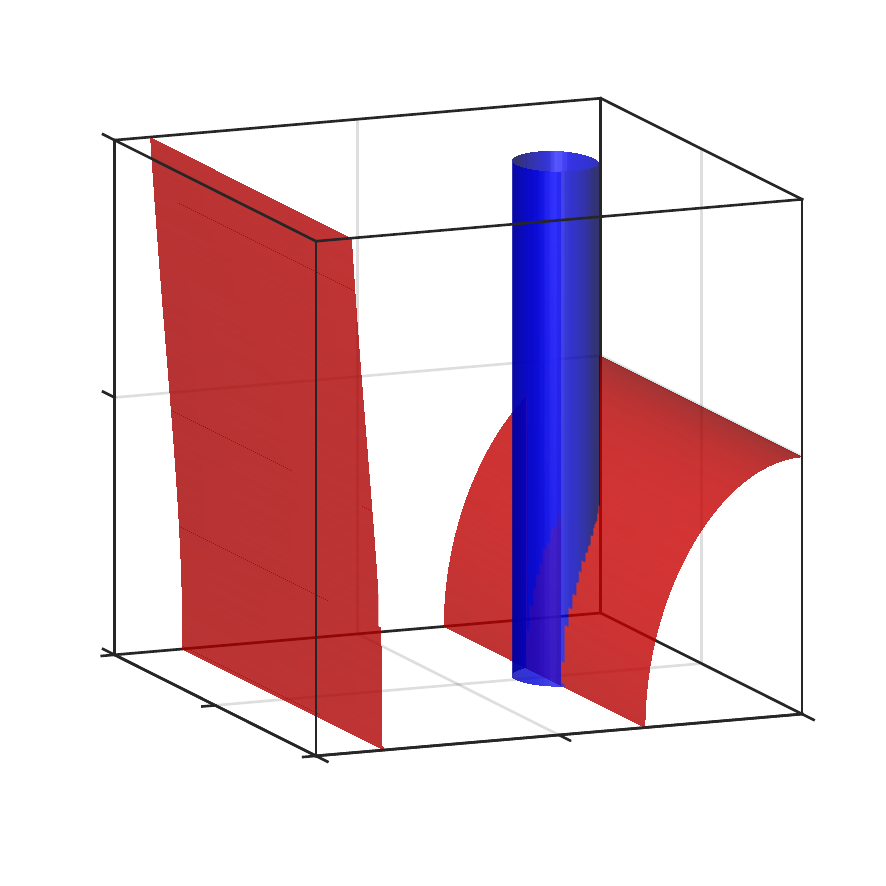}
            \put(38,9){0}
            \put(64,10){$\frac{\pi}{4}$}
            \put(92,13){$\frac{\pi}{2}$}
            \put(75,0){$\theta$}
            \put(0,65){$\Theta$}        
            \put(8,55){1}
            \put(8,84){2}
            \put(10,10){$\varphi$}
            \put(13,17){$-\frac{\pi}{2}$}
            \put(24,12){$-{\pi}$}
            \put(8,25){$0$}
        \end{overpic}
        \caption{$\Psi=\pm2.2$}
        \label{fig:p3.5}
    \end{subfigure}
    \hfil
    \begin{subfigure}[t]{\templength}
        \begin{overpic}[width=\textwidth]{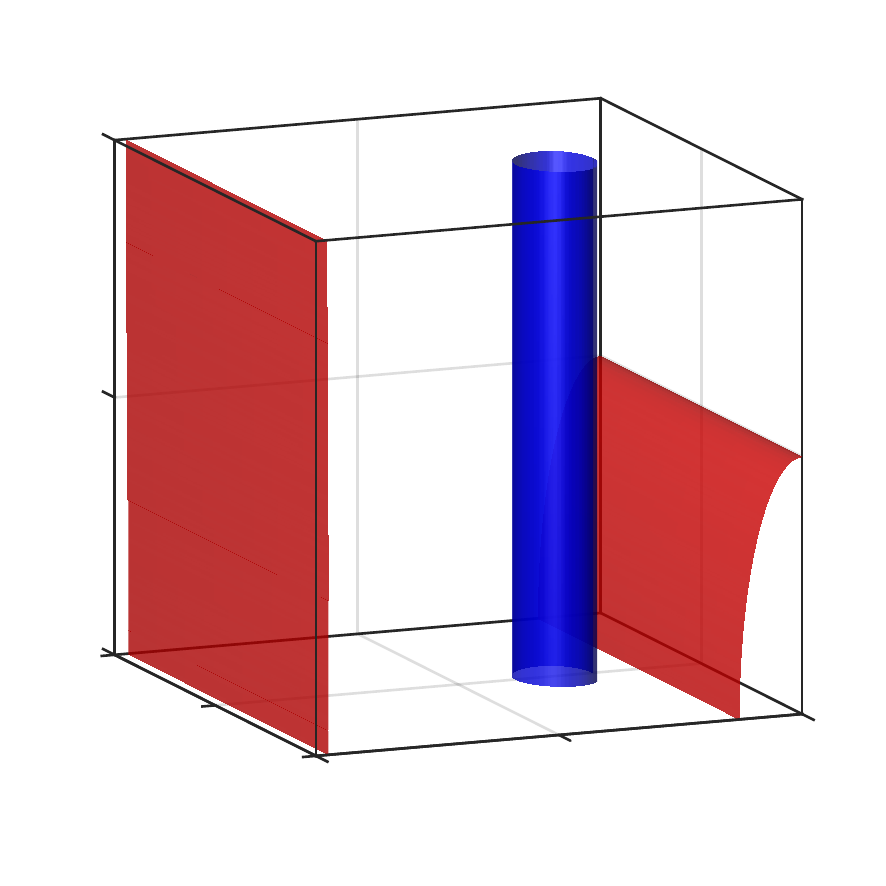}
            \put(38,9){0}
            \put(64,10){$\frac{\pi}{4}$}
            \put(92,13){$\frac{\pi}{2}$}
            \put(75,0){$\theta$}
            \put(0,65){$\Theta$}        
            \put(8,55){1}
            \put(8,84){2}
            \put(10,10){$\varphi$}
            \put(13,17){$-\frac{\pi}{2}$}
            \put(24,12){$-{\pi}$}
            \put(8,25){$0$}
        \end{overpic}
        \caption{$\Psi=\pm5$}
        \label{fig:p3.6}
    \end{subfigure}
    \caption{Surfaces $b=0$ (in red) and $p=0$ (in blue) as $\Psi$ varies. In all figures we set $\alpha=3$, $\mu=1.4$. Only $\Theta\geq0$ shown due to symmetry of $b=0$ and $p=0$ about $\Theta=0$.}
     \label{fig:p3ss} 
\end{figure}

We now find expressions for these critical values of $\Psi$. The smallest value of $\Psi$ where lift-off can be obtained when $\Theta=\dot{\theta}=0$ occurs when $\pdiff{b}{\theta}=\pdiff{b}{\Theta}=0$ on $b=0$. From \cref{eq:b}, it can be shown that this occurs when $(\theta,\Psi)=(\theta_\mathrm{L},\pm\Psi_\mathrm{L})$, where
\begin{equation}\label{eq:thetal}
    \sin \theta_\mathrm{L} := \frac{1}{\sqrt{3}},
\end{equation} 
so that
\begin{equation}\label{eq:thetalvalue}
    \theta_\mathrm{L}:=\arcsin{\left(\frac{1}{\sqrt{3}}\right)} \approx 0.6155 \approx 35.3^{\circ},
\end{equation} 
and
\begin{equation}\label{eq:psil}
    \Psi_\mathrm{L}:=\left(\frac{3\sqrt{3}}{2}\right)^{\frac{1}{2}}\approx 1.6119.
\end{equation}

We show this case as a $2D$ section with $\varphi=-\frac{\pi}{2}$ in \cref{fig:psi1.61192D} and in $3D$ for $\varphi \in (-\pi,0]$ in \cref{fig:psi1.61193D}. Note that this potential for lift-off when $\Theta=0$ for $\Psi > \Psi_\mathrm{L}$ can occur even when there is no paradox.
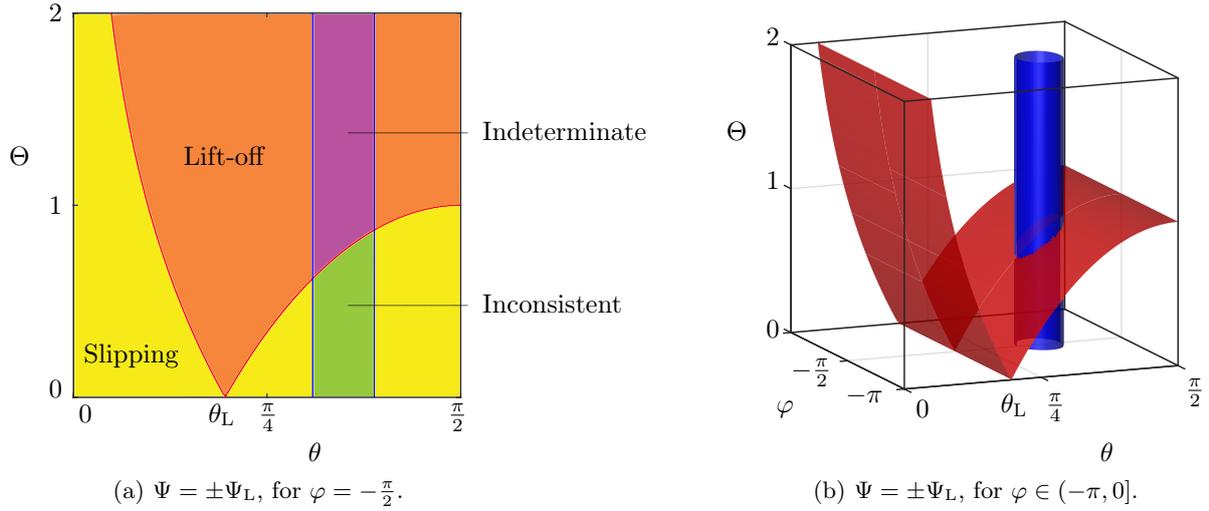
\begin{figure}[htbp]
    \centering
    \begin{subfigure}[t]{0.41\textwidth}
        \begin{overpic}[width=\textwidth]{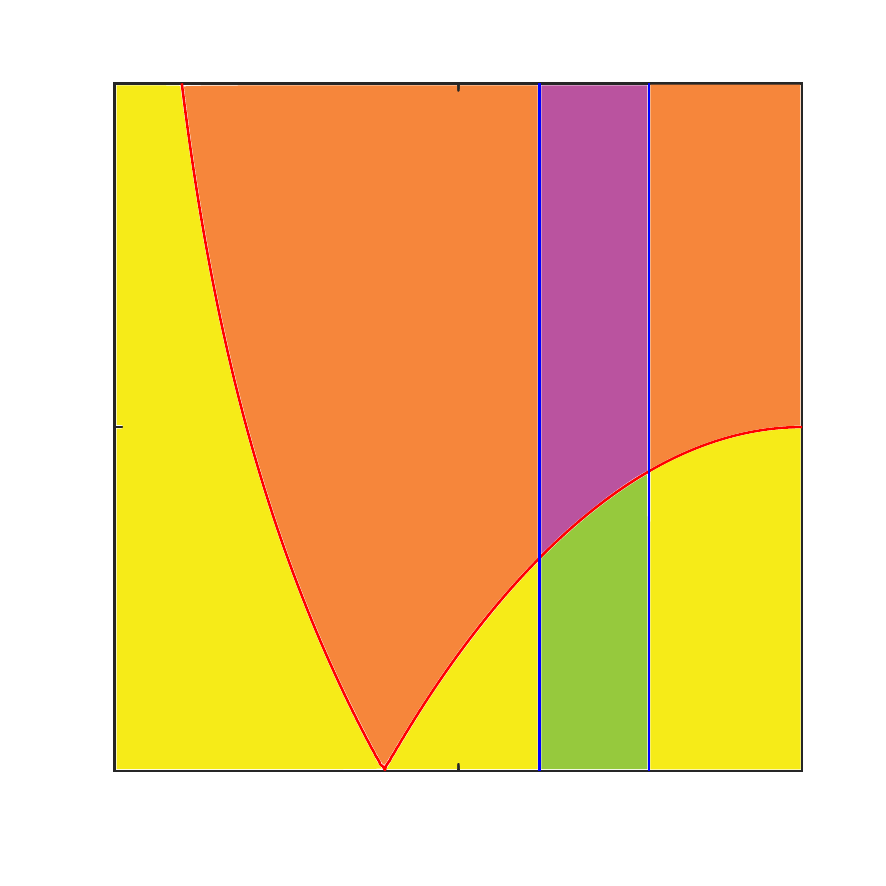}
            \put(14,8){0}
            \put(50,8){$\frac{\pi}{4}$}
            \put(40,8){$\theta_\mathrm{L}$}
            \put(88,8){$\frac{\pi}{2}$}
            \put(60,0){$\theta$}
            \put(0,60){$\Theta$}
            \put(8,13){0}
            \put(8,50){1}
            \put(8,89){2}
            \put(15,20){Slipping}
            \put(35,60){Lift-off}
         \put(68,30){--------------}
            \put(95,30){Inconsistent}
            \put(68,65){--------------}
            \put(95,65){Indeterminate}
        \end{overpic}
        \caption{$\Psi =\pm\Psi_\mathrm{L}$, for $\varphi=-\frac{\pi}{2}$.}
        \label{fig:psi1.61192D}
    \end{subfigure}
    \hfill
    \begin{subfigure}[t]{0.41\textwidth}
        \begin{overpic}[width=\textwidth]{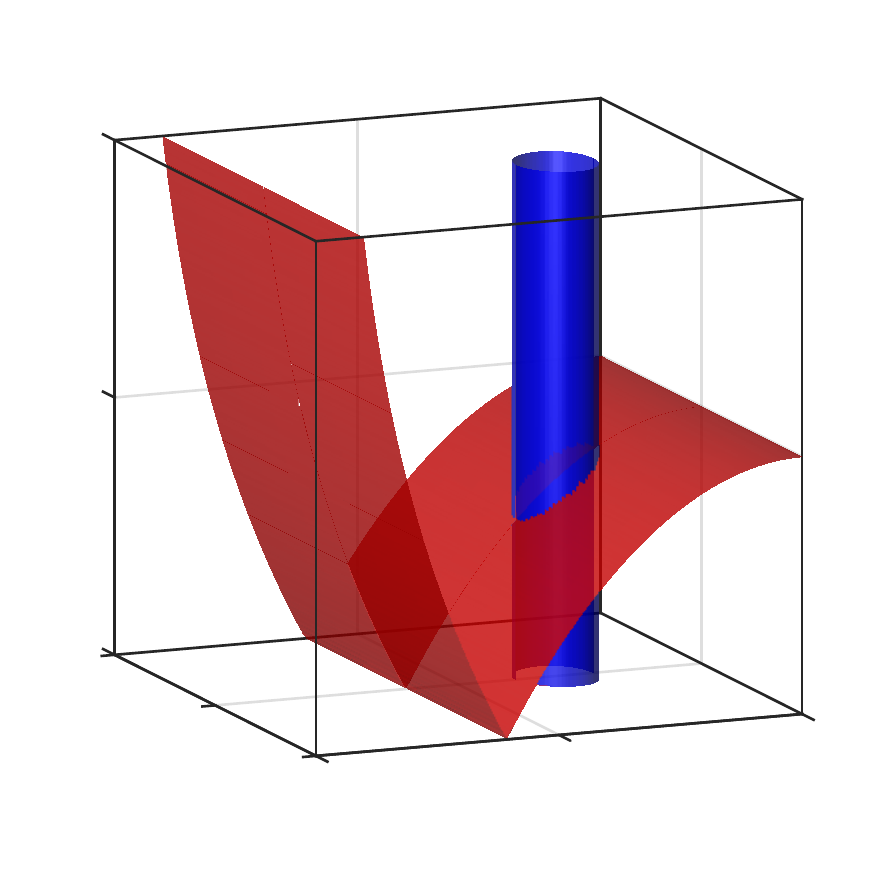}
            \put(38,9){0}
            \put(64,10){$\frac{\pi}{4}$}
            \put(55,10){$\theta_\mathrm{L}$}
            \put(92,13){$\frac{\pi}{2}$}
            \put(75,0){$\theta$}
            \put(0,65){$\Theta$}        
            \put(8,55){1}
            \put(8,84){2}
            \put(10,10){$\varphi$}
            \put(13,17){$-\frac{\pi}{2}$}
            \put(24,12){$-{\pi}$}
            \put(8,25){$0$}
        \end{overpic}
        \caption{$\Psi =\pm\Psi_\mathrm{L}$, for $\varphi \in (-\pi,0]$.}
        \label{fig:psi1.61193D}
    \end{subfigure}
    \caption{Areas of different dynamics for $3D$ motion when $\Psi =\pm\Psi_\mathrm{L}$, from \cref{eq:psil}. Lift-off can occur for $|\Psi| \ge \Psi_\mathrm{L}$ even when $\Theta=\dot{\theta}=0$: (a) $\varphi=-\frac{\pi}{2}$, (b) $\varphi \in (-\pi,0]$. In both figures we set $\alpha=3$, $\mu=1.4$.}
    \label{fig:psi1.6119} 
\end{figure}

The next critical value of $\Psi$ occurs when the surface $b=0$ is tangential to the cylinder $p=0$ at $\theta=\theta_1$, where $\theta_1$ is given in \cref{eq:theta12}. A straightforward calculation shows that this happens when $\Psi=\pm \Psi_1$, where
\begin{equation}\label{eq:psi1}
    \Psi_1(\mu,\alpha) := \frac{(1+\tan^2 \theta_1)^{\frac{3}{4}}}{(\tan \theta_1)^{\frac{1}{2}}}.
\end{equation} 
When $\alpha=3$, $\mu=1.4$, this corresponds to $\Psi_1 = 1.9480$ with $\theta_1=0.9702$. We show this case in \cref{fig:psi1.9483}; as a $2D$ section with $\varphi=-\frac{\pi}{2}$ in \cref{fig:psi1.9482D} and for $\varphi \in (-\pi,0]$ in \cref{fig:psi1.9483D}. When $\theta_\mathrm{L}<\theta_1$, this is the smallest value of $\Psi$ for which indeterminacy can occur for $\Theta=0$.
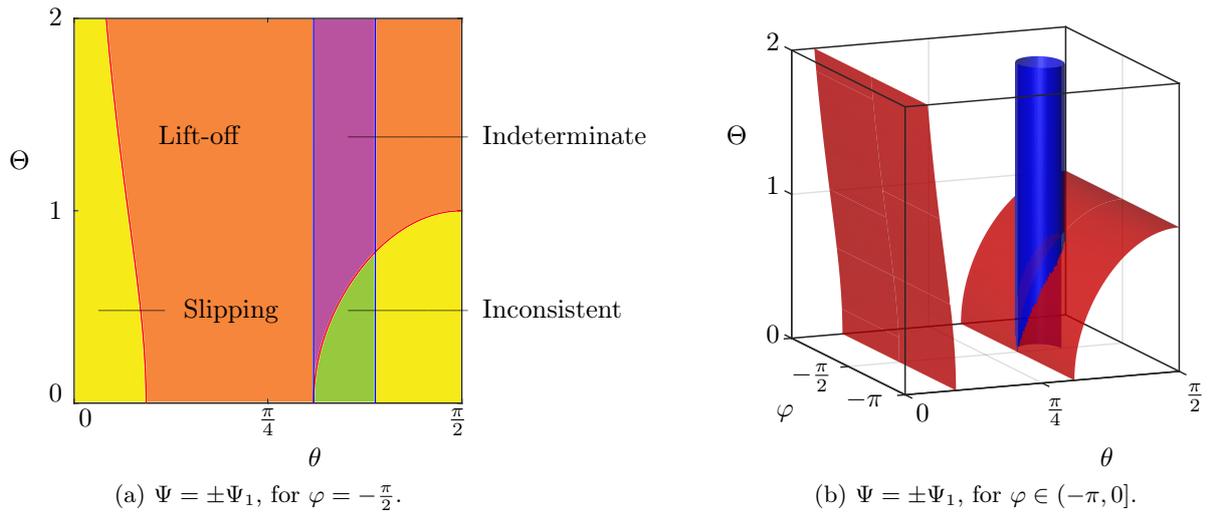
\begin{figure}[htbp]
    \centering
    \begin{subfigure}[t]{0.41\textwidth}
        \begin{overpic}[width=\textwidth]{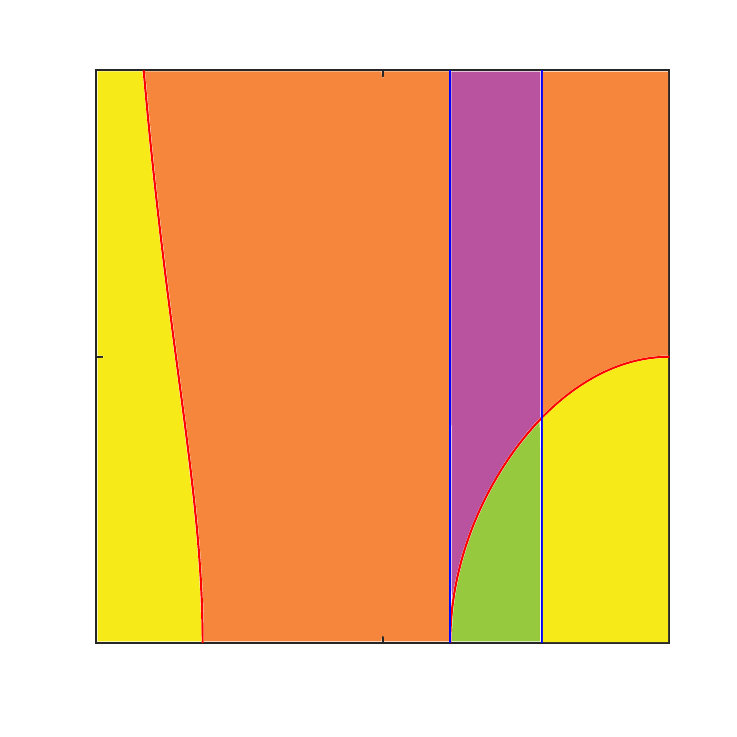}
            \put(14,8){0}
            \put(50,8){$\frac{\pi}{4}$}
            \put(88,8){$\frac{\pi}{2}$}
            \put(60,0){$\theta$}
            \put(0,60){$\Theta$}
            \put(8,13){0}
            \put(8,50){1}
            \put(8,89){2}
            \put(35,30){Slipping}
             \put(18,30){--------}
            \put(30,65){Lift-off}
         \put(68,30){--------------}
            \put(95,30){Inconsistent}
            \put(68,65){--------------}
            \put(95,65){Indeterminate}
        \end{overpic}
        \caption{$\Psi =\pm\Psi_1$, for $\varphi=-\frac{\pi}{2}$.}
        \label{fig:psi1.9482D}
    \end{subfigure}
    \hfill
    \begin{subfigure}[t]{0.41\textwidth}
        \begin{overpic}[width=\textwidth]{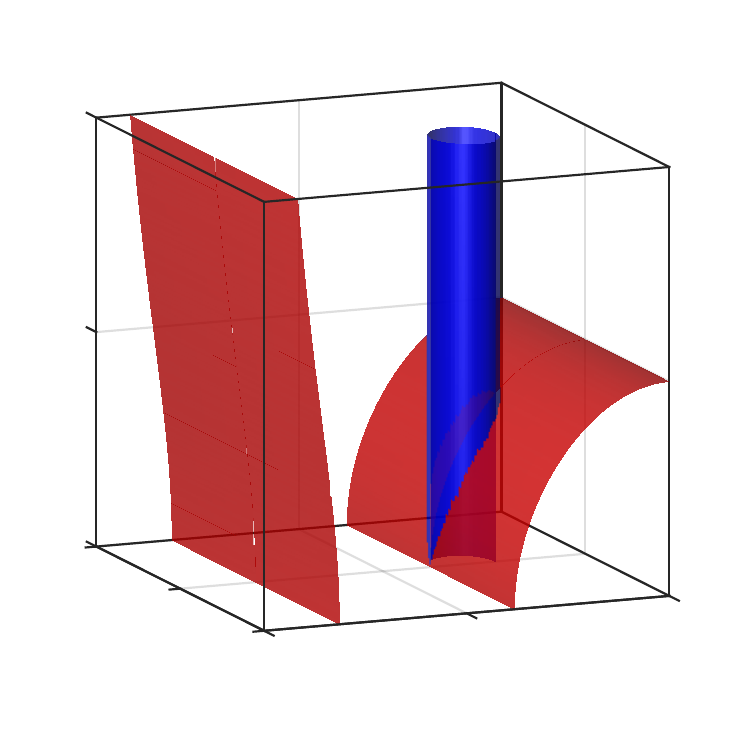}
            \put(38,9){0}
            \put(64,10){$\frac{\pi}{4}$}
            \put(92,13){$\frac{\pi}{2}$}
            \put(75,0){$\theta$}
            \put(0,65){$\Theta$}        
            \put(8,55){1}
            \put(8,84){2}
            \put(10,10){$\varphi$}
            \put(13,17){$-\frac{\pi}{2}$}
            \put(24,12){$-{\pi}$}
            \put(8,25){$0$}
        \end{overpic}
        \caption{$\Psi =\pm\Psi_1$, for $\varphi \in (-\pi,0]$.}
        \label{fig:psi1.9483D}
    \end{subfigure}
    \caption{Areas of different dynamics for $3D$ motion when $\Psi =\pm\Psi_1$, from \cref{eq:psi1}. The surface $b=0$ is tangential to the cylinder $p=0$ at $\theta=\theta_1$: (a) $\varphi=-\frac{\pi}{2}$, (b) $\varphi \in (-\pi,0])$. Both figures use $\alpha=3$, $\mu=1.4$.}
    \label{fig:psi1.9483} 
\end{figure}

The final critical value of $\Psi$ occurs when the surface $b=0$ is tangential to the cylinder $p=0$ at $\theta=\theta_2$, where $\theta_2$ is given in \cref{eq:theta12}. This happens when $\Psi=\pm\Psi_2$, where
\begin{equation}\label{eq:psi2}
    \Psi_2(\mu,\alpha) := \frac{(1+\tan^2 \theta_2)^{\frac{3}{4}}}{\left(\tan \theta_2\right)^{\frac{1}{2}}}.
\end{equation} 
When $\alpha=3$, $\mu=1.4$, this corresponds to $\Psi_2 = 3.0097$ with $\theta_2=1.2209$. We show this case in \cref{fig:psi3.0097}; as a $2D$ section with $\varphi=-\frac{\pi}{2}$ in \cref{fig:psi3.00972D} and for $\varphi \in (-\pi,0]$ in \cref{fig:psi3.00973D}. When $|\Psi| \ge \Psi_2$, there is no inconsistent paradox.
\begin{figure}[htbp]
    \centering
    \begin{subfigure}[t]{0.41\textwidth}
        \begin{overpic}[width=\textwidth]{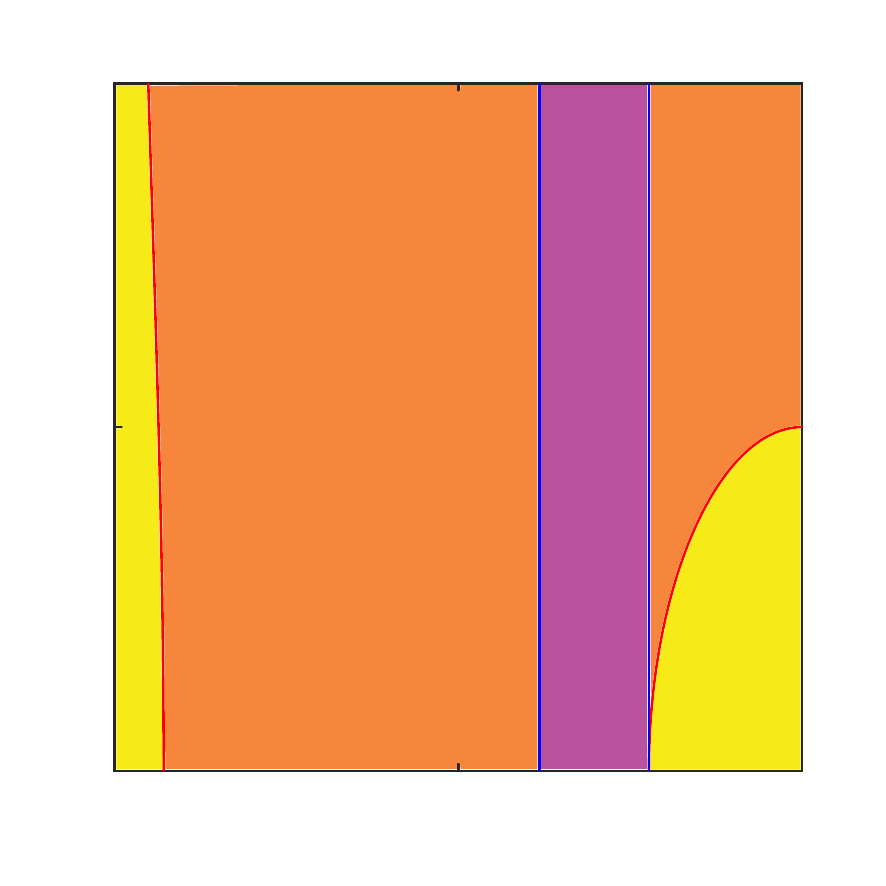}
            \put(14,8){0}
            \put(50,8){$\frac{\pi}{4}$}
            \put(88,8){$\frac{\pi}{2}$}
            \put(60,0){$\theta$}
            \put(0,60){$\Theta$}
            \put(8,13){0}
            \put(8,50){1}
            \put(8,89){2}
                   \put(32,30){Slipping}
             \put(15,30){--------}
            \put(30,65){Lift-off}
            \put(68,65){--------------}
            \put(95,65){Indeterminate}
        \end{overpic}
        \caption{$\Psi =\pm\Psi_2$, for $\varphi=-\frac{\pi}{2}$.}
        \label{fig:psi3.00972D}
    \end{subfigure}
    \hfill
    \begin{subfigure}[t]{0.41\textwidth}
        \begin{overpic}[width=\textwidth]{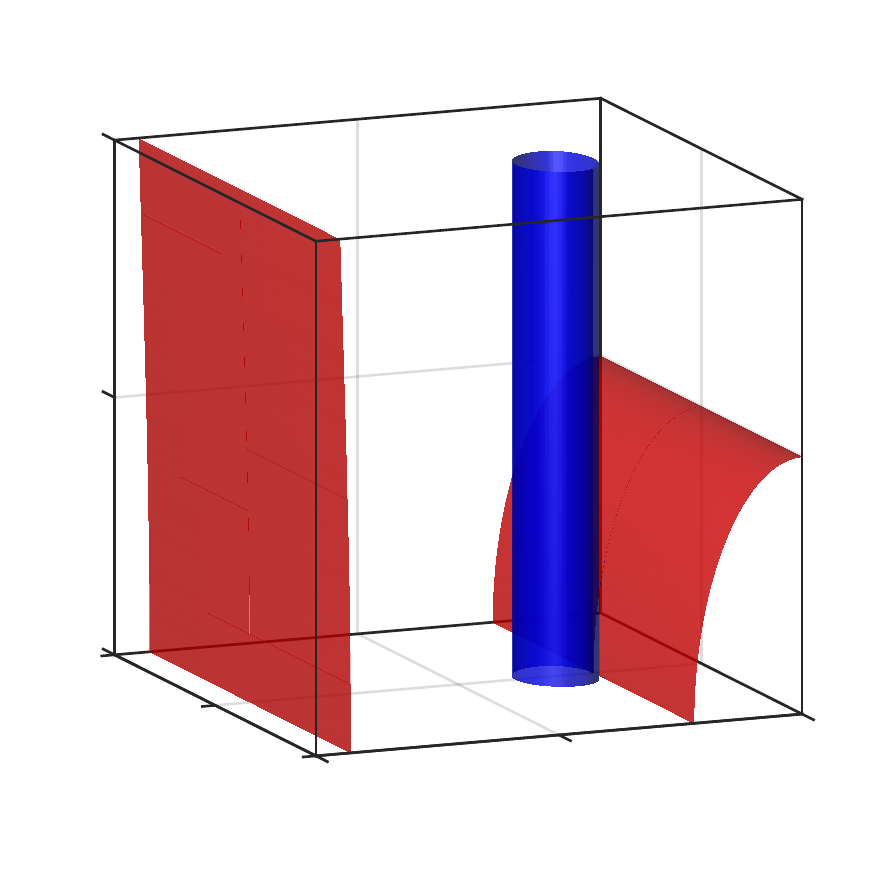}
            \put(38,9){0}
            \put(64,10){$\frac{\pi}{4}$}
            \put(92,13){$\frac{\pi}{2}$}
            \put(75,0){$\theta$}
            \put(0,65){$\Theta$}        
            \put(8,55){1}
            \put(8,84){2}
            \put(10,10){$\varphi$}
            \put(13,17){$-\frac{\pi}{2}$}
            \put(24,12){$-{\pi}$}
            \put(8,25){$0$}
        \end{overpic}
        \caption{$\Psi =\pm\Psi_2$, for $\varphi \in (-\pi,0]$.}
        \label{fig:psi3.00973D}
    \end{subfigure}
    \caption{Areas of different dynamics for $3D$ motion when $\Psi =\pm\Psi_2$, from \cref{eq:psi2}. The inconsistent paradox disappears for $|\Psi| \ge \Psi_2$: (a) $\varphi=-\frac{\pi}{2}$, (b) $\varphi \in (-\pi,0]$. Both figures use $\alpha=3$, $\mu=1.4$.}
    \label{fig:psi3.0097} 
\end{figure}

\subsection{Generic behaviour}
\label{subsec:generic}
As we have seen, variation in $\Psi$ causes changes to the shape of the surface $b=0$ projected into $(\theta,\varphi,\theta)$ and hence to the topology of the GB manifold $b=p=0$ in this view.
In fact there are three different \textit{mechanisms} I-III by which the GB manifold deforms, for fixed $\mu$, as $\Psi$ increases. They are illustrated in \cref{fig:mechanisms}, as follows:
\begin{itemize}
\item[I] ($\mu<\mu_\mathrm{P}$): there is no paradox and the rod can either slip or lift-off. 
    \begin{itemize}
        \item For  $|\Psi|<\Psi_\mathrm{L}$, we have case \circled{1}.
        \item For  $|\Psi|>\Psi_\mathrm{L}$, we have case \circled{2}.
    \end{itemize}
\item[II] ($\mu_\mathrm{P}<\mu<\mu_\mathrm{L}$): $p<0$ for $\theta\in(\theta_-,\theta_+)$ and so a Painlev\'{e} paradox can occur.
    \begin{itemize}
        \item For $\Psi<|\Psi_\mathrm{L}|$, we have case \circled{3} and all four modes (slipping, lift-off, inconsistent and indeterminate) are possible.
        \item For $\Psi_\mathrm{L}<|\Psi|<\Psi_1$, lift-off is possible for $\Theta=0$; case \circled{4}.
        \item For $\Psi_1<|\Psi|<\Psi_2$, both lift-off and indeterminate modes as possible for $\Theta=0$; case \circled{6}.
        \item For $|\Psi|>\Psi_2$, the inconsistent mode ceases to exist; case \circled{7}.
    \end{itemize}
\item[III] ($\mu>\mu_\mathrm{L}$): $p<0$ for $\theta\in(\theta_-,\theta_+)$ and so a Painlev\'{e} paradox can occur. Mechanism III is the same as mechanism II except when $\Psi_\mathrm{L}<|\Psi|<\Psi_1$, where the indeterminate mode (rather than lift-off) is possible for $\Theta=0$; case \circled{5}.
\end{itemize}

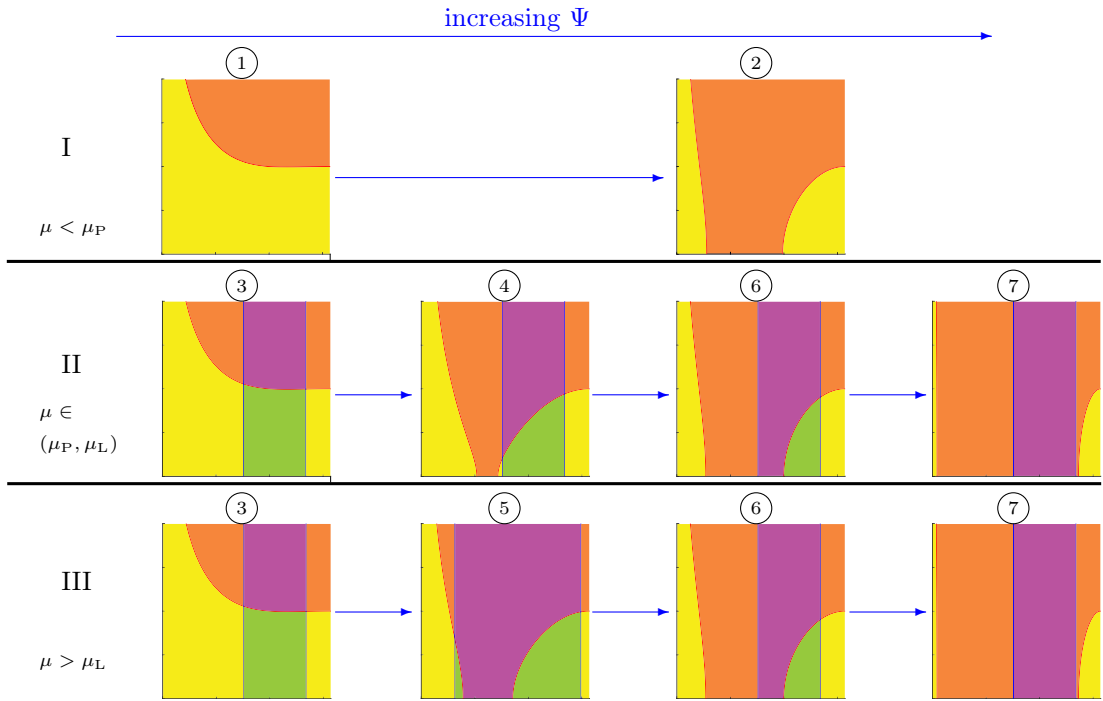
\begin{figure}[htbp]
    \centering
    \vspace{8pt}
    \begin{overpic}[width=0.9\textwidth]{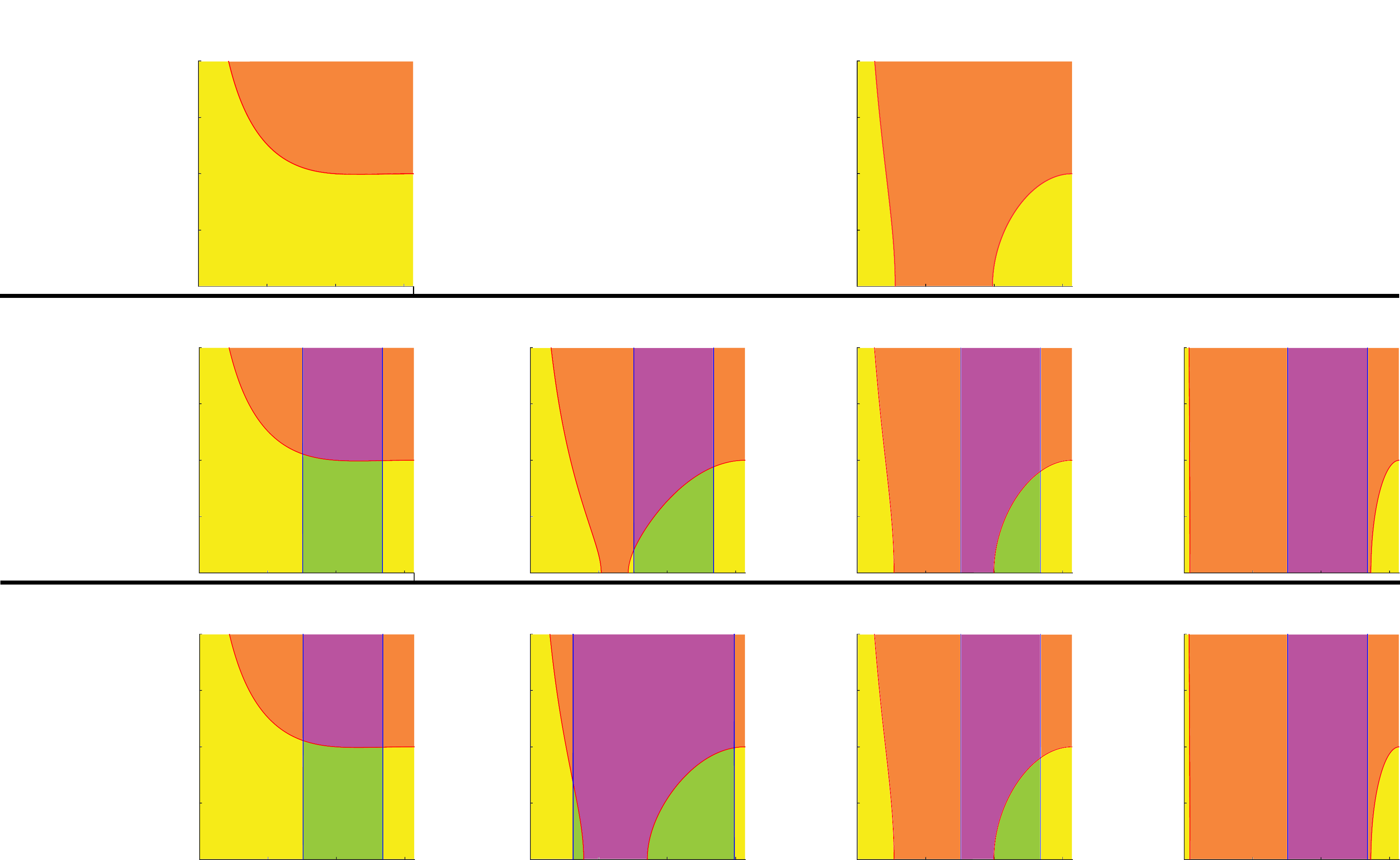}
      \put(5,10){III}
    \put(3,3){\scriptsize{$\mu>\mu_\mathrm{L}$}}
    \put(5,30){II}
    \put(3,26){\scriptsize{$\mu\in$}}
    \put(3,23){\scriptsize{$(\mu_\mathrm{P},\mu_\mathrm{L})$}}
    \put(5,50){I}
    \put(3,43){\scriptsize{$\mu<\mu_\mathrm{P}$}}
    {\color{blue}
    \put(40,62){increasing $\Psi$}\put(10,61){\vector(1,0){80}}
    \put(30,8){\vector(1,0){7}}
    \put(53.5,8){\vector(1,0){7}}
    \put(77,8){\vector(1,0){7}}
    \put(30,28){\vector(1,0){7}}
    \put(53.5,28){\vector(1,0){7}}
    \put(77,28){\vector(1,0){7}}
    \put(30,48){\vector(1,0){30}}
    }
    \put(20,17){\scriptsize{\circled{3}}}
    \put(20,37.5){\scriptsize{\circled{3}}}
    \put(20,58){\scriptsize{\circled{1}}}
    \put(44,17){\scriptsize{\circled{5}}}
    \put(44,37.5){\scriptsize{\circled{4}}}
    \put(67,17){\scriptsize{\circled{6}}}
    \put(67,37.5){\scriptsize{\circled{6}}}
    \put(67,58){\scriptsize{\circled{2}}}
    \put(90.5,17){\scriptsize{\circled{7}}}
    \put(90.5,37.5){\scriptsize{\circled{7}}}
    \end{overpic}
    \caption{Mechanisms I--III for the $3D$ Painlev\'e problem, for increasing $\Psi$.  }
    \label{fig:mechanisms}
\end{figure}

The different types of behaviour labelled \circled{1}--\circled{7} occur for different values of $\Psi$ as a function of $\mu$, as shown in \cref{fig:muPsicases}. Mechanisms I--III occur at different values of $\mu$ as a function of $\alpha$, as shown in \cref{fig:alphamumech}.

In \cref{fig:muPsicases}, we show $\Psi_\mathrm{L}$, $\Psi_1$ and $\Psi_2$ as functions of $\mu$ for $\alpha=3$ when $\Psi>0$. 
At $\mu=\mu_\mathrm{P}(\alpha)$, we have
\begin{equation}\label{eq:psip}
    \Psi_1(\mu_\mathrm{P}(\alpha),\alpha)=\Psi_2(\mu_\mathrm{P}(\alpha),\alpha)=\Psi_\mathrm{P}(\alpha):=\sqrt[4]{\frac{(2+\alpha)^3}{(1+\alpha)}}.
\end{equation} 

So $\Psi_\mathrm{P} \approx 2.3644$ when $\alpha=3$. Note that $\Psi_1 \ge \Psi_\mathrm{L}$, with equality occurring when  $\theta_1=\theta_\mathrm{L} \approx 0.6155$ from \cref{eq:thetalvalue}, when 
\begin{equation}\label{eq:mul}
    \mu=\mu_\mathrm{L}(\alpha):=\frac{2\alpha+3}{\alpha \sqrt{2}}.
\end{equation}
For a uniform slender rod with $\alpha=3$, $\mu_\mathrm{L}(3)=\frac{3}{\sqrt{2}} \approx 2.1213$. 

In \cref{fig:alphamumech}, we show $\mu_\mathrm{P}(\alpha)$ from \cref{eq:mup2D} and $\mu_\mathrm{L}(\alpha)$ from \cref{eq:mul}.
\begin{figure}[htbp]
    \centering
    \begin{subfigure}[b]{0.47\textwidth}
    \begin{overpic}[width=\textwidth]{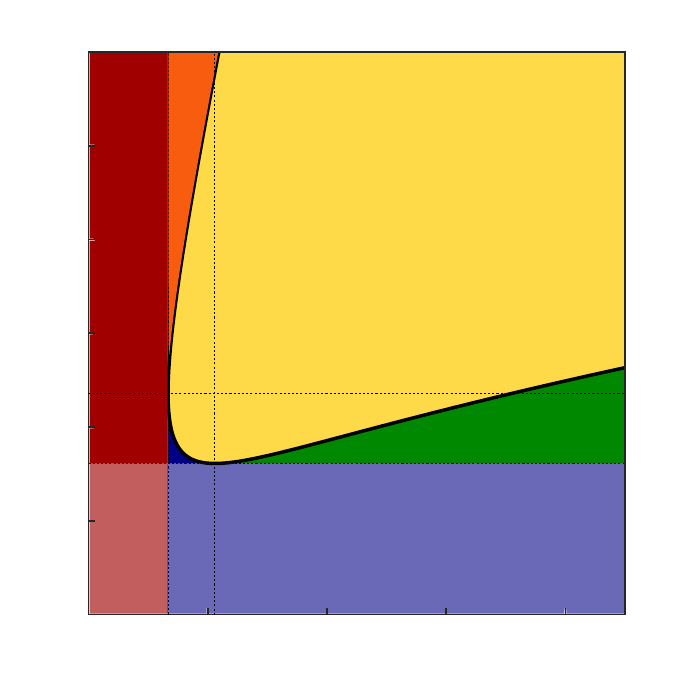}
        {
        \put(14,20){\circled{1}}
        \put(14,60){\circled{2}}
        \put(24,86){\circled{7}}
        \put(60,70){\circled{6}}
        \put(75,36){\circled{5}}
        \put(55,20){\circled{3}}
        \put(25.5,33.5){\line(1,2){4}}
        \put(27.5,43.5){\circled{4}}}
        \put(12,6){0}
        \put(22,6){$\mu_\mathrm{P}$}
        \put(30,6){$\mu_\mathrm{L}$}
        \put(8,10){$0$}
        \put(6,32){$\Psi_\mathrm{L}$}
        \put(6,42){$\Psi_\mathrm{P}$}
        \put(35,85){$\Psi=\Psi_2$}
        \put(60,50){$\Psi=\Psi_1$}
        \put(0,55){$\Psi$}
        \put(80,6){8}
        \put(46,6){4}
        \put(7,64){4}
        \put(70,3){$\mu$}
        \end{overpic}
        \caption{}
        \label{fig:muPsicases}
    \end{subfigure}
    \hfil
     \begin{subfigure}[b]{0.47\textwidth}
    \begin{overpic}[width=0.95\textwidth]{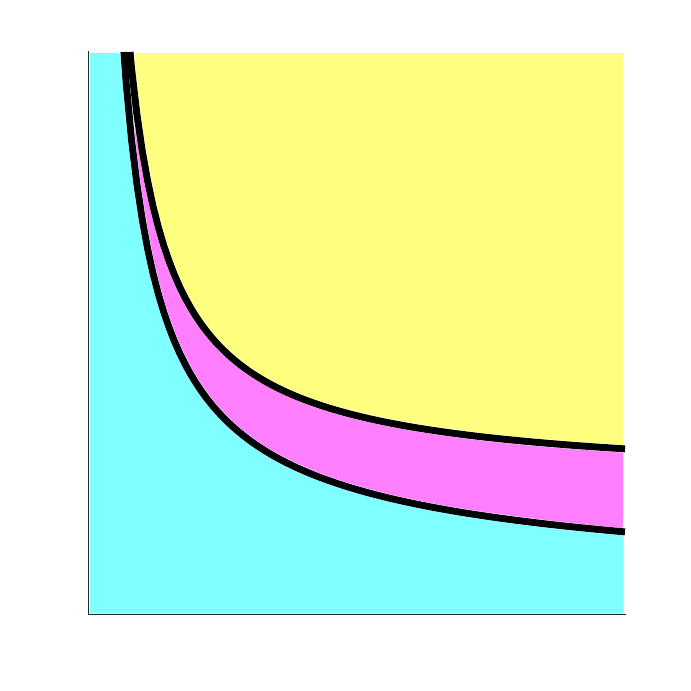}
    \put(50,15){$\mu=\mu_\mathrm{P}(\alpha)$}
    \put(49,17){\vector(-1,2){7.5}}
    \put(60,45){$\mu=\mu_\mathrm{L}(\alpha)$}
    \put(58,45){\vector(-3,-2){6}}
    \put(20,20){I}
    \put(65,29){II}
    \put(70,70){III}
    \put(0,50){$\mu$}
    \put(70,3){$\alpha$}
    \end{overpic}
    \caption{}
    \label{fig:alphamumech}
    \end{subfigure}
    \caption{(a) Regions where the seven types  of kinematic behaviour occur in the $3D$ Painlev\'e problem for $\alpha=3$. (b) Regions where mechanisms I--III occur in the $3D$ Painlev\'e problem. Note that $\mu_\mathrm{P}(\alpha)$ is given in \cref{eq:mup2D}, $\mu_\mathrm{L}(\alpha)$ in \cref{eq:mul}, $\Psi_\mathrm{L}$ in \cref{eq:psil}, $\Psi_\mathrm{P}$ in \cref{eq:psip}, $\Psi_1$ in \cref{eq:psi1}, $\Psi_2$ in \cref{eq:psi2}.}
    \label{fig:CaseIs}
\end{figure}
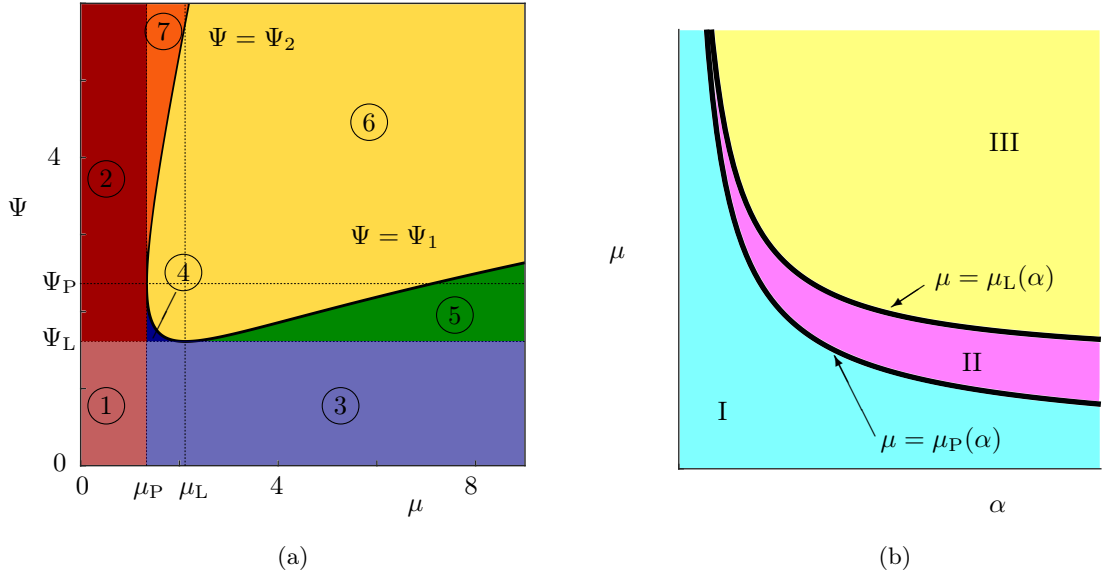
In \cref{fig:Psistheta}, we show the critical polar angles $\theta_\mathrm{L}$, $\theta_1$, $\theta_2$ as functions of $\mu$ for $\alpha=3$. At $\mu=\mu_\mathrm{P}(\alpha)$, $\theta_1=\theta_2=\theta_\mathrm{P}(\alpha)$, see \cref{eq:thetap}. It is straightforward to show that $\theta_2>\theta_\mathrm{L} \, \forall \alpha, \mu$. 
\begin{figure}[htbp]
    \centering
          \begin{overpic}[width=0.47\textwidth]{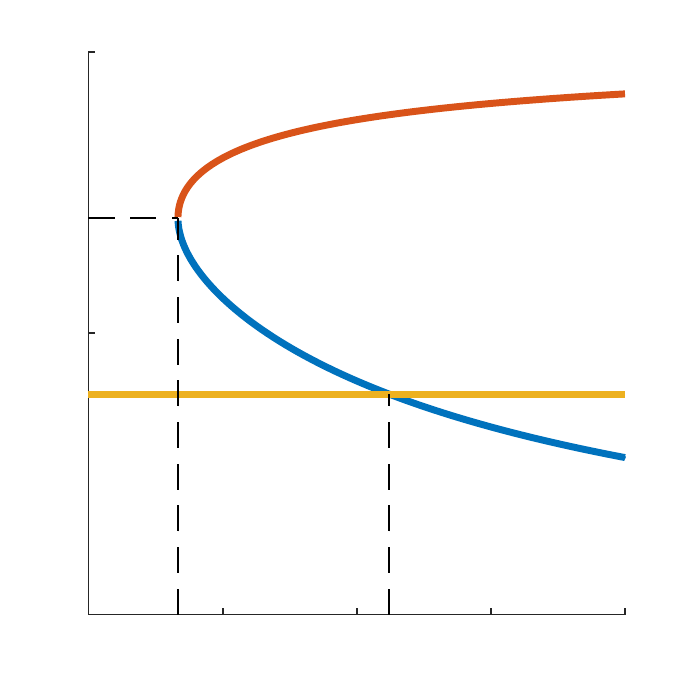}
        \put(-5,50){$\theta$}
        \put(8,50){$\frac{\pi}{4}$}
        \put(8,90){$\frac{\pi}{2}$}
        \put(9,10){0}
        \put(12,5){1}
        \put(5,42){$\theta_\mathrm{L}$}
        \put(30,85){$\theta=\theta_2$}
        \put(70,29){$\theta=\theta_1$}
        \put(5,70){$\theta_\mathrm{P}$}
        \put(22,6){$\mu_\mathrm{P}$}
        \put(55,6){$\mu_\mathrm{L}$}
        \put(47,0){$\mu$}
        \put(89,6){3}
        \end{overpic}
    \caption{Critical values $\theta_\mathrm{L}$, $\theta_1$ and $\theta_2$ as functions of $\mu$ for $\alpha=3$. Here $\theta_\mathrm{L} =\arcsin{\frac{1}{\sqrt{3}}} \approx 0.6155$ from \cref{eq:thetal} and $\theta_\mathrm{P} \approx 1.1071$ from \cref{eq:thetap}. $\theta_1, \theta_2$ are given in \cref{eq:theta12} and $\mu_\mathrm{P} = \frac{4}{3}$, $\mu_\mathrm{L}=\frac{3}{\sqrt{2}}$.}
    \label{fig:Psistheta} 
\end{figure}
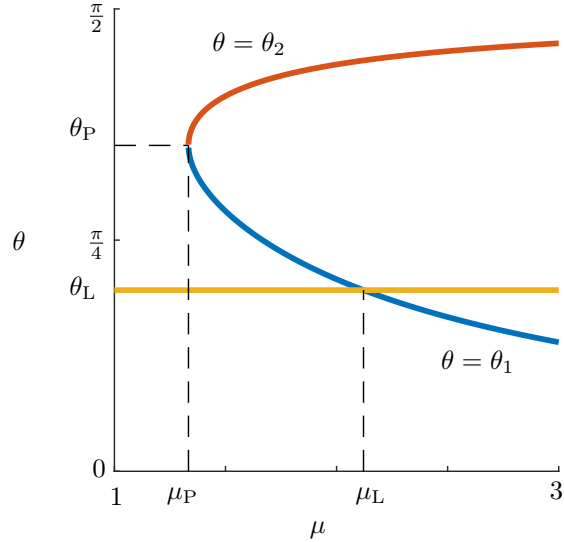

In \cref{fig:mechanisms}, we show only the generic cases of the geometry of the surfaces $p=0$ and $b=0$ for fixed $\Psi$. In \cref{fig:mechbifs} in \cref{sec:geobifs}, we show all the bifurcations between these cases with changes in the variable $\Psi$ and the parameter $\mu$.

\subsection{GB manifolds}
\label{subsec:GB}
The GB manifold is the intersection of the codimension-1 sets $b=0$ and $p=0$, and hence is codimension-2.
For the $2D$ problem ($\varphi=\pm\pi/2$, $\Psi=0$) this intersection corresponds to 4 points ($P^\pm$, $Q^\pm$) when projected into the plane $(\theta,\Theta)$ if $\mu>\mu_P$, see \cref{fig:p3s} and \cite[fig. 2]{Genot1999}.
For the $3D$ problem, this codimenion-2 set is more difficult to visualise. Here, we project this set into $(\theta,\varphi,\Theta)$ for fixed $\Psi$. For a given $\Psi$, this set can take a number of different forms in  $(\theta,\varphi,\Theta)$, or it may not exist.
In \cref{fig:GBmanifoldmu1p4} $\mu = 1.4 \in (\mu_\mathrm{P},\mu_\mathrm{L})$ corresponds to mechanism II. Since $\Psi_\mathrm{L}=1.6118$, $\Psi_1=1.9480$ and $\Psi_2=3.0097$ then $\Psi=0$ ({\color{md2}{$\blacksquare$}}) corresponds to type \circled{3} behaviour, $\Psi=1.7800$ ({\color{md3}{$\blacksquare$}}) to type \circled{4} and $\Psi=2.4789$ ({\color{md5}{$\blacksquare$}}) to type \circled{6}. In \cref{fig:GBmanifoldmu6} $\mu = 6 > \mu_\mathrm{L}$ corresponds to mechanism III. Since $\Psi_\mathrm{L}=1.6118$, $\Psi_1=2.1876$ and $\Psi_2=17.8171$ then $\Psi=0$ ({\color{md2}{$\blacksquare$}}) corresponds to type \circled{3} behaviour, $\Psi=1.8997$ ({\color{md3}{$\blacksquare$}}) to type \circled{5} and $\Psi=10.0024$ ({\color{md5}{$\blacksquare$}}) to type \circled{6}. 

\begin{figure}[htbp]
\centering
\begin{subfigure}{0.45\textwidth}
\begin{overpic}[width=\textwidth]{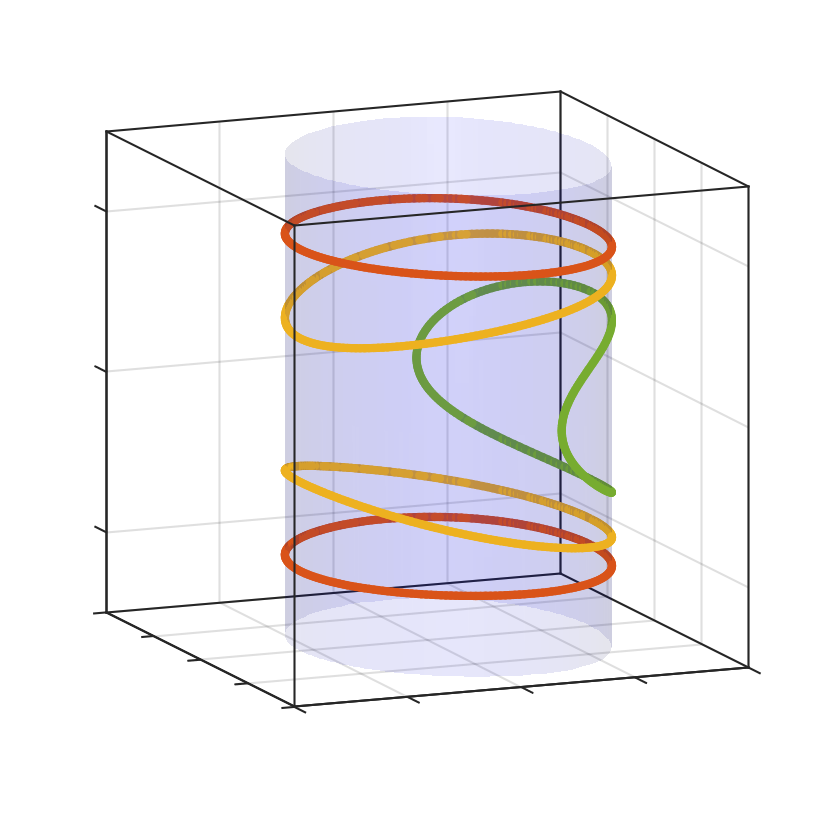}
\put(8,55){0}
\put(8,75){1}
\put(6,36){-1}
\put(0,60){$\Theta$}
\put(1,23){$-\frac{3\pi}{8}$}
\put(13,18){$-\frac{\pi}{2}$}
\put(23,11){$-\frac{5\pi}{8}$}
\put(5,10){$\varphi$}
\put(76,12.5){$\frac{3\pi}{8}$}
\put(50,9){$\frac{5\pi}{16}$}
\put(80,2){$\theta$}
\end{overpic}
\caption{$\mu=1.4$}
\label{fig:GBmanifoldmu1p4}
\end{subfigure}
\hfil
\begin{subfigure}{0.45\textwidth}
\begin{overpic}[width=\textwidth]{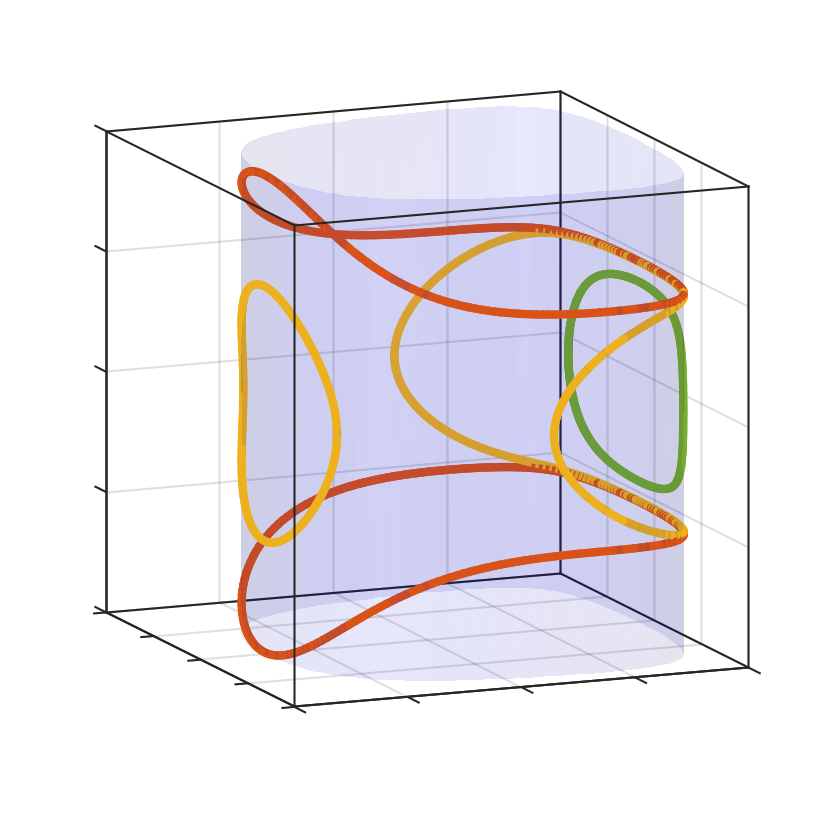}
\put(8,55){0}
\put(8,85){2}
\put(6,27){-2}
\put(0,60){$\Theta$}
\put(9,22){$0$}
\put(13,18){$-\frac{\pi}{2}$}
\put(26,12){$-\pi$}
\put(5,10){$\varphi$}
\put(36,9){$0$}
\put(64,11){$\frac{\pi}{4}$}
\put(91,13){$\frac{\pi}{2}$}
\put(80,2){$\theta$}
\end{overpic}
\caption{$\mu=6$}
\label{fig:GBmanifoldmu6}
\end{subfigure}
\caption{GB manifolds for $\alpha=3$. (a) mechanism II: $\mu=1.4$ and $\Psi=0<\Psi_\mathrm{L}$ ({\color{md2}{$\blacksquare$}}), $1.7800\in(\Psi_\mathrm{L},\Psi_1)$ ({\color{md3}{$\blacksquare$}}), $2.4789\in(\Psi_1,\Psi_2)$ ({\color{md5}{$\blacksquare$}}). (b) mechanism III: $\mu=6$ and $\Psi=0<\Psi_\mathrm{L}$ ({\color{md2}{$\blacksquare$}}), $1.8997\in(\Psi_\mathrm{L},\Psi_1)$ ({\color{md3}{$\blacksquare$}}), $10.0024\in(\Psi_1,\Psi_2)$ ({\color{md5}{$\blacksquare$}}). See also \cref{fig:mechbifs}}
\label{fig:GBs_bif3}
\end{figure}

\section{Slipping dynamics near the GB manifold}
\label{sec:slip}
In the previous section, we have seen how the surfaces $b=0$ and $p=0$ divide  phase space. But this kinematic analysis does not tell us how the rod, when slipping, can move into another region (or mode). In this section, we consider the rod dynamics close to $b=p=0$, in order to see which transitions from slipping\footnote{We do not consider movements from the other modes for obvious reasons.} are allowed.

In the classical $2D$ Painlev\'{e} problem, for $\mu>\mu_\mathrm{P}$, G\'enot and Brogliato \cite{Genot1999} showed that slipping trajectories starting with $\theta \in (0,\theta_1)$ cannot pass through the boundary $p=0$ of the Painlev\'{e} region without also passing through the lift-off boundary $b=0$. In that case, the rod must pass through the point $P^+$ of \cref{fig:pandb} (the rod can also stick, of course)\footnote{Trajectories starting with $\theta \in (\theta_2, \frac{\pi}{2})$ are all directed away from the point $Q^+$ in \cref{fig:pandb}, so the line $p=0$ can not be crossed there at all. The line $p=0$ can not be crossed from slipping anywhere in $\Theta <0$ \cite{Genot1999}.}.

G\'enot and Brogliato \cite{Genot1999} also found a new critical value $\mu_\mathrm{C}(\alpha)$ of the coefficient of friction  given, for general $\alpha$ \cite{Or2012}, by
\begin{align}
\label{eq:muC}
\mu_\mathrm{C}(\alpha) & :=\frac{2}{\sqrt{3}}\mu_\mathrm{P}(\alpha)=\frac{4}{\alpha}\sqrt{\frac{\alpha+1}{3}}.
\end{align} 
The importance of $\mu>\mu_\mathrm{C}$ in the $2D$ problem is summarised in \cite[Theorem 1]{Or2012}. For $\mu \in (\mu_\mathrm{P},\mu_\mathrm{C}]$, all slipping solutions undergo lift-off before reaching inconsistency; the point $P^+$ cannot be reached. For $\mu>\mu_\mathrm{C}$, there is a region of initial conditions where $P^+$ is reached, where the normal reaction force becomes unbounded.

In this section, we investigate the slipping dynamics close to the GB manifold $b=p=0$ in the $3D$ problem. In \cite[section 7.2]{Champneys2016}, it was shown that system trajectories can cross $p=0$ transversely away from the GB manifold because $p(\theta,\varphi)$ tends to zero faster than $b(\Psi,\Theta,\theta)$ (see also numerical evidence in \cite[Fig. 17]{Champneys2016}). 

In fact we can show that crossing of $p=0$ away from $b=0$ is the norm in $3D$ and that the $2D$ problem is highly singular. From \cref{eq:p,eq:polareqs_varphi}, we have that
\begin{equation}\label{eq:pdotarg}
    \begin{split}
    \left.\dot{p}\right|_{p=0}&=\left(\pdiff{p}{\theta}\dot{\theta}+\pdiff{p}{\varphi}\dot{\varphi}\right)_{p=0}\\
    &=\frac{1}{\eta}(1+\alpha)\alpha\mu\cos^2\theta\cos^2\varphi\, b.
    \end{split}
\end{equation}
So, away from $\theta=\pi/2$ and $\varphi=-\pi/2$, we find that $\left.\sign{\left(\dot{p}\right)}\right|_{p=0}=\sign{\left(b\right) }$. When slipping, we have $b<0$, so in this mode we have $\left.\sign{\left(\dot{p}\right)}\right|_{p=0}=\sign{\left(b\right) } < 0$. Hence we move from $p=0$ into $p<0$. But since $b<0$, the rod moves into the inconsistent region. 

There are two cases where $\left.\sign{\left(\dot{p}\right)}\right|_{p=0}=0$ $\forall \Psi$. When $\theta=\pi/2$, the rod is vertical and we have excluded this case. The other exception is when $\varphi=-\pi/2$. When $\Psi=0$, this is the classical $2D$ Painlev\'{e} problem. In other words, the $2D$ problem is the \textit{exception} in requiring that $p=0$ can be crossed only when $b=0$.

But the result in \cref{eq:pdotarg} assumes that there is an inconsistent region, which is not the case if $|\Psi|>\Psi_2$. So we need to understand in more detail the dynamics of the $3D$ problem \cref{eq:polareqs_varphi} near $b=p=0$.

To proceed, we disregard equations for $\dot{x}$ and $\dot{y}$ in \cref{eq:polareqs_varphi}, as they are cyclic. When the rod is slipping, $z=w=0$, and so $F_z=-b/p$ from \cref{eq:polareqs_varphi}. If we substitute this result into the remaining six equations in \cref{eq:polareqs_varphi} and transform time\footnote{The new time $s$ is given by $\eta p ds=dt$. This step preserves orbits. It also reverses time for $p<0$, but this does not concern us since then the rod would be in a paradox.} in order to de-singularise $p=0$ and $\eta=0$, we find
\begin{equation}
\label{eq:general}
{\vec{x}}'=\vec{f}(\vec{x})
\end{equation}
where
\begin{equation}
\label{eq:x}
\vec{x}=(\eta, \varphi, \psi, \Psi, \theta, \Theta)^\intercal,
\end{equation}
differentiation with respect to the new time is denoted by a dash and
\begin{equation}
\begin{split}
\label{eq:f}
\vec{f}(\vec{x}) & =(f_{\eta},f_{\varphi},f_{\psi},f_{\Psi},f_{\theta},f_{\Theta})^\intercal \\
&=\begin{pmatrix}
\eta\left(-Q_1 b + A_1 p\right)\\
 -Q_2 b +( A_2 -\eta\Psi)p\\
  \eta p \Psi\\
   \eta\left(-d_1 b + c_1 p\right)\\
   \eta p \Theta\\
   \eta\left(-d_2 b + c_2 p\right))\\
   \end{pmatrix}.
\end{split}
\end{equation}

The GB manifold ($\{p(\theta,\varphi)=0\}\cap\{b(\Psi,\Theta,\theta)=0\}$) is a set of equilibria of \cref{eq:f}. Let us then consider a point on the GB manifold, where $\theta=\theta_\mathrm{GB}$, $\varphi=\varphi_\mathrm{GB}$, $\Psi=\Psi_\mathrm{GB}$, $\Theta=\Theta_\mathrm{GB}$ (such that $b(\Psi_\mathrm{GB},\Theta_\mathrm{GB},\theta_\mathrm{GB})=0$ and $p(\theta_\mathrm{GB},\varphi_\mathrm{GB})=0$). We linearise \cref{eq:general} about this point to find
\begin{gather}
\label{eq:gatheredish}
 \begin{pmatrix}\eta \\   \varphi \\  \psi \\  \Psi \\ \theta \\ \Theta   \end{pmatrix}'
 =
  \begin{pmatrix}
0 & \eta A_1p_{,\varphi}           & 0 & -\eta Q_1 b_{,\Psi}    & \eta(A_1p_{,\theta}-Q_1 b_{,\theta})                 & -\eta Q_1 b_{,\Theta}  \\
0 & (A_2 -\eta\Psi)p_{,\varphi} & 0 & -Q_2 b_{,\Psi}          & (A_2-\eta\Psi)p_{,\theta}-Q_2 b_{,\theta}                & -Q_2 b_{,\Theta}  \\
0 & \eta\Psi p_{,\varphi}           & 0 & 0                              & \eta\Psi p_{,\theta}                                                  & 0 \\
0 & \eta c_1p_{,\varphi}           & 0 & -\eta d_1 b_{,\Psi}    & \eta(c_1p_{,\theta}-d_1 b_{,\theta})                    & -\eta d_1 b_{,\Theta}  \\
0 & \eta\Theta p_{,\varphi}      & 0 & 0                             & \eta\Theta p_{,\theta}                                               & 0 \\
0 & \eta c_2p_{,\varphi}          & 0 & -\eta d_2 b_{,\Psi}    & \eta(c_2p_{,\theta}-d_2 b_{,\theta})                     & -\eta d_2 b_{,\Theta}  \\
   \end{pmatrix} 
 \begin{pmatrix} \eta \\   \varphi \\ \psi \\ \Psi \\ \theta \\  \Theta   \end{pmatrix}
\end{gather}
where $(p_{,q}, b_{,q})=\left(\pdiff{p}{q}, \pdiff{b}{q}\right)$, the square matrix is evaluated on $b=p=0$, (that is, at $\theta=\theta_\mathrm{GB}$, $\varphi=\varphi_\mathrm{GB}$, $\Psi=\Psi_\mathrm{GB}$, $\Theta=\Theta_\mathrm{GB}$)
and $\varphi, \Psi, \theta, \Theta$ now denote perturbations from the GB manifold.

Then \cref{eq:gatheredish} can be written in the form
\begin{equation}
\label{eq:jacobian}
\vec{x}'=\mat{{J}}\vec{x}=\mat{{A}}\,\mat{{B}}\,\vec{x}=\left(\frac{\partial}{\partial (p, b)} \vec{f} \right)\left(\frac{\partial}{\partial\vec{x}^\intercal} (p, b)^\intercal \right)\vec{x}
\end{equation}
where $\mat{{J}}$ is the Jacobian of $\vec{f}(\vec{x})$ and the matrices $\mat{{A}}$ and $\mat{{B}}$ are given by
\begin{align}
\label{eq:matA}
\mat{{A}} & := \begin{pmatrix}
f_{\eta,p} & f_{\eta,b}\\
f_{\varphi,p} & f_{\varphi,b}\\
f_{\psi,p} & f_{\psi,b}\\
f_{\Psi,p} & f_{\Psi,b}\\
f_{\theta,p} & f_{\theta,b}\\
f_{\Theta,p} & f_{\Theta,b}\\
   \end{pmatrix}
   =\begin{pmatrix}
\eta A_1 & -\eta Q_1\\
A_2-\eta \Psi & -Q_2\\
\eta \Psi & 0\\
   \eta c_1 & -\eta d_1\\
\eta \Theta & 0\\
   \eta c_2 & -\eta d_2\\
   \end{pmatrix}
\end{align}
and
\begin{align}
\label{eq:matB}
\mat{{B}} & := \begin{pmatrix}
p_{,\eta} & p_{,\varphi} & p_{,\psi} & p_{,\Psi} & p_{,\theta} & p_{,\Theta}\\
b_{,\eta} & b_{,\varphi} & b_{,\psi} & b_{,\Psi} & b_{,\theta} & b_{,\Theta}\\
\end{pmatrix}\\
&=\begin{pmatrix}
0 & p_{,\varphi} & 0 & 0 & p_{,\theta} & 0\\
0 & 0 & 0 & b_{,\Psi} & b_{,\theta} & b_{,\Theta}\\
\end{pmatrix}.
\end{align}
In order to understand the dynamics near $b=p=0$, we must calculate the eigenvalues and eigenvectors of the Jacobian $\mat{J}=\mat{A}\, \mat{B}$. To do this we use the following result \cite[Proposition 4.4.10]{Bernstein2009}.

\begin{lemma}
If an $m\times m$ matrix $\mat{J}$ can be expressed as the matrix product of $\mat{J}=\mat{A}\,\mat{B}$ (where $\mat{A}$ is an $m\times n$ matrix, $\mat{B}$ is an $n\times m$ matrix and $m\geq n$), then the eigenvalues of $\mat{J}$ are the $n$ eigenvalues of $\mat{K}=\mat{B}\,\mat{A}$ and $m-n$ zero eigenvalues.  If $\lambda$ is an eigenvalue of $\mat{K}$, then it is also an eigenvalue of the $\mat{J}$. 

\end{lemma}

So in \cref{eq:jacobian}, since $m=6$, the eigenvalues of $\mat{J}=\mat{A}\,\mat{B}$  are given by $n=2$  eigenvalues of  $\mat{K}=\mat{{B}}\,\mat{{A}}$, plus $m-n=4$ zero eigenvalues. 
A straightforward calculation shows that $\mat{K}$ is given by
\begin{align}
\label{eq:BA}
\mat{K} &=\begin{pmatrix}
\eta \Theta p_{,\theta} + (A_2-\eta \Psi)p_{,\varphi}            & -Q_2p_{,\varphi} \\
\eta (\Theta b_{,\theta} +c_1b_{,\Psi}+c_2b_{,\Theta})   &-\eta(d_1b_{,\Psi}+d_2b_{,\Theta})\\
\end{pmatrix}
\end{align}
evaluated on $b=p=0$. 

We shall now show that $\mat{K}$ governs the rod dynamics in the $(p,b)$-plane near $b=p=0$, the GB manifold, given by
\begin{equation}
\label{eq:jacobiany}
\vec{y}'=\mat{{K}}\,\vec{y}
\end{equation}
where $\vec{y}=(p, b)^\intercal$. 

The rate of change of $p(\theta,\varphi)$ with respect to the new time introduced in \cref{eq:general}, using \cref{eq:gatheredish}, is given by
\begin{align*}
p'&=p_{,\theta}\theta'+p_{,\varphi}\varphi'\\
&=p \left( \eta \Theta p_{,\theta}  + ( A_2 -\eta\Psi)p_{,\varphi}  \right) - Q_2p_{,\varphi} b
\end{align*} 
Similarly the derivative of $b(\Psi, \Theta, \theta)$ is given by
\begin{align*}
b'&=b_{,\Psi}\Psi'+b_{,\Theta}\Theta'+b_{,\theta}\theta'\\
&= p \eta(\Theta b_{,\theta} +c_1b_{,\Psi}+c_2b_{,\Theta})-  b\eta(d_1b_{,\Psi}+d_2b_{,\Theta})
\end{align*}
Hence
\begin{gather}\label{eq:p3a3}
 \begin{pmatrix}  p \\   b  \end{pmatrix}'
 =
  \begin{pmatrix}
\eta \Theta p_{,\theta} + (A_2-\eta \Psi)p_{,\varphi}            & -Q_2p_{,\varphi} \\
\eta (\Theta b_{,\theta} +c_1b_{,\Psi}+c_2b_{,\Theta})   &-\eta(d_1b_{,\Psi}+d_2b_{,\Theta})\\
\end{pmatrix} 
\begin{pmatrix}
   p \\   b
   \end{pmatrix}
\end{gather}
which agrees with \cref{eq:BA} and \cref{eq:jacobiany}. Hence $\mat{K}$ is the Jacobian of the rod dynamics in the $(p,b)$-plane near the GB manifold. 

Using expressions in \cref{sec:appA} for the coefficients $Q_i, A_i, d_i, c_i \; (i=1,2)$, we find that
\begin{align}
\label{eq:K}
\mat{K}
& = \begin{pmatrix}
K_{11} & K_{12}\\
K_{21} & K_{22} 
\end{pmatrix}\nonumber\\
& =\begin{pmatrix}
\alpha\mu\cos^2 \theta \cos^2 \varphi + \eta\Big[\Theta\big(\tan\theta-(1+\alpha)\cot\theta\big)   &  	(1+\alpha)	\alpha\mu\cos^2 \theta \cos^2 \varphi\\
-\Psi\alpha\mu\sin\theta\cos\theta\cos\varphi\Big]                                                             &    \\
\eta\Theta\cot\theta                                                                            &\eta(2\Psi\alpha\mu\sin\theta\cos\theta\cos\varphi-2\Theta\tan\theta) \\
\end{pmatrix},
\end{align}
and hence
\begin{equation}
 \label{eq:detK}
\begin{split}
\det{\mat{K}} & =\eta\alpha\cos^2\theta\cos^2\varphi \Big[\Theta\big(-2\tan\theta -(1+\alpha)\cot\theta\big)+2\Psi\alpha\mu\sin\theta\cos\theta\cos\varphi\Big]\\
&\phantom{=}\qquad +\eta^2\Big[\Theta^2\big(2(1+\alpha)-2\tan^2\theta\big)+\Theta\Psi\alpha\mu\sin\theta\cos\theta\cos\varphi\big(4\tan\theta-2(1+\alpha)\cot\theta\big) \\
&\phantom{=}\qquad\qquad\quad-2\Psi^2(\alpha\mu\sin\theta\cos\theta\cos\varphi)^2\Big].
\end{split}
\end{equation}
We note the following group symmetry:
\begin{align}
\label{eq:group}
\mat{K}|_{\varphi=\bar{\varphi},\Psi=\bar{\Psi}}&\equiv \mat{K}|_{\varphi=-\pi-\bar{\varphi},\Psi=-\bar{\Psi}}    
\end{align}
and the absence of any symmetry in $\Theta$ when $\theta\in [0,\pi/2]$ .

Locally to the GB manifold, $\vec{y}=(p,b)^\intercal=\vec{0}$, the dynamics of the rod are approximately governed by \cref{eq:jacobiany} and hence  described by the eigenvalues $\lambda_\pm$ and eigenvectors $\vec{e}_\pm$ of $\mat{K}$.

Let us briefly recap the properties of $\mat{K}$ for the planar ($2D$) problem, where the GB manifold reduces to the four points $P^{\pm}, Q^{\pm}$, see \cref{fig:p3s} and \cite[fig.2]{Genot1999}. From \cref{eq:K} with $\varphi=-\frac{\pi}{2}$ we have
\begin{align}
\label{eq:K2d}
\mat{K}_{\varphi=-\frac{\pi}{2}} \equiv
\mat{K}_{2D}
& =\eta\Theta\begin{pmatrix}
\tan\theta-(1+\alpha)\cot\theta   &  	0\\
\cot\theta                          & -2\tan\theta\\
\end{pmatrix},
\end{align}
evaluated at $\theta=\theta_{1,2}$ from \cref{eq:theta12}. The eigenvalues $\lambda_{1,2}=\lambda_{\pm} |_{\varphi=-\frac{\pi}{2}}$ and associated eigenvectors $\vec{e}_{1,2}=\vec{e}_{\pm} | _{\varphi=-\frac{\pi}{2}}$ of $\mat{K}_{2D}$ are
\begin{align}
\label{eq:evalK2D}
\left \{ \lambda_1, \lambda_2 \right \} & = \left \{ \eta\Theta(\tan \theta - (1+\alpha) \cot \theta), -2\eta\Theta\tan \theta \right \}
\end{align}
and
\begin{align}
\label{eq:evecK2D}
\left \{ \vec{e}_1, \vec{e}_2 \right \} & = \left \{ \begin{pmatrix} 3\tan^2 \theta-(1+\alpha)\\ 1 \\ \end{pmatrix}, \begin{pmatrix} 0 \\ 1 \\ \end{pmatrix} \right \}.
\end{align}
It can be shown that $\lambda_1 = \eta\Theta(\tan \theta - (1+\alpha) \cot \theta) =0$ when $\mu = \mu_\mathrm{P}$ and that $e_{1,1} = 3\tan^2 \theta-(1+\alpha) = 0$ when $\mu = \mu_\mathrm{C}$. In summary, for the $2D$ problem, with respect to the new time introduced in  \cref{eq:general},
\begin{itemize}
\item For $\mu \in [0,\mu_\mathrm{P})$, there is no paradox.
\item For $\mu \in (\mu_\mathrm{P},\mu_\mathrm{C})$:
\begin{itemize}
\item $P^+$ is a stable node and $P^-$ is an unstable node; $\vec{e}_1$ lies in the first and third quadrants of the $(p,b)$ plane and $\vec{e}_2$ lies on the $b$-axis.
\item $Q^{\pm}$ are both saddles.
\end{itemize}
\item For $\mu > \mu_\mathrm{C}$:
\begin{itemize}
\item $P^+$ is a stable node and $P^-$ is an unstable node; $\vec{e}_1$ in the second and fourth quadrants  of the $(p,b)$ plane and $\vec{e}_2$ lies on the $b$-axis.
\item $Q^{\pm}$ are both saddles.
\end{itemize}
\end{itemize}
Mathematically we see that $\mu=\mu_\mathrm{C}$ corresponds to a change in the orientation of $\vec{e}_1$ relative to $\vec{e}_2$.

The $2D$ problem is the $\varphi=-\frac{\pi}{2}$ section of the $3D$ problem with $\Psi=0$ (\cref{fig:p3s}), so as we move around the GB manifold, as in \cref{fig:GBs_bif3}, nodes at $P^{\pm}$ must undergo a bifurcation to become saddles at $Q^{\pm}$, for any value of $\mu>\mu_\mathrm{P}$.

To fully understand the dynamics of the $3D$ problem, we will need a combination of analysis and numerical methods. We show one numerical example in \cref{fig:planesmu1p4_w_traj}, where we set $\eta=1$, $\Psi=0$, $\alpha=3$ and $\mu=1.4$. Each figure shows the eigenvectors\footnote{The eigenvectors are shown in the new time introduced in  \cref{eq:general}. Orbits in the first, second and third quadrants of \cref{fig:planesmu1p4_w_traj,fig:planesmu1p7_w_traj,fig:ppsThetapos,fig:ppsThetaneg} are artefacts.} in the $(p,b)$-plane at selected points around the GB manifold with $\Theta>0$ (shown in \cref{fig:p3.2}). The shading corresponds to that of \cref{fig:pandb}. \cref{fig:Plane1small} shows a stable node (the point $P^+$ in \cref{fig:pandb}). The strong direction (in blue) is aligned with the $b$-axis and the weak direction (in red) lies in the first and third quadrants. \cref{fig:Plane5small} shows a saddle (the point $Q^+$ in \cref{fig:pandb}). 

As we move through \cref{fig:planesmu1p4_w_traj}, we move around the GB manifold from $\theta=\theta_1$ to $\theta=\theta_2$. At a certain point,
shown in \cref{fig:Plane3small}, the eigenvalue $\lambda_+$ passes through zero, as the stable node becomes a saddle. At this point, the system is non-hyperbolic.

In each of \cref{fig:Plane2small,fig:Plane3small,fig:Plane4small}, the orbit that approaches the GB manifold along the (strong) stable manifold at that point acts as a separatrix\footnote{We use the term separatrix to mean an orbit, or set thereof, which locally separates orbits that result in different behaviours. 
} between orbits that reach the inconsistent region and those that lift off. 

We can now see how the $2D$ picture (\cref{fig:pandb}) fits into these results. Moving around the GB manifold from \cref{fig:Plane1small}, the eigenvector frame distorts away from the $b$-axis, allowing movement from the slipping region (in yellow) to the inconsistent region (in lime). After the bifurcation (\cref{fig:Plane3small}), the eigenvector frame distorts back towards the $b$-axis, eventually giving \cref{fig:Plane5small}, the point  $Q^+$ in \cref{fig:pandb}, where crossing $p=0$ is no longer possible.

\begin{figure}[htbp]
    \centering
    \setlength{\templength}{0.19\textwidth}
        \begin{subfigure}[t]{\templength}
        \begin{overpic}[ width=\textwidth]{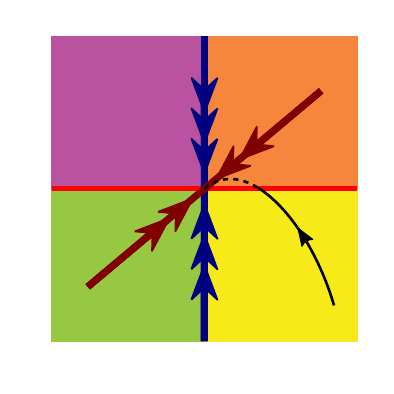}
        \put(50,0){$p$}
        \put(25,0){$-$}
        \put(75,0){$+$}
        \put(0,50){$b$}
        \put(0,25){$-$}
        \put(0,75){$+$}
        \end{overpic}
        \caption{}
        \label{fig:Plane1small}
    \end{subfigure}
    \hfil
    \begin{subfigure}[t]{\templength}
        \begin{overpic}[ width=\textwidth]{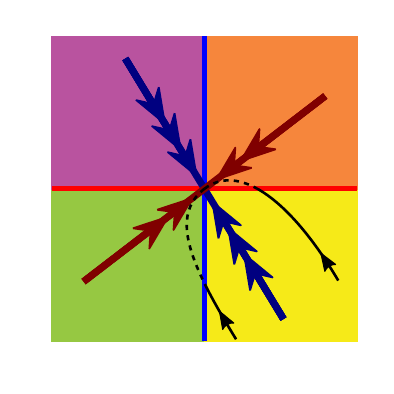}
        \put(50,0){$p$}
        \put(25,0){$-$}
        \put(75,0){$+$}
        \put(0,50){$b$}
        \put(0,25){$-$}
        \put(0,75){$+$}
        \end{overpic}
        \caption{}
        \label{fig:Plane2small}
    \end{subfigure}
    \hfil
    \begin{subfigure}[t]{\templength}
        \begin{overpic}[ width=\textwidth]{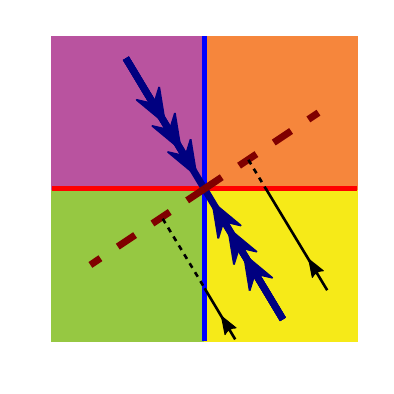}
        \put(50,0){$p$}
        \put(25,0){$-$}
        \put(75,0){$+$}
        \put(0,50){$b$}
        \put(0,25){$-$}
        \put(0,75){$+$}
        \end{overpic}
        \caption{}
        \label{fig:Plane3small}
    \end{subfigure}
    \begin{subfigure}[t]{\templength}
        \begin{overpic}[ width=\textwidth]{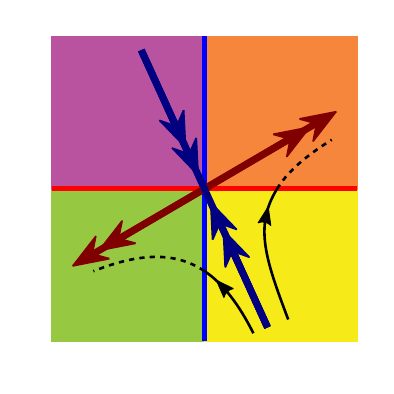}
        \put(50,0){$p$}
        \put(25,0){$-$}
        \put(75,0){$+$}
        \put(0,50){$b$}
        \put(0,25){$-$}
        \put(0,75){$+$}
        \end{overpic}
        \caption{}
        \label{fig:Plane4small}
    \end{subfigure}
    \hfil
        \begin{subfigure}[t]{\templength}
        \begin{overpic}[ width=\textwidth]{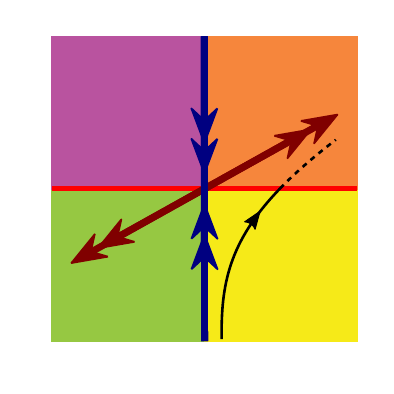}
        \put(50,0){$p$}
        \put(25,0){$-$}
        \put(75,0){$+$}
        \put(0,50){$b$}
        \put(0,25){$-$}
        \put(0,75){$+$}
        \end{overpic}
        \caption{}
        \label{fig:Plane5small}
    \end{subfigure}
    \caption{Dynamics in the $(p,b)$-plane for the $3D$ Painlev\'e paradox for $\eta=1$, $\Psi=0$, $\alpha=3$ and $\mu=1.4$ close to the GB manifold $b=p=0$ when $\Psi=0$. The shading of each region corresponds to that of \cref{fig:pandb}: slipping (yellow), lift-off (orange), indeterminate (purple), inconsistent (lime). Each subfigure shows eigenvectors at selected points around the GB manifold (see \cref{fig:p3.2}), with coordinates $(\theta, \varphi, \Psi, \Theta)=(\theta_\mathrm{GB}, \varphi_\mathrm{GB}, 0, \Theta_\mathrm{GB})$, where $\Theta_\mathrm{GB}=\sqrt{\csc \theta_\mathrm{GB}}$ from \cref{eq:b}: $(\theta_\mathrm{GB}, \varphi_\mathrm{GB})$ are given by  {\bf (a)} $(\theta_1,-\frac{\pi}{2})=(0.9702,-1.5708)$ corresponding to $P^+$ in \cref{fig:pandb}, {\bf (b)} $(1.0094,-1.7832)$, {\bf (c)} $(1.0953,-1.8793)$, {\bf (d)} $(1.1872,-1.7978)$ and {\bf (e)} $(\theta_2,-\frac{\pi}{2})=(1.2209,-1.5708)$, corresponding to $Q^+$ in \cref{fig:pandb}. The eigenvectors are shown in the new time introduced in  \cref{eq:general} which reverses $t$ in $p<0$. They are not drawn to scale.}
    \label{fig:planesmu1p4_w_traj} 
\end{figure}
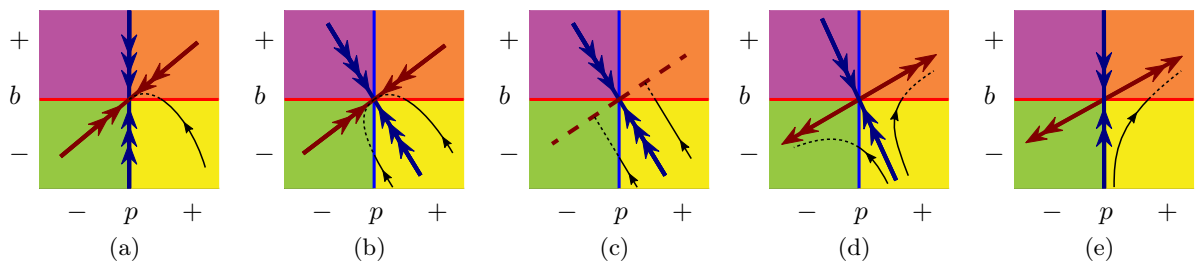

We show another numerical example in \cref{fig:planesmu1p7_w_traj}, where $\eta=1$, $\Psi=0$, $\alpha=3$ and now $\mu=1.7 > \mu_\mathrm{C}= \frac{8}{3\sqrt{3}} \approx 1.5396$. \cref{fig:Plane1small1p7,fig:Plane5small1p7} correspond to the $2D$ problem. In \cref{fig:Plane1small1p7}, the strong direction (in blue) now lies in the first and third quadrants, with the weak direction (in red) aligned with the $b$-axis.  Moving around the GB manifold from \cref{fig:Plane1small1p7}, the eigenvector frame distorts, with the weak (red) eigenvector moving off the $b$-axis, into the first and third quadrants, allowing movement from the slipping region (in yellow) to the inconsistent region (in lime). After the bifurcation (\cref{fig:Plane3small1p7}), the eigenvector frame distorts even more, with the strong (blue) eigenvector eventually aligning itself with the $b$-axis, where crossing the line $p=0$ is no longer possible. The weak (red) eigenvector remains in the first and third quadrants. In each of \cref{fig:Plane2small1p7,fig:Plane3small1p7,fig:Plane4small1p7}, the orbits that approach along the (strong) stable manifold of points along the GB manifold act as a separatrix between orbits that reach the inconsistent region and those that lift off.

Note that \cref{fig:Plane1small,fig:Plane1small1p7} are the one qualitatively different pair between \cref{fig:planesmu1p4_w_traj,fig:planesmu1p7_w_traj}, which suggests that the critical value $\mu_\mathrm{C}$ maybe be less significant in the $3D$ problem.

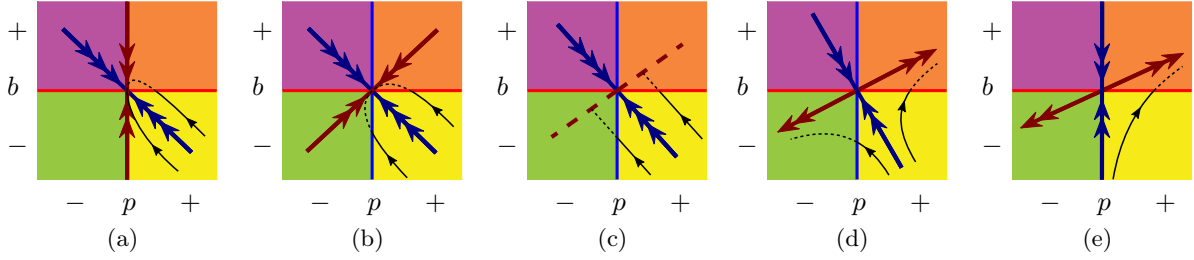
\begin{figure}[htbp]
    \centering
    \setlength{\templength}{0.19\textwidth}
    \begin{subfigure}[t]{\templength}
        \begin{overpic}[ width=\textwidth]{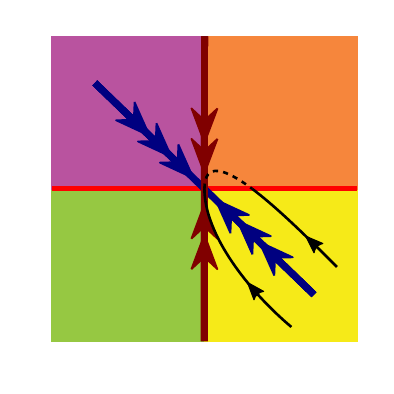}
        \put(50,0){$p$}
        \put(25,0){$-$}
        \put(75,0){$+$}
                \put(0,50){$b$}
        \put(0,25){$-$}
        \put(0,75){$+$}
        \end{overpic}
        \caption{}
        \label{fig:Plane1small1p7}
    \end{subfigure}
    \hfil
    \begin{subfigure}[t]{\templength}
        \begin{overpic}[ width=\textwidth]{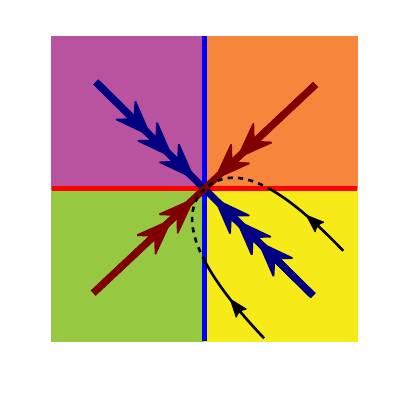}
        \put(50,0){$p$}
        \put(25,0){$-$}
        \put(75,0){$+$}
                \put(0,50){$b$}
        \put(0,25){$-$}
        \put(0,75){$+$}
        \end{overpic}
        \caption{}
        \label{fig:Plane2small1p7}
    \end{subfigure}
    \hfil
    \begin{subfigure}[t]{\templength}
        \begin{overpic}[ width=\textwidth]{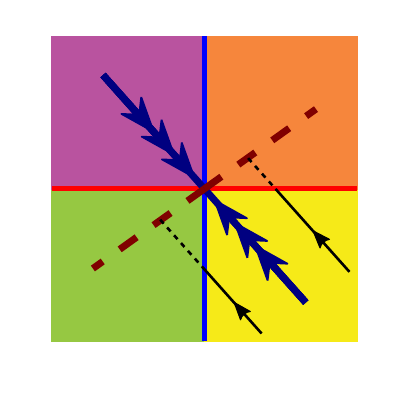}
        \put(50,0){$p$}
        \put(25,0){$-$}
        \put(75,0){$+$}
                \put(0,50){$b$}
        \put(0,25){$-$}
        \put(0,75){$+$}
        \end{overpic}
        \caption{}
        \label{fig:Plane3small1p7}
    \end{subfigure}
    \begin{subfigure}[t]{\templength}
        \begin{overpic}[ width=\textwidth]{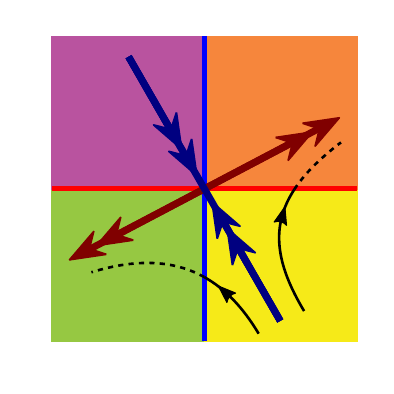}
        \put(50,0){$p$}
        \put(25,0){$-$}
        \put(75,0){$+$}
                \put(0,50){$b$}
        \put(0,25){$-$}
        \put(0,75){$+$}
        \end{overpic}
        \caption{}
        \label{fig:Plane4small1p7}
    \end{subfigure}
    \hfil
        \begin{subfigure}[t]{\templength}
        \begin{overpic}[ width=\textwidth]{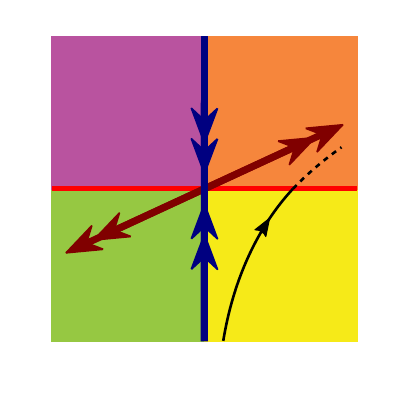}
        \put(50,0){$p$}
        \put(25,0){$-$}
        \put(75,0){$+$}
                \put(0,50){$b$}
        \put(0,25){$-$}
        \put(0,75){$+$}
        \end{overpic}
        \caption{}
        \label{fig:Plane5small1p7}
    \end{subfigure}
    \caption{As \cref{fig:planesmu1p4_w_traj}, for $\mu=1.7$. Here $(\theta_\mathrm{GB}, \varphi_\mathrm{GB})$ are given by  {(a)} $(\theta_1,-\frac{\pi}{2})=(0.7692, -1.5708)$, {(b)} $(0.8517,-1.9969)$, {(c)} $(1.0227,-2.2143)$, {(d)} $(1.2562,2.1107)$ and {(e)} $(\theta_2,-\frac{\pi}{2})=(1.3333,-1.5708)$.} 
    \label{fig:planesmu1p7_w_traj} 
\end{figure}
\cref{fig:Plane3small,fig:Plane3small1p7} show bifurcations on the GB manifold. In the following section we embark on a general study of such bifurcations. However, we shall first give some general statements about the generic dynamics local to the GB manifold.  

{\begin{theorem}\label{thm:GB}
Local to the GB manifold but for $\varphi\neq-\frac{\pi}{2}$:
\begin{enumerate}
    \item{ For $\Theta>0$, there is a separatrix between orbits that reach the inconsistent paradox and those that lift off, and there are three generic phase portraits in $(p,b)$ space (shown in \cref{fig:ppsThetaneg}).    }
    \item{ For $\Theta<0$, there are 5 generic phase portraits in $(p,b)$ space (shown in \cref{fig:ppsThetaneg}), and depending on the local phase portrait either the rod remains slipping or reaches inconsistency locally. Lift-off is not possible and the GB manifold is unstable.}
\end{enumerate}
\end{theorem}}
\begin{figure}[htbp]
    \setlength{\templength}{0.19\textwidth}
    \centering
    \begin{subfigure}[t]{\templength}
        \begin{overpic}[ width=\textwidth]{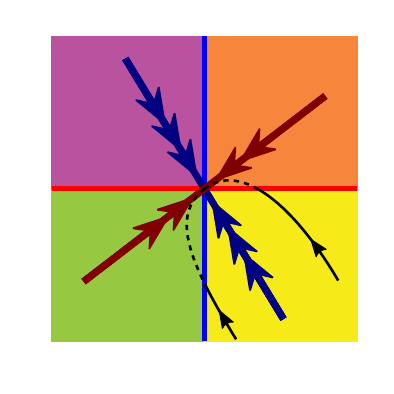}
        \put(50,0){$p$}
        \put(25,0){$-$}
        \put(75,0){$+$}
        \put(0,50){$b$}
        \put(0,25){$-$}
        \put(0,75){$+$}
        \end{overpic}
        \caption{}
        \label{fig:pp1}
    \end{subfigure}
    \hfil
        \begin{subfigure}[t]{\templength}
        \begin{overpic}[ width=\textwidth]{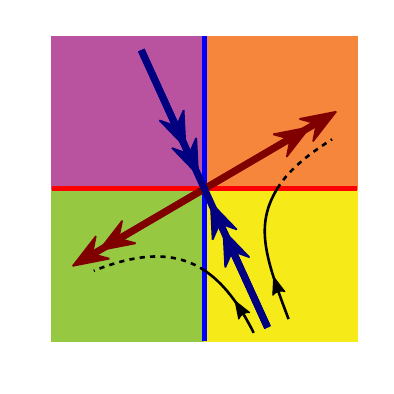}
        \put(50,0){$p$}
        \put(25,0){$-$}
        \put(75,0){$+$}
        \put(0,50){$b$}
        \put(0,25){$-$}
        \put(0,75){$+$}
        \end{overpic}
        \caption{}
        \label{fig:pp2}
    \end{subfigure}
    \hfil
        \begin{subfigure}[t]{\templength}
        \begin{overpic}[ width=\textwidth]{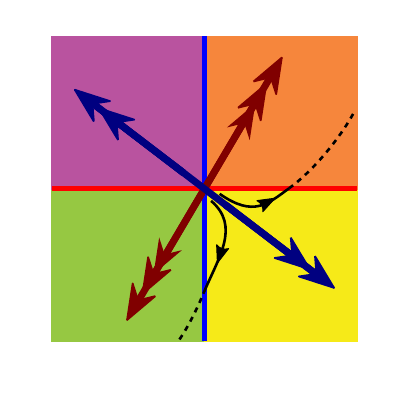}
        \put(50,0){$p$}
        \put(25,0){$-$}
        \put(75,0){$+$}
        \put(0,50){$b$}
        \put(0,25){$-$}
        \put(0,75){$+$}
        \end{overpic}
        \caption{}
        \label{fig:pp8}
    \end{subfigure}
    \caption{Possible dynamics in $(p,b)$ space local to the GB manifold for $\Theta>0$. In (a), the point on the GB manifold is a stable node where the strong eigenvector is in the fourth quadrant and the corresponding strong stable manifold acts as a separatrix between orbits that reach the inconsistent region and those that lift off. In (b), the point on the GB manifold is a saddle where stable eigenvector is in the fourth quadrant and the corresponding stable manifold acts as a separatrix between orbits that reach the inconsistent region and those that lift off. In (c), the point on the GB manifold is an unstable node where the weakly unstable eigenvector is in the fourth quadrant and the corresponding weakly unstable manifold acts as a separatrix between orbits that reach the inconsistent region and those that lift off.  }
    \label{fig:ppsThetapos} 
\end{figure}
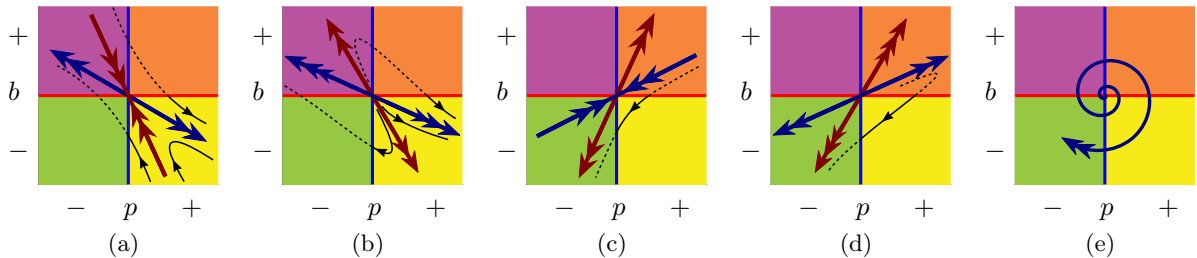
\begin{figure}[htbp]
    \setlength{\templength}{0.19\textwidth}
    \centering
    \begin{subfigure}[t]{\templength}
        \begin{overpic}[ width=\textwidth]{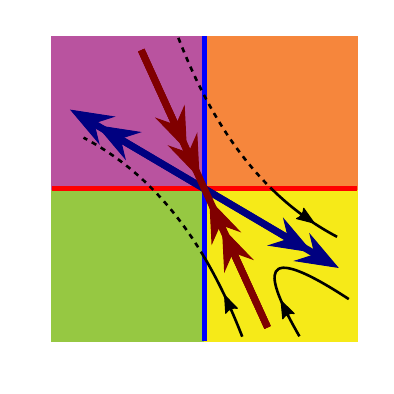}
        \put(50,0){$p$}
        \put(25,0){$-$}
        \put(75,0){$+$}
        \put(0,50){$b$}
        \put(0,25){$-$}
        \put(0,75){$+$}
        \end{overpic}
        \caption{}
        \label{fig:pp3}
    \end{subfigure}
    \hfil
        \begin{subfigure}[t]{\templength}
        \begin{overpic}[ width=\textwidth]{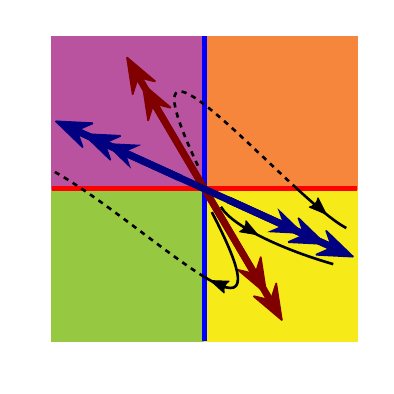}
        \put(50,0){$p$}
        \put(25,0){$-$}
        \put(75,0){$+$}
        \put(0,50){$b$}
        \put(0,25){$-$}
        \put(0,75){$+$}
        \end{overpic}
        \caption{}
        \label{fig:pp6}
    \end{subfigure}
    \hfil
        \begin{subfigure}[t]{\templength}
        \begin{overpic}[ width=\textwidth]{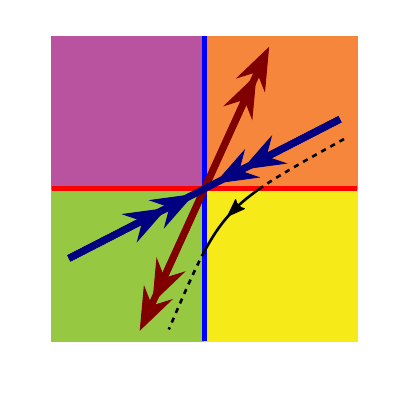}
        \put(50,0){$p$}
        \put(25,0){$-$}
        \put(75,0){$+$}
        \put(0,50){$b$}
        \put(0,25){$-$}
        \put(0,75){$+$}
        \end{overpic}
        \caption{}
        \label{fig:pp5}
    \end{subfigure}
    \hfil
        \begin{subfigure}[t]{\templength}
        \begin{overpic}[ width=\textwidth]{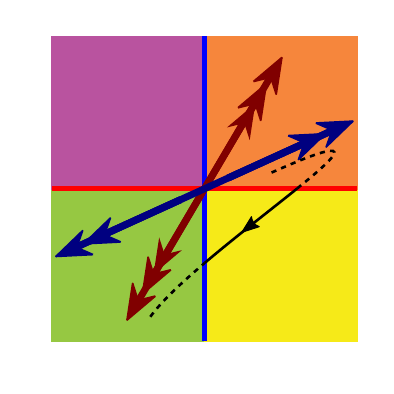}
        \put(50,0){$p$}
        \put(25,0){$-$}
        \put(75,0){$+$}
        \put(0,50){$b$}
        \put(0,25){$-$}
        \put(0,75){$+$}
        \end{overpic}
        \caption{}
        \label{fig:pp4}
    \end{subfigure}
    \hfil
        \begin{subfigure}[t]{\templength}
        \begin{overpic}[ width=\textwidth]{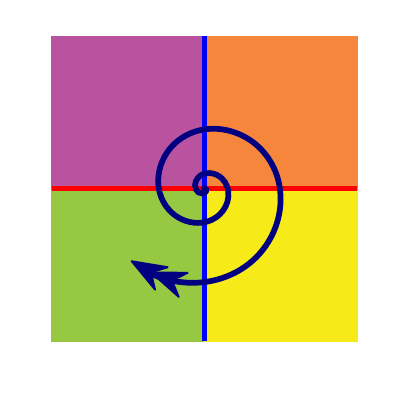}
        \put(50,0){$p$}
        \put(25,0){$-$}
        \put(75,0){$+$}
        \put(0,50){$b$}
        \put(0,25){$-$}
        \put(0,75){$+$}
        \end{overpic}
        \caption{}
        \label{fig:pp7}
    \end{subfigure}
    \caption{Possible dynamics in $(p,b)$ space local to the GB manifold for $\Theta<0$. In (a) and (b), orbits either remain slipping, following the (strongly) unstable manifold of the GB manifold in the fourth quadrant, or reach the inconsistent region, tending to the (strongly) unstable manifold of the GB manifold in the second quadrant. In (c)-(e) all local orbits tend to reach inconsistency. }
    \label{fig:ppsThetaneg} 
\end{figure}
\begin{proof}Consider the Jacobian $\mat{K}$ from \eqref{eq:K}, with eigenvectors $\lambda_\pm$ and eigenvectors $\vec{e}_\pm$ given by
\begin{equation}
    \vec{e}_\pm:=\left(\frac{K_{12}}{\lambda_\pm-K_{11}},1\right)^\intercal
    .
\end{equation}
The product of the first component of the eigenvectors is
\begin{align}
        \left(\frac{K_{12}}{\lambda_+-K_{11}}\right)\left(\frac{K_{12}}{\lambda_--K_{11}}\right)&=\frac{K_{12}^2}{\lambda_+\lambda_- - K_{11}\left(\lambda_1+\lambda_2\right)+K_{11}^2}\nonumber\\
        &=\frac{K_{12}^2}{\left(K_{11}K_{22}-K_{12}K_{21}\right) - K_{11}\left(K_{11}+K_{22}\right)+K_{11}^2}\nonumber\\
        &\equiv-\frac{K_{12}}{K_{21}}.\label{eq:offdiagonals}
    \end{align}
\begin{enumerate}
    \item{For $\Theta>0$, the product \cref{eq:offdiagonals} is negative, from \cref{eq:K} and hence the eigenvectors are in different quadrants. Furthermore, from the signs of the off-diagonal elements of $\mat{K}$, we can infer that a point on the GB manifold is either a stable node with the strong eigenvector in the fourth quadrant (see \cref{fig:pp1}), a saddle with the stable eigenvector in the fourth quadrant (see \cref{fig:pp2}), or an unstable node with the weakly unstable eigenvector in the fourth quadrant (see \cref{fig:pp8}). Each of the cases in \cref{fig:pp1,fig:pp2,fig:pp8} can be realised.}
    \item{For $\Theta<0$, the product \cref{eq:offdiagonals} is positive, and hence the eigenvectors (if real) are in the same quadrant. Furthermore, eliminating $\Psi$ from $\mat{K}$, we find
    \begin{align}
    K_{11}&=+2\alpha\mu\cos^2\theta\cos^2\varphi-2(1+\alpha)\eta\Theta\cot\theta-\tr(\mat{K}),\\
    K_{22}&=-2\alpha\mu\cos^2\theta\cos^2\varphi+2(1+\alpha)\eta\Theta\cot\theta +2\tr(\mat{K}).
\end{align} 
where $\tr(\mat{K})$ is the trace of $\mat{K}$. Hence, the determinant of $\mat{K}$ \cref{eq:detK} can be written as the following polynomial in $\tr(\mat{K})$:
\begin{align}
    \det(\mat{K})&=K_{11}K_{22}-K_{12}K_{21}\nonumber\\
    &=C_2\tr(\mat{K})^2+C_1\tr(\mat{K})+C_0\label{eq:dettra}
    \end{align}
    where
    \begin{align}
    C_2&:=-2,\\
    C_1&:=6\left(\alpha\mu\cos^2\theta\cos^2\varphi-(1+\alpha)\eta\Theta\cot\theta\right),\\
    C_0&:=\left(-4\eta^2\Theta^2(1+\alpha)^2\cot^2\theta +7\eta\Theta (1+\alpha) \alpha\mu\cot\theta\cos^2\theta-4\alpha^2\mu^2\cos^4\theta\cos^4\varphi\right).
\end{align}
We can show that there are no negative roots of this polynomial in $\tr(\mat{K})$ when $\Theta<0$, using Descartes' rule of signs and so stable nodes and foci are not possible. From the signs of the off-diagonal entries of $\mat{K}$, we can deduce that only the phase portraits in \cref{fig:ppsThetaneg} are possible. Hence lift-off is not possible locally. Each case in \cref{fig:pp3,fig:pp4,fig:pp5,fig:pp6,fig:pp7} can be realised.}
\end{enumerate}
\end{proof}
\begin{corollary}\label{cor:main}
For $\mu>\mu_\mathrm{P}$, there exists an open set of initial conditions that reach the inconsistent region from slipping in finite time. 
\end{corollary}
\begin{proof}
Follows from \cref{thm:GB}.
\end{proof}

In \cref{fig:pp3,fig:pp6}, orbits either follow the (strongly) unstable manifold of the GB manifold in the fourth quadrant, and remain slipping, or tend to the (strongly) unstable manifold of the GB manifold in the second quadrant, reaching the inconsistent region. In \cref{fig:pp4,fig:pp5,fig:pp7}, all orbits reach the inconsistent region, whether tending to a (strongly) unstable manifold in the third quadrant (\cref{fig:pp4,fig:pp5}) or rotating round the GB manifold (locally a set of unstable foci \cref{fig:pp7}).

In \cref{fig:pp1,fig:pp2,fig:pp3}, there exist unique orbits which reach the GB manifold. These types of orbit have been shown to be of interest \cite{Kristiansen2018}.


\subsection{Zero eigenvalue}
\label{subsec:zeroeigenvalue}
\cref{fig:planesmu1p4_w_traj,fig:planesmu1p7_w_traj} show that the dynamics of the $3D$ problem  is governed, not only by changes in the topology of the GB manifold $b=p=0$, but also by changes in the sign of the eigenvalues.

At least one of the eigenvalues vanishes when $\det{\mat{K}}=\lambda_+\lambda_-=K_{11}K_{22}-K_{12}K_{21}=0$. Hence, from \cref{eq:detK}, $\det{\mat{K}}=0$ when
\begin{align}
\label{eq:detK0}
 A \Psi^2+ B \Psi\Theta+\Gamma\Psi &= \Delta\Theta^2+E \Theta
\end{align}
where
\begin{align}
\label{eq:detK0coeffs}
A & = -2\eta\alpha^2\mu^2\tan^3\theta\cos^2\varphi, \nonumber \\
B & = 2\eta\alpha\mu\cos\varphi\tan\theta(1+\tan^2\theta)[2\tan^2\theta-(1+\alpha)], \nonumber \\
\Gamma & = 2\alpha^2\mu^2\cos^3\varphi\tan^2\theta, \nonumber \\
\Delta & = 2\eta\tan\theta(1+\tan^2\theta)^2[\tan^2\theta-(1+\alpha)], \nonumber \\
E & = \alpha\mu\cos^2\varphi (1+\tan^2\theta)[2\tan^2\theta+(1+\alpha)].
\end{align}
Equations \cref{eq:detK0} and \cref{eq:detK0coeffs} must be evaluated on the GB manifold $b=p=0$. The term $p(\theta,\varphi)$ is independent of $\Psi$ and $\Theta$ \cref{eq:p}, and $b(\Psi,\Theta,\theta)=0$ can be written as 
\begin{align}
\label{eq:B}
M\Psi^2+\Theta^2=N
\end{align} 
where 
\begin{align}
\label{eq:Bcoeffs}
M & = \frac{1}{1+\tan^2\theta}, \nonumber \\
N & = \frac{(1+\tan^2\theta)^{\frac{1}{2}}}{\tan\theta}
\end{align}
\cref{eq:b}.
In $(\Psi,\Theta)$-space, \cref{eq:detK0} is a conic, depending on the sign of  the discriminant $B^2+4A\Delta$, and \cref{eq:B} is an ellipse. So in general, we can expect 0, 2 or 4 solutions of these equations, corresponding to intersections of the two conics. Eliminating $\Theta$ from \cref{eq:detK0} and \cref{eq:B} gives a quartic in $\Psi$:
\begin{align}
\label{eq:Psiquartic}
\sum_{i=0}^4C_i\Psi^i=0,
\end{align}
where
\begin{align}
C_4 & = (A+M\Delta)^2+MB^2, \nonumber \\
C_3 & = 2[A\Gamma+M(\Gamma\Delta-EB)], \nonumber \\
C_2 & = \Gamma^2-2\Delta N (A+M\Delta)-NB^2+ME^2, \nonumber \\
C_1 & = 2N(EB-\Gamma\Delta), \nonumber \\
C_0 & = N(N\Delta^2-E^2).
\end{align}
Since the $C_i$ are all real, we do indeed have either 0, 2 or 4 real roots for $\Psi$, which will depend on $\eta$ in general.

A theoretical analysis of \cref{eq:Psiquartic} becomes unwieldy very quickly. In order to progress, we focus on the case $\Psi=0$ and on the behaviour near $\mu=\mu_\mathrm{P}$ for $\Psi\ne0$. 

\subsection{The case \texorpdfstring{$\Psi=0$}{Psi=0}}
\label{subsec:bifPsi0}
From \cref{eq:K}, a focus occurs when $(K_{11}-K_{22})^2+4K_{12}K_{21}<0$. Setting $\Psi=0$ in \cref{eq:K}, we see that
\begin{align}
\label{eq:surd}
 (K_{11}-K_{22})^2+4K_{12}K_{21} & = R_2\Theta^2 + R_1\Theta + R_0  
\end{align}
where
\begin{align}
\label{eq:surdcoeffs}
    R_2 & = \eta^2[3\tan\theta-(1+\alpha)\cot\theta)]^2, \nonumber \\
    R_1 & = 2\eta\alpha\mu\cos^2\theta\cos^2\varphi[3\tan\theta+(1+\alpha)\cot\theta], \nonumber \\
    R_0 & = \alpha^2\mu^2\cos^4\theta\cos^4\varphi.
\end{align}
The quadratic \cref{eq:surd} is always positive when $\Theta>0$, and so foci do not occur on the GB manifold in this case.

But it is possible for $(K_{11}-K_{22})^2+4K_{12}K_{21} <0$ for $\mu>\mu_\mathrm{C}$ when $\Theta \in (\Theta_-, \Theta_+)$ where
\begin{align}
\label{eq:Thetapm}
    \Theta_{\pm} & = -\frac{\alpha\mu\cos^2\theta\cos^2\varphi}{\eta[\sqrt{3\tan\theta}\pm\sqrt{(1+\alpha)\cot\theta}]^2}<0.
\end{align}

Numerical results show that unstable foci can also occur in $\Theta<0$ when $\Psi \ne 0$, see \cref{sec:alltogether}.

When $\Psi = 0$, \cref{eq:Psiquartic} reduces to 
\begin{align}
\label{eq:C0}
C_0=0 & \iff N\Delta^2-E^2=0,
\end{align}
 since $N \ne 0$ in general\footnote{$N=0$ corresponds to $\theta=\frac{\pi}{2}$, the case of the vertical rod, which we have excluded from our analysis.}. After a lengthy calculation,  \cref{eq:C0} reduces to 
 \begin{align}
 \label{eq:ghastly}
 \alpha\mu\cos^2\varphi & =\pm2\eta\tan^{\frac{1}{2}}\theta(1+\tan^2\theta)^{\frac{5}{4}}\frac{[(1+\alpha)-\tan^2\theta]}{[(1+\alpha)+2\tan^2\theta]}
  \end{align}
where $\pm$ corresponds to $\Theta \gtrless 0$. We use \cref{eq:ghastly} as a check on our numerical results in \cref{sec:alltogether}.


\subsection{Behaviour near \texorpdfstring{$\mu=\mu_\mathrm{P}$}{critical mu}}
\label{subsec:nearmuP}
We consider \cref{eq:Psiquartic} near $\mu=\mu_\mathrm{P}=\frac{2}{\alpha}\sqrt{1+\alpha}$. From \cref{fig:muPsicases}, we can see that any perturbation analysis will tell us about bifurcations in the dynamics on the GB manifold for $\Psi \lesssim \Psi_\mathrm{P}$, where $\Psi_\mathrm{P}$ is given in \cref{eq:psip}. 

At $\mu = \mu_\mathrm{P}$ we have $\varphi=-\frac{\pi}{2}$ and $\tan \theta_\mathrm{P}=\sqrt{1+\alpha}$. Let us take
\begin{align}
\label{eq:mupperturb}
\mu & = \mu_\mathrm{P}(1+\varepsilon^2\hat{\mu}) = \frac{2}{\alpha}\sqrt{1+\alpha}(1 +\varepsilon^2\hat{\mu}), \nonumber \\
\varphi & =  -\frac{\pi}{2} + \varepsilon \hat{\varphi}, \nonumber \\
\theta_{\pm}&=\theta_\mathrm{P}+\varepsilon\hat{\theta}_{\pm},
\end{align}
where $\varepsilon \ll 1$ and $\hat{\mu} \ge 0$ since there is no paradox for $\mu <\mu_\mathrm{P}$. Then 
\begin{align}
\label{eq:tanthetaperturb}
\tan \theta_{\pm}&\approx\tan \theta_\mathrm{P}+\varepsilon\hat{\theta}_{\pm}(1+\tan^2 \theta_\mathrm{P})=\sqrt{1+\alpha}+(2+\alpha)\varepsilon\hat{\theta}_{\pm}
\end{align}
and hence from \cref{eq:thetastar} we have
\begin{align}
\label{eq:thetaperturb}
\hat{\theta}_{\pm} &= \pm\frac{\sqrt{1+\alpha}}{2+\alpha}\sqrt{2\hat{\mu}-\hat{\varphi}^2}.
\end{align}
In addition, we note that $\cos \varphi =\cos (-\frac{\pi}{2} + \varepsilon \hat{\varphi}) \approx \varepsilon \hat{\varphi}$. 

So now we can obtain the leading order terms in the coefficients $A, B, \Gamma, \Delta, E$ of \cref{eq:detK0coeffs} and hence in the coefficients $C_i$ of \cref{eq:Psiquartic}. After a lengthy calculation we find, to  leading order, that \cref{eq:Psiquartic} becomes
\begin{align}
\label{eq:Psiquarticeps}
\sum_{i=0}^4C^{\varepsilon}_i\Psi^i=0,
\end{align}
where
\begin{align}
C^{\varepsilon}_4 & = (1+\alpha)^2[2(2+\alpha)\hat{\mu}-\hat{\varphi}^2], \nonumber \\
C^{\varepsilon}_3 & = 0, \nonumber \\
C^{\varepsilon}_2 & = -(1+\alpha)^{\frac{3}{2}}(2+\alpha)^{\frac{3}{2}}[4(2+\alpha)\hat{\mu}-(3+\alpha)\hat{\varphi}^2], \nonumber \\
C^{\varepsilon}_1 & = 0, \nonumber \\
C^{\varepsilon}_0 & = (1+\alpha)(2+\alpha)^4[2\hat{\mu}-\hat{\varphi}^2].
\end{align}
In \cref{fig:expan} we show agreement between \cref{eq:Psiquartic} and \cref{eq:Psiquarticeps} for $\varepsilon^2\hat{\mu} =0.01$ when $\alpha=3$, corresponding to mechanism II, \cref{fig:mechanisms}. There are several points to note. From the group symmetry property \cref{eq:group}, we plot results for $\cos\varphi>0$ only. For $|\Psi|<\Psi_\mathrm{L}=1.6119$, the perturbation analysis \cref{eq:Psiquarticeps} in blue is in good agreement with the exact results \cref{eq:Psiquartic} in red and yellow, for this type \circled{3} behaviour\footnote{From \cref{eq:Psiquarticeps}, when $\Psi=0$, we have simply $C^{\varepsilon}_0=0 \implies \hat{\mu}=\frac{1}{2}\hat{\varphi}^2$ and so $\varepsilon \hat{\varphi}\approx\pm0.1414\approx\cos \varphi$, giving symmetrically placed bifurcations around $\varphi =-\frac{\pi}{2}$.}.

Since here $\mu=1.01\mu_\mathrm{P}\approx1.3467$, we have that $\Psi_1=2.1526$, $\Psi_2=2.6220$. In \cref{fig:expan}, for $\Psi_\mathrm{L}<|\Psi|<\Psi_1$, we have type \circled{4} behaviour. The positions of the bifurcations get closer to $\varphi =-\frac{\pi}{2}$, as the topology of the GB manifold is about to change. Here the agreement between \cref{eq:Psiquarticeps} and \cref{eq:Psiquartic} is not so good. For $\Psi_1<|\Psi|<\Psi_2$, agreement is poor, because the topology of the GB manifold has completely changed.


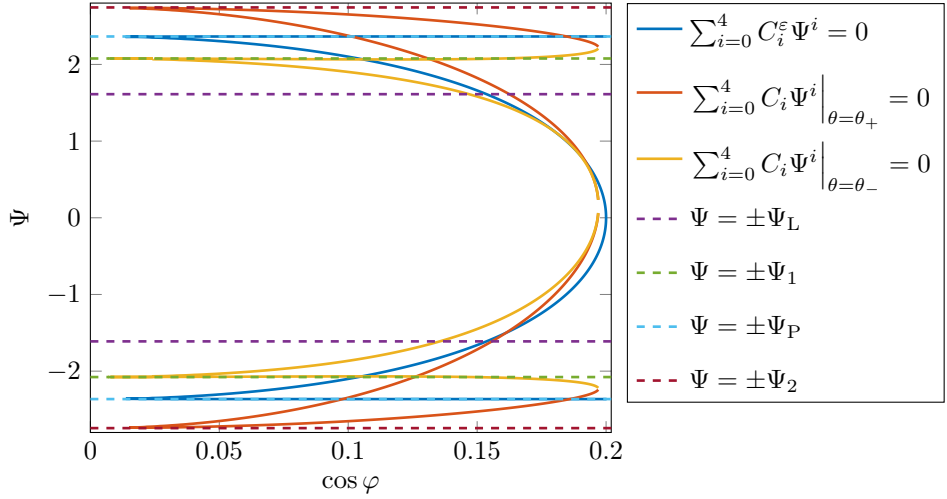
\begin{figure}
\centering
\pgfplotstableset{col sep=comma}
\pgfplotsset{scaled x ticks=false}
\begin{tikzpicture}
\begin{axis}[
    xlabel={$\cos{\varphi}$},
    ylabel={$\Psi$},
    xmin=0, xmax=0.202,
    ymin=-2.8, ymax=2.8,
    xtick={0,0.05,0.1,0.15,0.2},
    xticklabels={0,0.05,0.1,0.15,0.2},
    ytick={-2,-1,0,1,2},
    legend pos=outer north east,
    legend cell align={left},
    grid style=dashed,
    unbounded coords=jump,
    legend style={row sep=8pt},
]
\addplot[color=md1,line width=1pt] table{Figs/asymptotics/data/C.txt};
\addplot[color=md2,line width=1pt] table{Figs/asymptotics/data/Cepsp.txt};
\addplot[color=md3,line width=1pt] table{Figs/asymptotics/data/Cepsm.txt};
\addplot[dashed,color=md4,line width=1pt] table{Figs/asymptotics/data/Psi_l.txt};
\addplot[dashed,color=md5,line width=1pt] table{Figs/asymptotics/data/Psi_1.txt};
\addplot[dashed,color=md6,line width=1pt] table{Figs/asymptotics/data/Psi_P.txt};
\addplot[dashed,color=md7,line width=1pt] table{Figs/asymptotics/data/Psi_2.txt};
\legend{$\sum_{i=0}^4 C_i^\varepsilon\Psi^i=0$,
        $\left.\sum_{i=0}^4 C_i\Psi^i\right|_{\theta=\theta_+}=0$,
        $\left.\sum_{i=0}^4 C_i\Psi^i\right|_{\theta=\theta_-}=0$,
        $\Psi=\pm\Psi_\mathrm{L}$,$\Psi=\pm\Psi_1$,$\Psi=\pm\Psi_\mathrm{P}$,$\Psi=\pm\Psi_2$}
\end{axis}
\end{tikzpicture}
\caption{Comparison of $\sum_{i=0}^4 C_i\Psi^i=0$ \cref{eq:Psiquartic} and $\sum_{i=0}^4 C_i^\varepsilon\Psi^i=0$ \cref{eq:Psiquarticeps} for $\varepsilon^2\hat{\mu} =0.02$, $\eta=1$, and $\alpha=3$.}
    \label{fig:expan}
\end{figure}

\section{Putting it all together}
\label{sec:alltogether}
In \cref{subsec:roleofPsi}, we saw the importance of $\Psi$ in the topology of the GB manifold (to the extent that it does not exist for $|\Psi| > \Psi_2$). In \cref{sec:slip}, we saw how the dynamics can vary on the GB manifold. Here, we put both these effects together to provide a full picture of the geometry of the $3D$ Painlev\'e paradox.

In \cref{fig:stabilty_on_GB}, we see the effect the increase in $\Psi$ has on the dynamics when $\alpha=3$ for two values of $\mu$. The position of the zeros of $\det\mat{K}$ around the GB manifold is a strong function of $\Psi$.
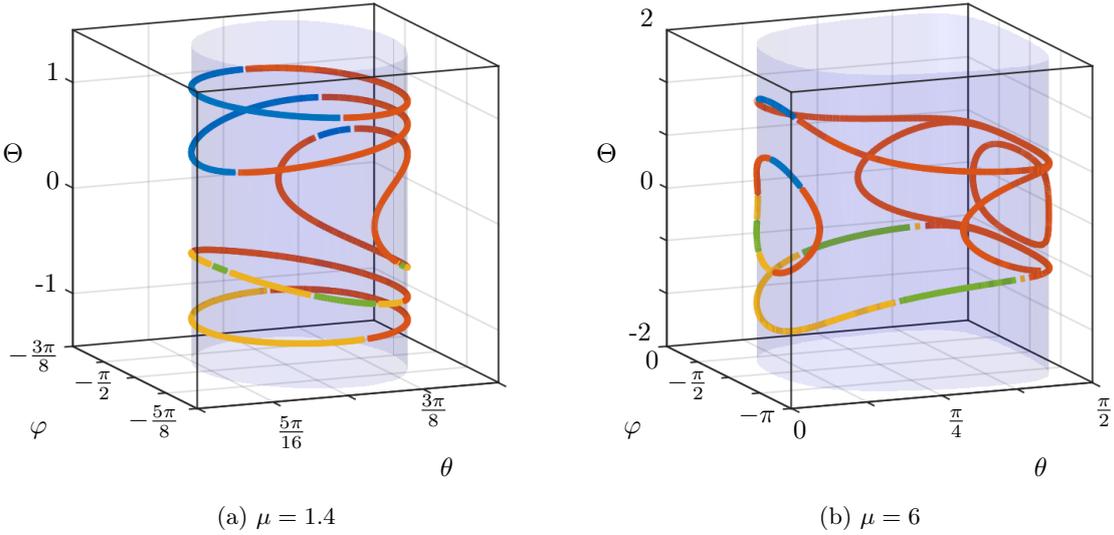
\begin{figure}[htbp]
    \centering
    \begin{subfigure}{0.45\textwidth}
\begin{overpic}[width=\textwidth]{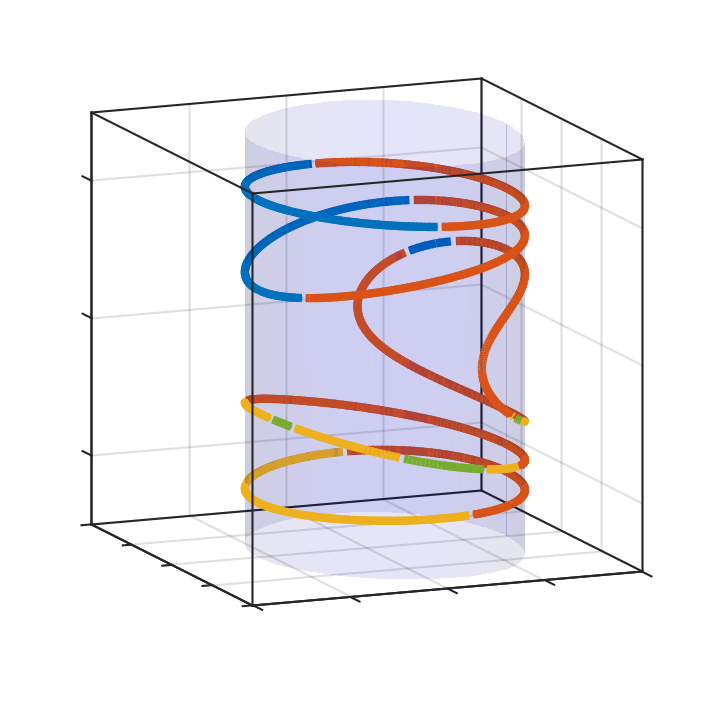}
\put(8,55){0}
\put(8,75){1}
\put(6,36){-1}
\put(0,60){$\Theta$}
\put(1,23){$-\frac{3\pi}{8}$}
\put(13,18){$-\frac{\pi}{2}$}
\put(23,11){$-\frac{5\pi}{8}$}
\put(5,10){$\varphi$}
\put(76,13){$\frac{3\pi}{8}$}
\put(50,9){$\frac{5\pi}{16}$}
\put(80,2){$\theta$}
\end{overpic}
\caption{$\mu=1.4$}
\end{subfigure}
\hfil
\begin{subfigure}{0.45\textwidth}
\begin{overpic}[width=\textwidth]{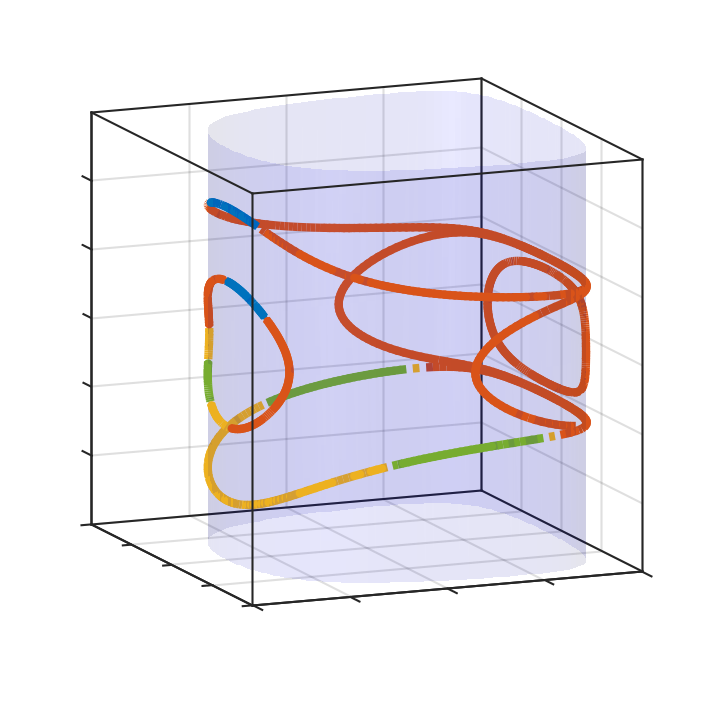}
\put(8,55){0}
\put(8,85){2}
\put(6,27){-2}
\put(0,60){$\Theta$}
\put(9,22){$0$}
\put(13,18){$-\frac{\pi}{2}$}
\put(26,12){$-\pi$}
\put(5,10){$\varphi$}
\put(36,9){$0$}
\put(64,11){$\frac{\pi}{4}$}
\put(91,13){$\frac{\pi}{2}$}
\put(80,2){$\theta$}
\end{overpic}
\caption{$\mu=6$}
\end{subfigure}
    \caption{Topology and stability of the GB manifold(s) for varying $\Psi$ when $\alpha=3$. Stable node {\color{md1}{$\blacksquare$}}, saddle {\color{md2}{$\blacksquare$}}, unstable node {\color{md3}{$\blacksquare$}}, and unstable focus {\color{md5}{$\blacksquare$}}. GB manifolds are shown for $\Psi=0$, $\Psi=(\Psi_\mathrm{L}+\Psi_1)/2$ and $\Psi=(\Psi_1+\Psi_2)/2$.}
    \label{fig:stabilty_on_GB}
\end{figure}
We demonstrate this asymmetry in another way in \cref{fig:determinant_bif}, for $\alpha=3$, $\mu=1.7$ and $\eta=10$, where $\Psi_\mathrm{P}=2.3644$ and $\Psi_\mathrm{L}=1.6119$.


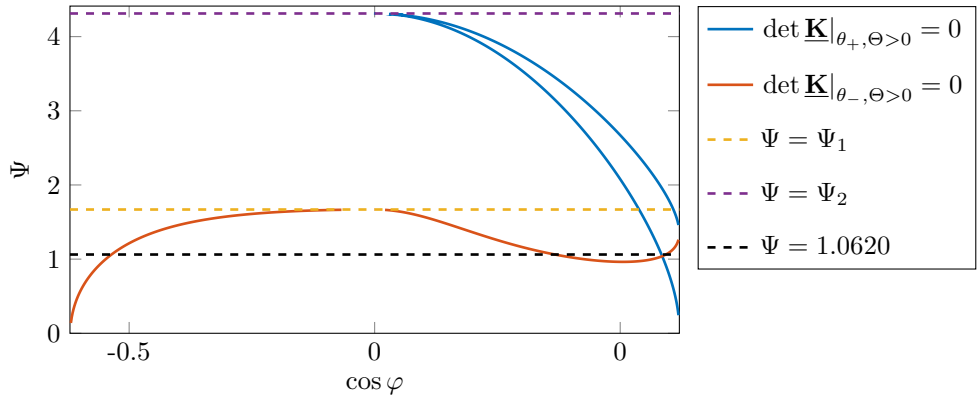
\begin{figure}
\centering
\pgfplotstableset{col sep=comma}
\pgfplotsset{scaled x ticks=false}
\pgfplotsset{width=0.6\textwidth,height=0.37\textwidth}
\begin{tikzpicture}
\begin{axis}[
    xlabel={$\cos{\varphi}$},
    ylabel={$\Psi$},
    xmin=-0.6205, xmax=0.6205,
    ymin=0, ymax=4.4121,
    xtick={-0.5,0,0.5},
    xticklabels={-0.5,0,0,5},
    ytick={0,1,2,3,4},
    legend pos=outer north east,
    legend cell align={left},
    grid style=dashed,
    unbounded coords=jump,
    legend style={row sep=8pt},
]
\addplot[color=md1,line width=1pt] table{Figs/notasymptotics/data/detp.txt};
\addplot[color=md2,line width=1pt] table{Figs/notasymptotics/data/detm.txt};
\addplot[dashed,color=md3,line width=1pt] table{Figs/notasymptotics/data/Psi1.txt};
\addplot[dashed,color=md4,line width=1pt] table{Figs/notasymptotics/data/Psi2.txt};
\addplot[dashed,color=black,line width=1pt]
coordinates{(-0.6205,1.0620)(0.6205,1.0620)};
\legend{$\left.\det\mat{K}\right|_{\theta_{+},\Theta>0}=0$,
        $\left.\det\mat{K}\right|_{\theta_{-},\Theta>0}=0$,
        $\Psi=\Psi_1$,$\Psi=\Psi_2$,$\Psi=1.0620$}
\end{axis}
\end{tikzpicture}
    \caption{Positions of the zeros of $\det\mat{K}$ on the GB manifold, for $\alpha=3$, $\mu=1.7$ and $\eta=10$. The section $\Psi=1.0620$ is shown in \cref{fig:pandora}. Here $\Psi_1=1.6689$ and $\Psi_2=4.3121$}
    \label{fig:determinant_bif}
\end{figure}

In \cref{fig:pandora} we show the eigenvectors around the GB manifold when $\alpha=3$, $\mu=1.7$ and $\eta=10$ for the case $\Psi=1.0620$ from \cref{fig:determinant_bif}. There are four critical points, labelled (d), (h), (j) and (l). Since $\Psi \ne 0$, there is no symmetry about $\varphi=-\pi/2$ and hence we show eigenvectors for $\varphi\gtrless-\pi/2$. In every case, except (b) and (f) where $\varphi=-\frac{\pi}{2}$, trajectories in the slipping region (in yellow) can enter the inconsistent region (in lime). \cref{fig:evsc,fig:evse,fig:evsg,fig:evsi,fig:evsk,fig:evsm}, the phase portraits away from bifurcations and $\varphi=-\pi/2$, agree with \cref{thm:GB}.

\begin{figure}[htbp]
    \centering
    \begin{subfigure}{0.45\textwidth}
    \centering
    \begin{overpic}[width=0.99\textwidth]{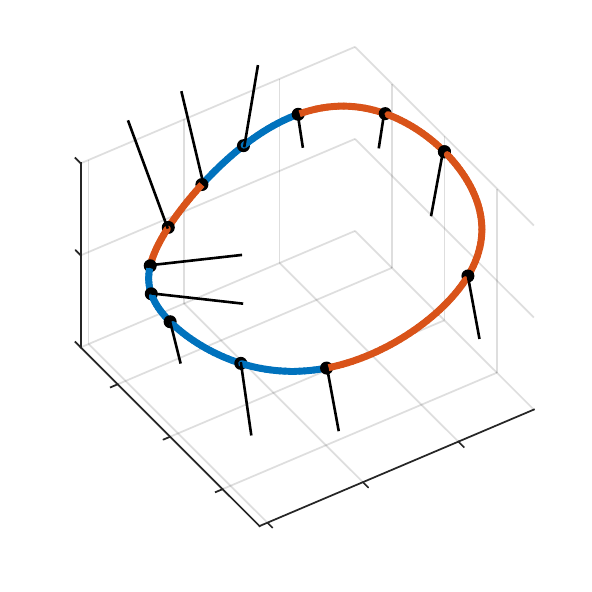}
    \put(10,10){$\varphi$}
    \put(7,33){$-\frac{3}{8}\pi$}
    \put(16,24){$-\frac{1}{2}\pi$}
    \put(25,15){$-\frac{5}{8}\pi$}
    \put(83,8){$\theta$}
    \put(45,4){$\frac{1}{4}\pi$}
    \put(60,11){$\frac{5}{16}\pi$}
    \put(77,17){$\frac{3}{8}\pi$}
    \put(0,65){$\Theta$}
    \put(5.5,43){$0.8$}
    \put(5.5,59){$0.9$}
    \put(9,74){$1$}
    \put(28,34){(b)}
    \put(40,21){(c)}
    \put(54,22){(d)}
    \put(79,37){(e)}
    \put(70,58){(f)}
    \put(61,70){(g)}
    \put(49,70){(h)}
    \put(41,91){(i)}
    \put(28,87){(j)}
    \put(18,81){(k)}
    \put(41.5,56){(l)}
    \put(42,46){(m)}
    \end{overpic}
    \caption{}
    \end{subfigure}
    \begin{minipage}{0.5\textwidth}
    \begin{subfigure}{0.48\textwidth}
    \centering
    \begin{overpic}[width=0.99\textwidth]{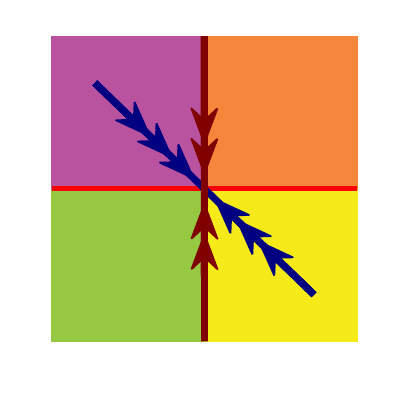}
            \put(50,0){$p$}
        \put(25,0){$-$}
        \put(75,0){$+$}
                \put(0,50){$b$}
        \put(0,25){$-$}
        \put(0,75){$+$}
    \end{overpic}
    \caption{}\label{fig:evsb}
    \end{subfigure}
        \begin{subfigure}{0.48\textwidth}
    \centering
    \begin{overpic}[width=0.99\textwidth]{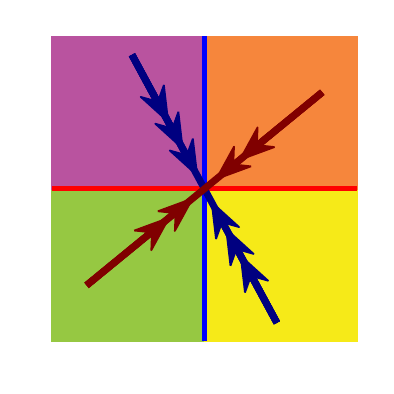}
            \put(50,0){$p$}
        \put(25,0){$-$}
        \put(75,0){$+$}
                \put(0,50){$b$}
        \put(0,25){$-$}
        \put(0,75){$+$}
    \end{overpic}
    \caption{}\label{fig:evsc}
    \end{subfigure}\\
    \begin{subfigure}{0.48\textwidth}
    \centering
    \begin{overpic}[width=0.99\textwidth]{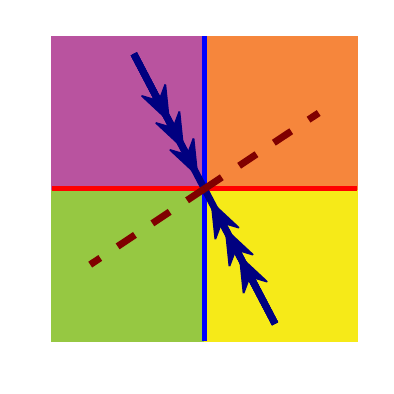}
            \put(50,0){$p$}
        \put(25,0){$-$}
        \put(75,0){$+$}
                \put(0,50){$b$}
        \put(0,25){$-$}
        \put(0,75){$+$}
    \end{overpic}
    \caption{}\label{fig:evsd}
    \end{subfigure}
        \begin{subfigure}{0.48\textwidth}
    \centering
    \begin{overpic}[width=0.99\textwidth]{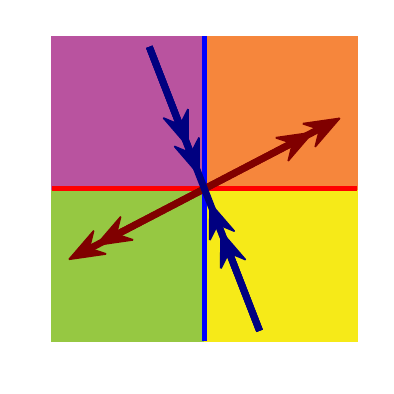}
            \put(50,0){$p$}
        \put(25,0){$-$}
        \put(75,0){$+$}
                \put(0,50){$b$}
        \put(0,25){$-$}
        \put(0,75){$+$}
    \end{overpic}
    \caption{}\label{fig:evse}
    \end{subfigure}
    \end{minipage}
    
            \begin{subfigure}{0.24\textwidth}
    \centering
    \begin{overpic}[width=0.99\textwidth]{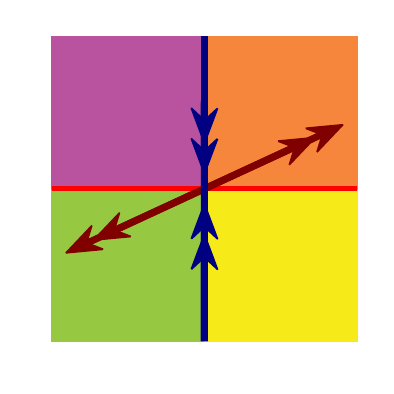}
            \put(50,0){$p$}
        \put(25,0){$-$}
        \put(75,0){$+$}
                \put(0,50){$b$}
        \put(0,25){$-$}
        \put(0,75){$+$}
    \end{overpic}
    \caption{}\label{fig:evsf}
    \end{subfigure}
                \begin{subfigure}{0.24\textwidth}
    \centering
    \begin{overpic}[width=0.99\textwidth]{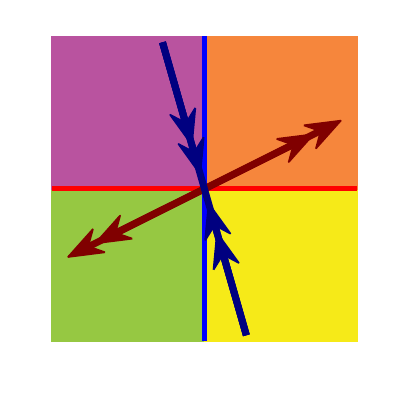}
            \put(50,0){$p$}
        \put(25,0){$-$}
        \put(75,0){$+$}
                \put(0,50){$b$}
        \put(0,25){$-$}
        \put(0,75){$+$}
    \end{overpic}
    \caption{}\label{fig:evsg}
    \end{subfigure}
                    \begin{subfigure}{0.24\textwidth}
    \centering
    \begin{overpic}[width=0.99\textwidth]{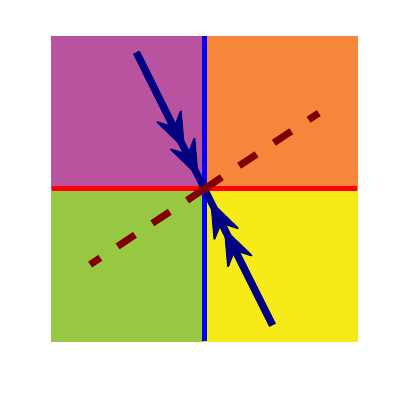}
            \put(50,0){$p$}
        \put(25,0){$-$}
        \put(75,0){$+$}
                \put(0,50){$b$}
        \put(0,25){$-$}
        \put(0,75){$+$}
    \end{overpic}
    \caption{}\label{fig:evsh}
    \end{subfigure}
    \begin{subfigure}{0.24\textwidth}
    \centering
    \begin{overpic}[width=0.99\textwidth]{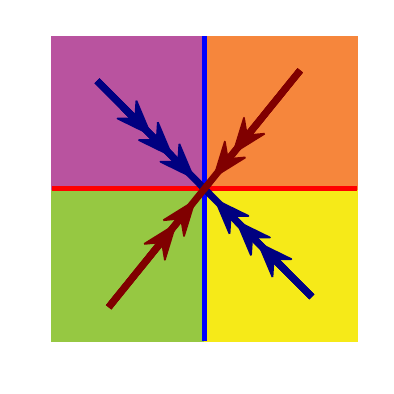}
            \put(50,0){$p$}
        \put(25,0){$-$}
        \put(75,0){$+$}
                \put(0,50){$b$}
        \put(0,25){$-$}
        \put(0,75){$+$}
    \end{overpic}
    \caption{}\label{fig:evsi}
    \end{subfigure}
    \begin{subfigure}{0.24\textwidth}
    \centering
    \begin{overpic}[width=0.99\textwidth]{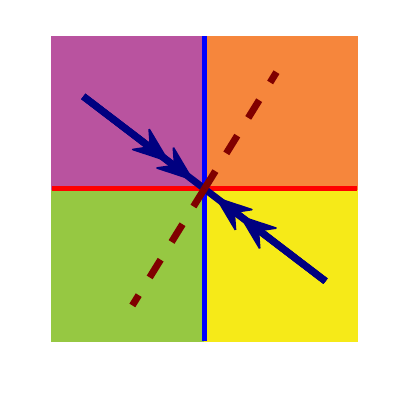}
            \put(50,0){$p$}
        \put(25,0){$-$}
        \put(75,0){$+$}
                \put(0,50){$b$}
        \put(0,25){$-$}
        \put(0,75){$+$}
    \end{overpic}
    \caption{}\label{fig:evsj}
    \end{subfigure}
    \begin{subfigure}{0.24\textwidth}
    \centering
    \begin{overpic}[width=0.99\textwidth]{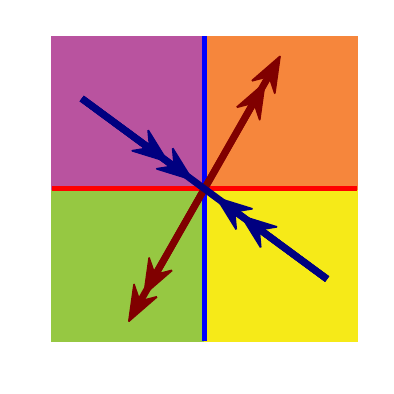}
            \put(50,0){$p$}
        \put(25,0){$-$}
        \put(75,0){$+$}
                \put(0,50){$b$}
        \put(0,25){$-$}
        \put(0,75){$+$}
    \end{overpic}
    \caption{}\label{fig:evsk}
    \end{subfigure}
    \begin{subfigure}{0.24\textwidth}
    \centering
    \begin{overpic}[width=0.99\textwidth]{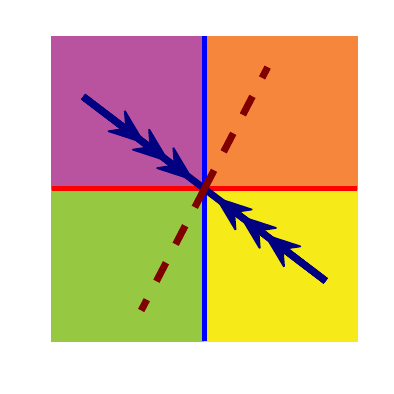}
            \put(50,0){$p$}
        \put(25,0){$-$}
        \put(75,0){$+$}
                \put(0,50){$b$}
        \put(0,25){$-$}
        \put(0,75){$+$}
    \end{overpic}
    \caption{}\label{fig:evsl}
    \end{subfigure}
    \begin{subfigure}{0.24\textwidth}
    \centering
    \begin{overpic}[width=0.99\textwidth]{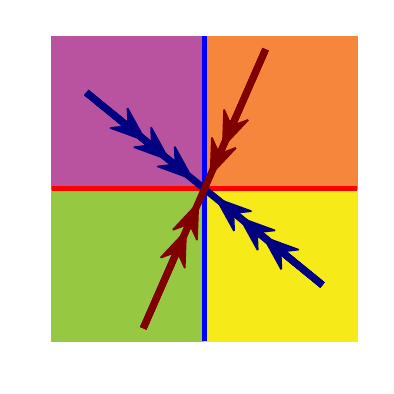}
            \put(50,0){$p$}
        \put(25,0){$-$}
        \put(75,0){$+$}
                \put(0,50){$b$}
        \put(0,25){$-$}
        \put(0,75){$+$}
    \end{overpic}
    \caption{}\label{fig:evsm}
    \end{subfigure}
    \caption{Eigenvectors along the GB manifold when $\alpha=3$, $\mu=1.7$ and $\eta=10$ for $\Psi=1.062$, with four critical points: (d), (h), (j) and (l). Since $\Psi \ne 0$, there is no symmetry about $\varphi=-\pi/2$ and hence we show eigenvectors for $\varphi\gtrless-\pi/2$.}
    \label{fig:pandora}
\end{figure}
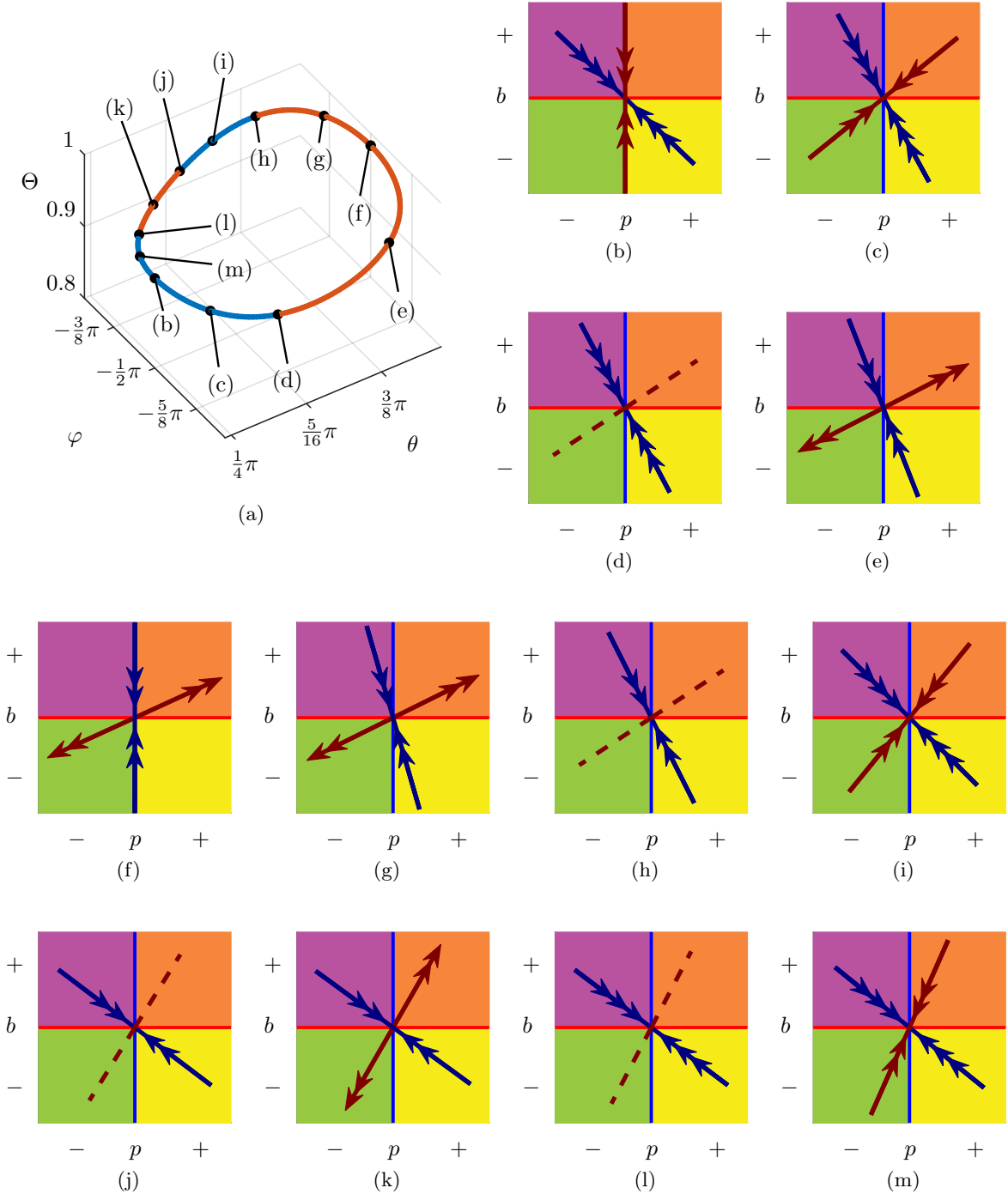

Using numerical methods, Champneys and V\'{a}rkonyi \cite{Varkonyi18} found that trajectories are able to reach $p=0$ away from $b=0$ in for the $3D$ problem. From their results \cite[Figure 17]{Varkonyi18}, it is evident that a separatrix exists between slipping trajectories that reach $p=0$ and those that do not. In \cref{fig:my_ugly_simulations}, we see that this separatrix corresponds to the eigenvector associated with the leading eigenvalue of the linearisation of the system about the GB manifold\footnote{V\'arkonyi, P. 2021 \textit{Personal Communication} has pointed out that it was difficult to obtain the numerical results in \cite[Figure 17]{Varkonyi18}. We can now see that this is because the small eigenframe rotation around the GB manifold makes it difficult to find initial conditions between the eigenvector and $p=0$.}.
\begin{figure}
\centering
\begin{subfigure}[b]{0.45\textwidth}
\pgfplotsset{width=0.8\textwidth,height=0.8\textwidth}
\begin{tikzpicture}
\begin{axis}[
enlargelimits=false,
axis on top,,
xlabel={$p$},ylabel={$b$},
y label style={rotate=-90}
]
\addplot graphics [
xmin=-0.02,xmax=0.02,
ymin=-0.25,ymax=0.25
] {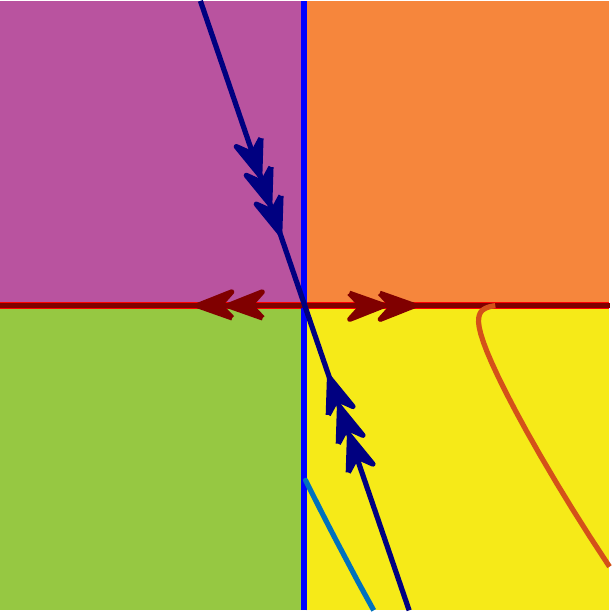};
\end{axis}
\end{tikzpicture}
\caption{}
\end{subfigure}
\hfil
\begin{subfigure}[b]{0.5\textwidth}
\begin{overpic}[width=\textwidth]{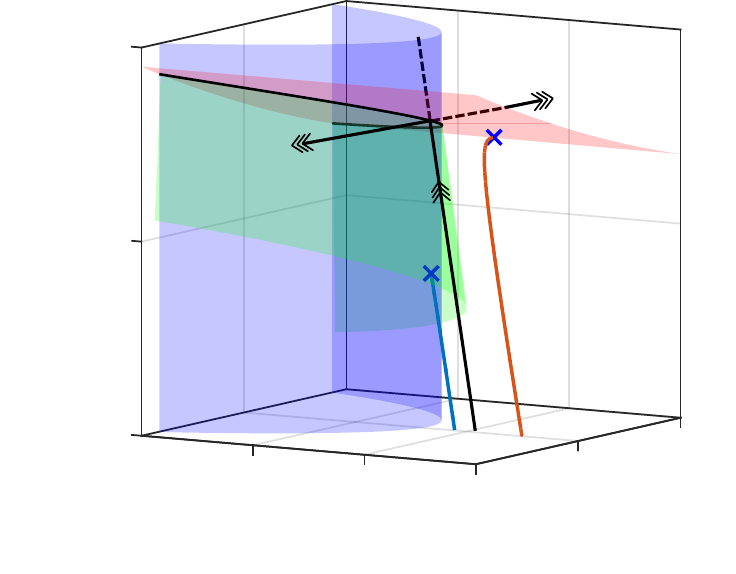}
\put(10,17){0.9}
\put(10,45){\phantom{0.}1}
\put(10,70){1.1}
\put(2,45){$\Theta$}
\put(28,3){$\varphi$}
\put(28,12){-1.85}
\put(45,10){-1.9}
\put(65,9){1}
\put(77,12){1.1}
\put(90,15){1.2}
\put(65,70){\textcolor{red}{$b(0,\Theta,\theta)=0$}}
{\color{red}\put(80,60){\line(-1,2){4}}}
\put(5,8){\textcolor{blue}{$p(\theta,\varphi)=0$}}
{\color{blue}\put(15,11){\line(3,2){15}}}
\put(88,5){$\theta$}
\end{overpic}
\caption{}
\end{subfigure}
\caption{Numerical integration of the equations \cref{eq:general} demonstrates that the separatrix between orbits that reach inconsistency ($p=0$) and those that lift off ($b=0$) is tangent to the eigenvector corresponding to the leading eigenvalue of the linearisation of the dynamics about each point on the GB manifold. Both solutions are for $\alpha=3$ and $\mu=1.4$. Initial conditions: {\color{md1}{$\blacksquare$}}  $\theta(0)=1.1067$, $\eta(0)=2.1307$, $\varphi(0)=-1.8913$, $\Psi(0)=-0.1027$, $\Theta(0)=0.9000$; {\color{md2}{$\blacksquare$}} $\theta(0)=1.1031$, $\eta(0)=2$, $\varphi(0)=-1.9233$, $\Psi(0)=0$, $\Theta(0)=0.9000$. The surface in green is tangent to the eigenvector corresponding to the leading eigenvalue at each point along the GB manifold.}
        \label{fig:my_ugly_simulations}
\end{figure}

\section{Conclusion}
\label{sec:conclusion}
We have studied the problem of a rigid body slipping along a rough horizontal plane with one point of contact, subject to Coulomb friction.

In $2D$ this is the celebrated Painlev\'e problem \cite{Painleve1895, Painleve1905a, Painleve1905b}, which gives rise to paradoxes when the coefficient of friction $\mu$ exceeds a critical value $\mu_\mathrm{P}(\alpha)$, which depends on the moment of inertia of the rigid body.

In $3D$ the critical value $\mu_\mathrm{P}^*(\alpha,\varphi)$ is also dependent on the relative slip angle $\varphi$. We have shown that it is possible to avoid an existing paradox by a judicious choice of $\varphi$ (see \cref{fig:badsection}).

The $3D$ problem involves motion in the azimuthal direction $\psi$. We have shown that, in the absence of motion in the polar direction $\theta$, it is possible for the body to lift off the surface, whether or not there is a paradox, when the azimuthal angular velocity $\Psi=\dot{\psi}$ is such that $|\Psi|>\Psi_\mathrm{L}$, where $\Psi_\mathrm{L}$ is given in \cref{eq:psil}, see \cref{fig:psi1.6119}.

There are two other critical values of $\Psi$; $\Psi_1$ and $\Psi_2$ . When 
\begin{equation}
|\Psi|>\begin{cases}\Psi_1,& \mu\in(\mu_\mathrm{P},\mu_\mathrm{L}]\\
\Psi_\mathrm{L},& \mu>\mu_\mathrm{L}\\\end{cases},
\end{equation}
the rigid body can undergo an indeterminate paradox even when the polar angular velocity $\Theta=\dot{\theta}=0$ (see \cref{fig:psi1.9483}). When $|\Psi|>\Psi_2 > \Psi_1$, the inconsistent paradox vanishes (see \cref{fig:psi3.0097}).

These observations lead to the conclusion that there are three different mechanisms in the $3D$ Painlev\'e problem as $\Psi$ increases, dependent on $\mu$ (see \cref{fig:mechanisms}).

When not sticking, the rigid body can exist in four modes: slipping, lift-off, inconsistent and indeterminate. Transitions from slipping are of greatest interest. In the absence of a paradox, the rigid body can either slip or lift off, depending on the sign of the free acceleration $b$, given in \cref{eq:b}. A paradox occurs whenever the quantity $p$, given in \cref{eq:p}, becomes negative.

The surfaces $b=0$ and $p=0$ intersect in the GB manifold \cite{Varkonyi18}. We analyse the dynamics close to the GB manifold in \cref{sec:painleve3D}. In $2D$, slipping trajectories cannot pass through $p=0$ without also passing through $b=0$ \cite{Genot1999}. In $3D$ slipping trajectories cross $p=0$ away from the GB manifold except when $\varphi=-\frac{\pi}{2}$, and an open set of initial conditions that reach inconsistency from slipping (\cref{cor:main}). Hence the $2D$ problem is highly singular. For $|\Psi|>\Psi_2$, the GB manifold does not exist. In this case, slipping trajectories, which only exist in restricted parts of phase space, may lift off (unless they stick).

We have discovered bifurcations in the dynamics on the GB manifold, which explain behaviour in the $2D$ problem (see \cref{fig:stabilty_on_GB}), as well as changes in the topology of the GB manifold. Finally, we have shown that there is a separatrix between orbits that reach inconsistency ($p=0$) and those that undergo lift-off  ($b=0$) tangent to the eigenvector corresponding to the smallest eigenvalue of the linearisation of the dynamics about the GB manifold for $\Theta>0$ (see \cref{thm:GB} and \cref{fig:ppsThetapos,fig:ppsThetaneg,fig:my_ugly_simulations}).

The discovery of the Painlev\'e paradoxes caused controversy because it implied that rigid body theory and Coulomb friction could be incompatible. Lecornu \cite{Lecornu1905b} suggested that in order to escape an inconsistent Painlev\'e paradox, there should be a jump in velocity (``\textit{l'arcboutement dynamique}"). Subsequently this became known as \textit{impact without collision} or \textit{tangential impact}, which was incorporated into rigid body theory \cite{Darboux1880,Keller1986}. 

But real progress was made by taking some elasticity into account (\emph{contact regularisation}) \cite{McClamroch1989, Dupont1997, Brogliato1999, Stewart2000, Zhao2015, Champneys2016, BlumenthalsBrogliatoBertails2016, Hogan2017, Kristiansen2018}. Here the rod is assumed to be rigid, but the rough surface is taken to be elastic. Now the friction-induced couple at the moving rod tip drives the {\emph{surface}} down until the rod, still in contact with the surface, stops sliding. With the couple no longer acting, the surface rebounds and the rod lifts off. That work was done for the $2D$ problem. Its extension to the $3D$ problem we leave to further work.


\appendix
\section{Coefficients \texorpdfstring{$Q_i, A_i, d_i, c_i \; (i=1,2)$}{Q1...}}
\label{sec:appA}
Expressions for the coefficients $Q_i, A_i, d_i, c_i \; (i=1,2)$ in \cref{eq:polareqs_varphi}, in terms of the \textit{relative slip angle} $\varphi=\beta-\psi$:
\begin{equation}\label{eq:polareqs_varphi_terms}
\begin{aligned}
Q_1(\theta,\varphi ; \mu, \alpha)&= \alpha \cos \theta \sin \varphi \left( \mu \cos \theta \sin \varphi -\sin \theta\right)- (1+\alpha)\mu \\
Q_2(\theta,\varphi  ; \mu, \alpha)&= \alpha \cos \theta \cos \varphi \left( \mu \cos \theta \sin \varphi -\sin \theta\right)\\
A_1(\Psi,\Theta,\theta,\varphi)&= \, \left(\Psi ^{2}  \cos^{2}  \theta   + \Theta ^{2}\right) \cos\theta  \sin \varphi\\
A_2(\Psi,\Theta,\theta,\varphi)&= \, \left(\Psi ^{2}  \cos^{2}  \theta   + \Theta ^{2}\right) \cos\theta  \cos \varphi\\
d_1(\theta,\varphi; \mu, \alpha)&=-\alpha \mu \sec{\theta}\cos{\varphi}\\
d_2(\theta,\varphi; \mu, \alpha)&=-\alpha \mu \sin{\theta}\sin{}\varphi -\alpha \cos \theta \\
c_1(\Psi,\Theta,\theta)&=2 \Psi\Theta\tan\theta\\
c_2(\Psi,\Theta,\theta)&=- \Psi^2\sin\theta\cos\theta.
\end{aligned}
\end{equation}

\section{Bifurcations of the geometry}\label{sec:geobifs}
In \cref{fig:mechbifs}, we demonstrate the bifurcations in the geometry of the surfaces $p=0$ and $b=0$ for fixed $\Psi$ with changes in the variable $\Psi$ and the parameter $\mu$. 
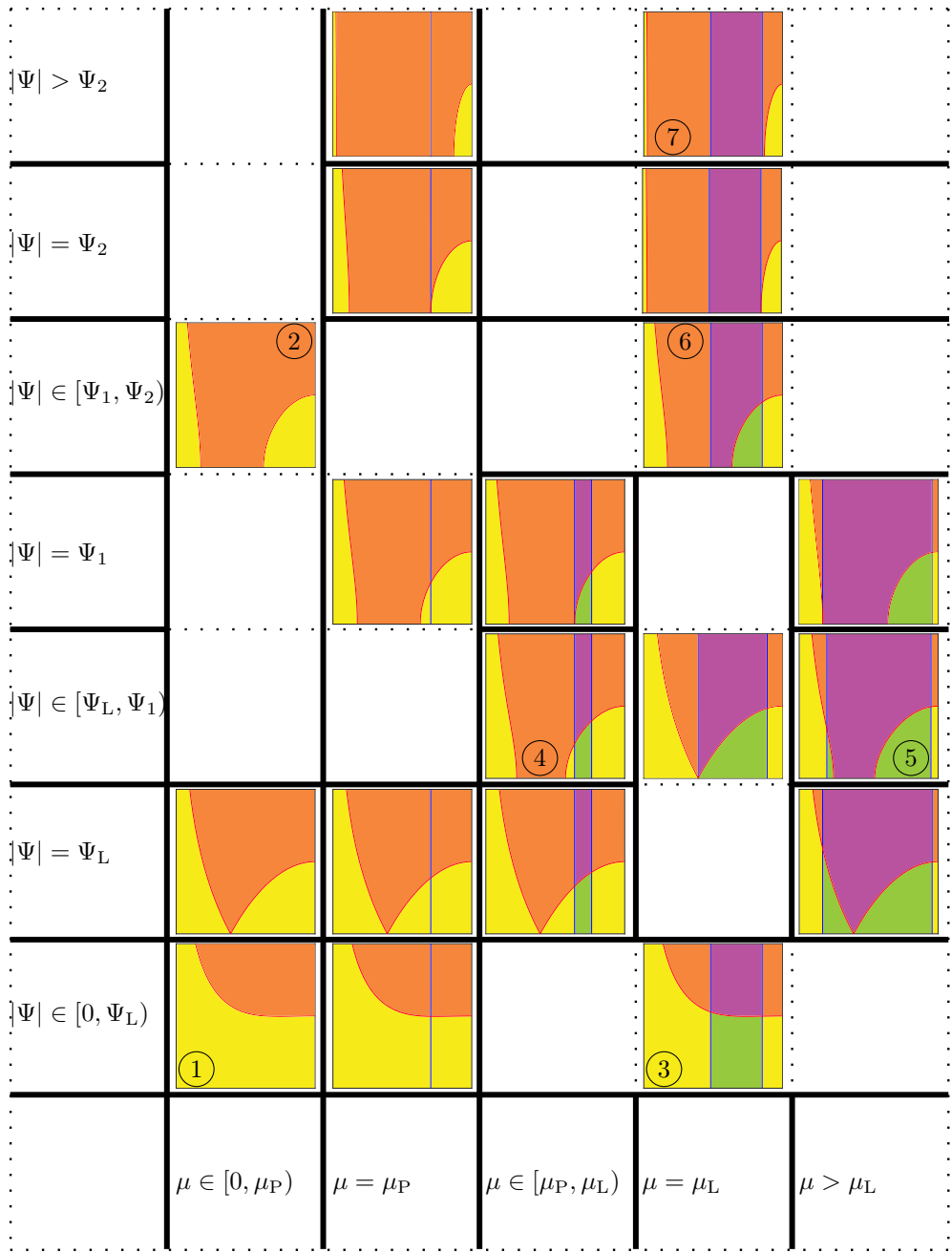
\begin{figure}
    \centering
    \begin{overpic}[width=0.81\textwidth]{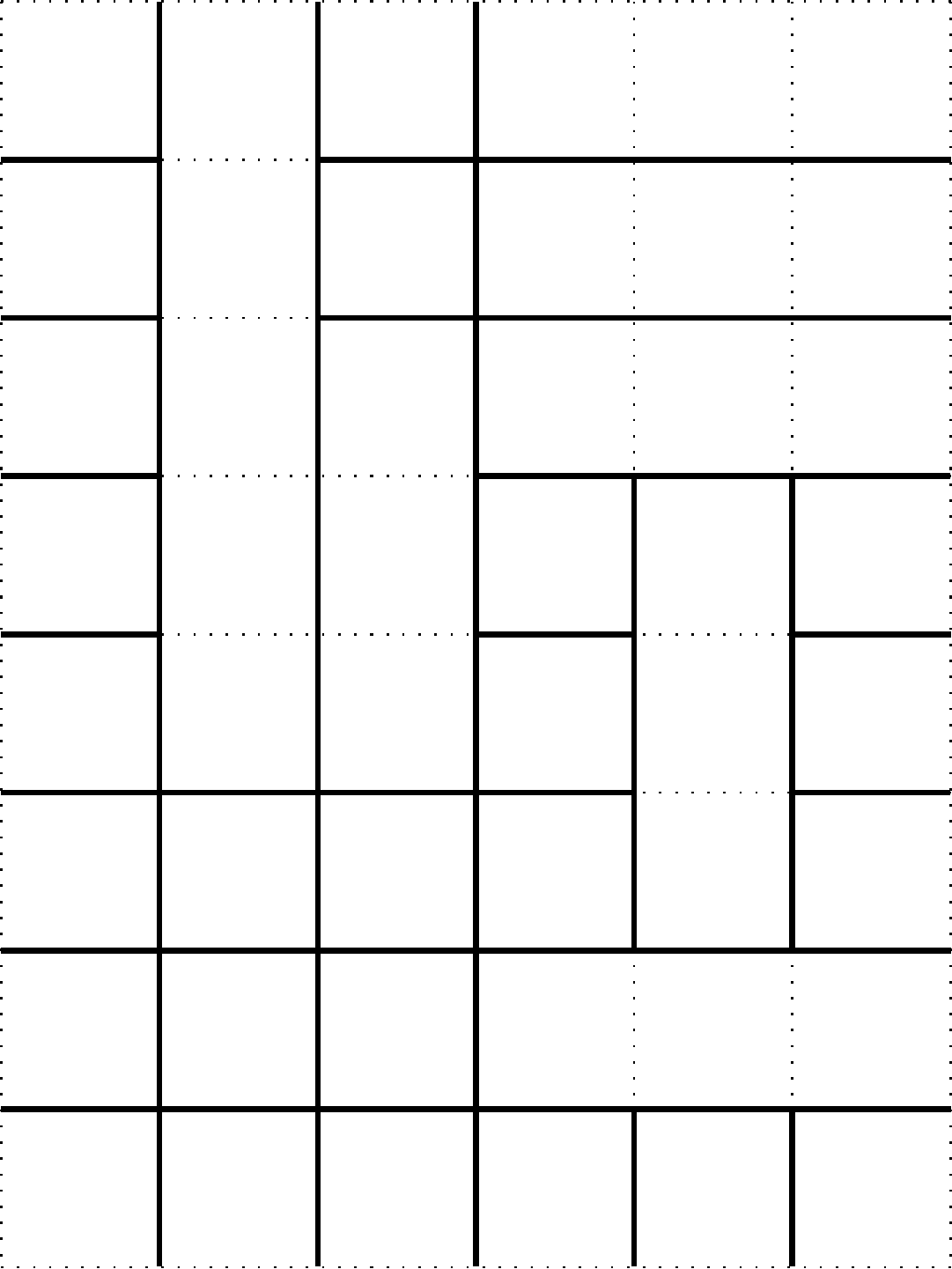}
    \put(13.3,13.1){\includegraphics[width=0.12\textwidth]{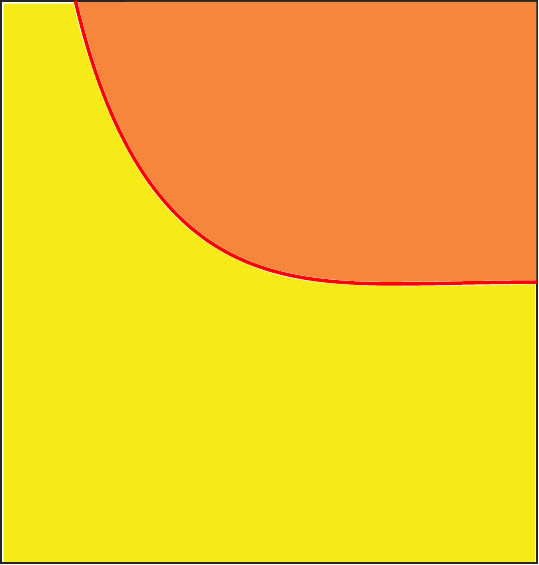}}
    \put(13.3,63){\includegraphics[width=0.12\textwidth]{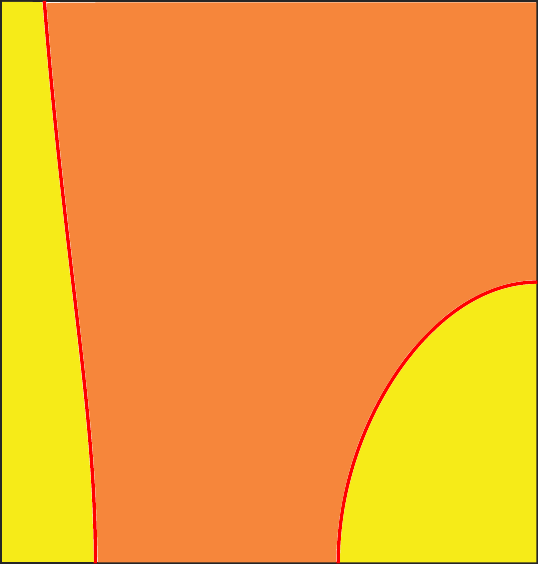}}
    \put(50.6,13.1){\includegraphics[width=0.12\textwidth]{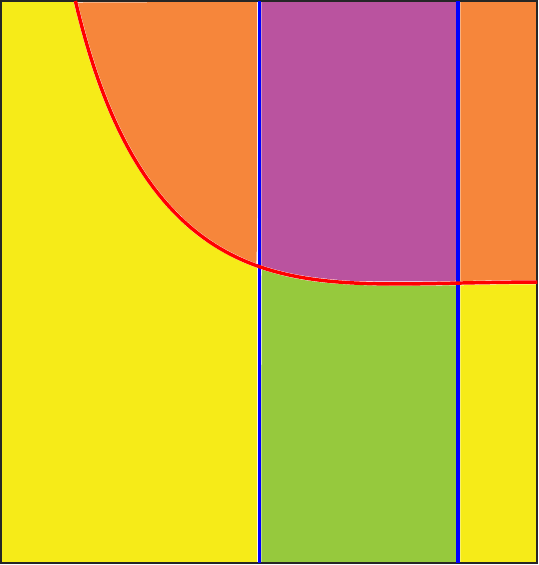}}
    \put(38,38){\includegraphics[width=0.12\textwidth]{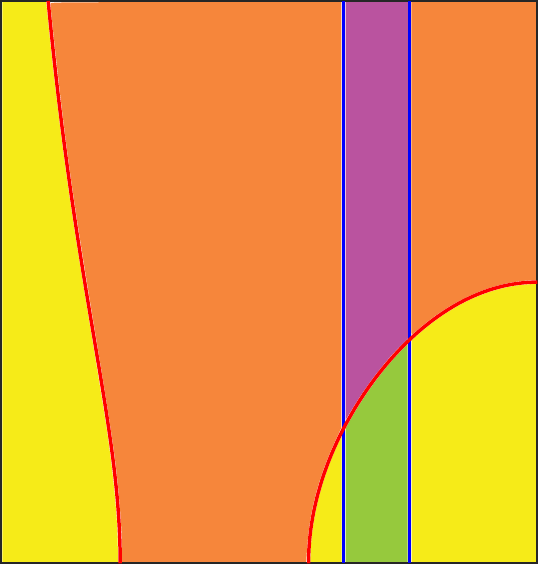}}
    \put(50.6,63){\includegraphics[width=0.12\textwidth]{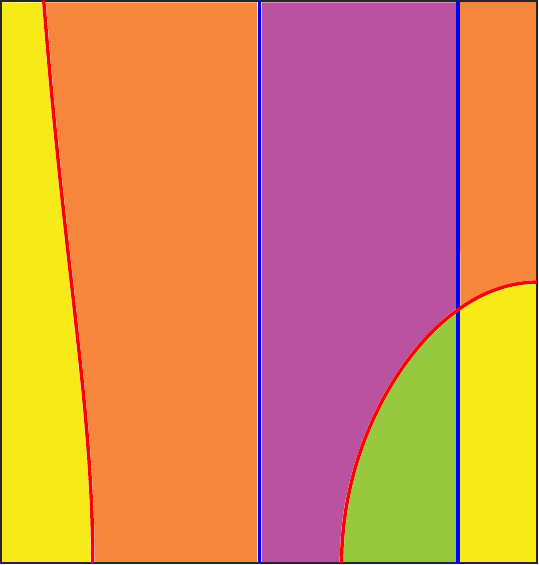}}
    \put(50.6,88){\includegraphics[width=0.12\textwidth]{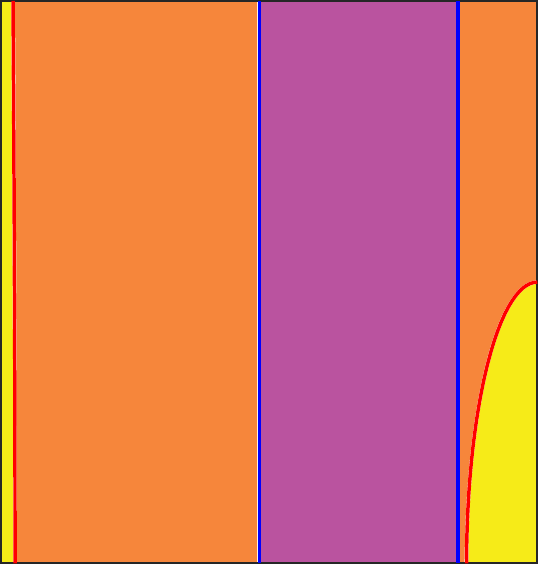}}
    \put(63,38){\includegraphics[width=0.12\textwidth]{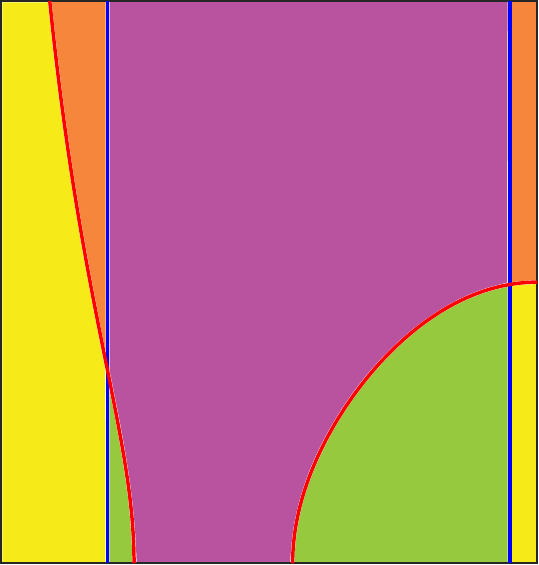}}
    \put(13.3,25.5){\includegraphics[width=0.12\textwidth]{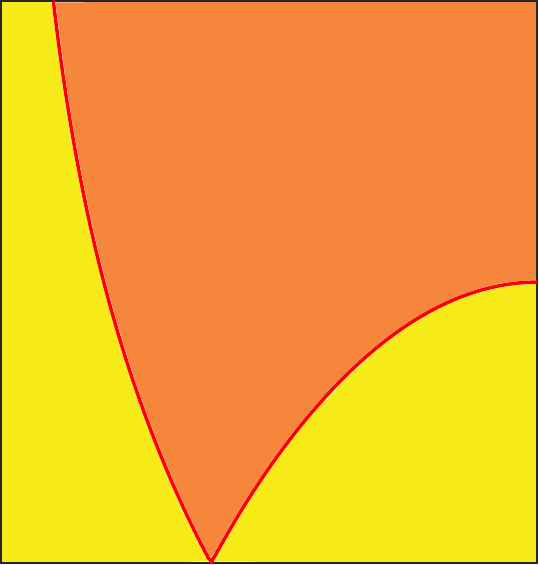}}
    \put(25.8,13.1){\includegraphics[width=0.12\textwidth]{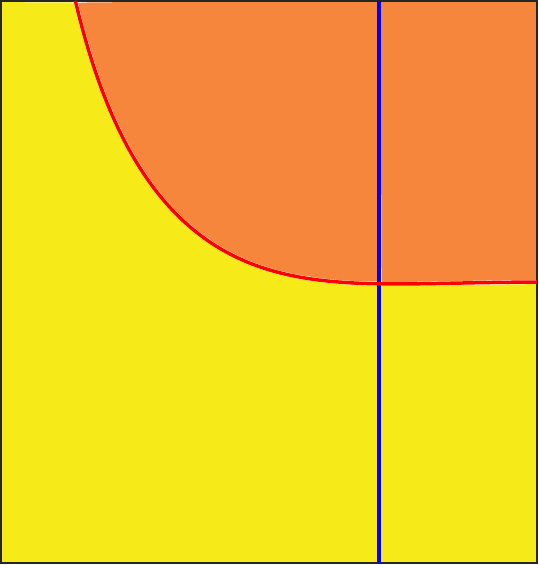}}
    \put(25.8,25.5){\includegraphics[width=0.12\textwidth]{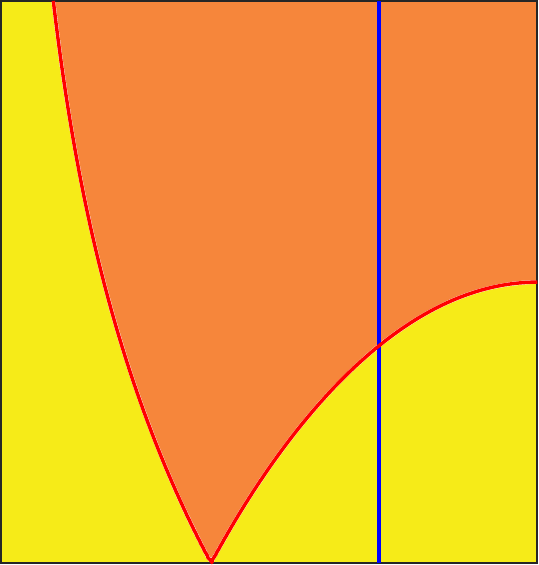}}
    \put(25.8,50.4){\includegraphics[width=0.12\textwidth]{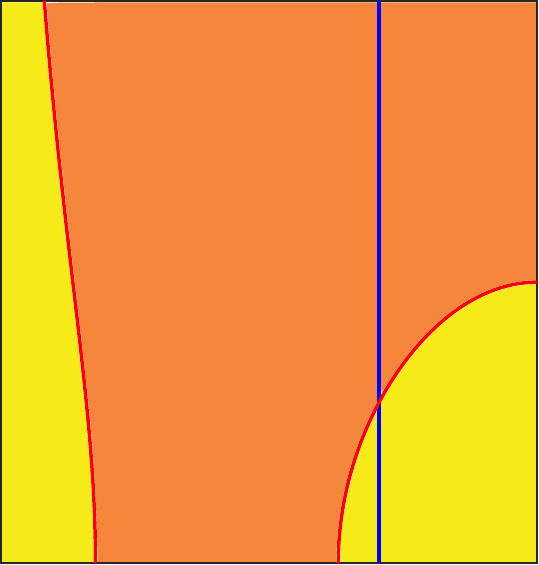}}
    \put(25.8,75.4){\includegraphics[width=0.12\textwidth]{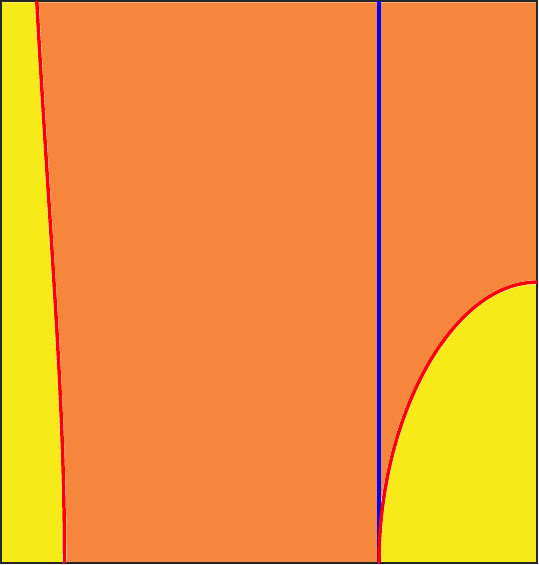}}
    \put(25.8,88){\includegraphics[width=0.12\textwidth]{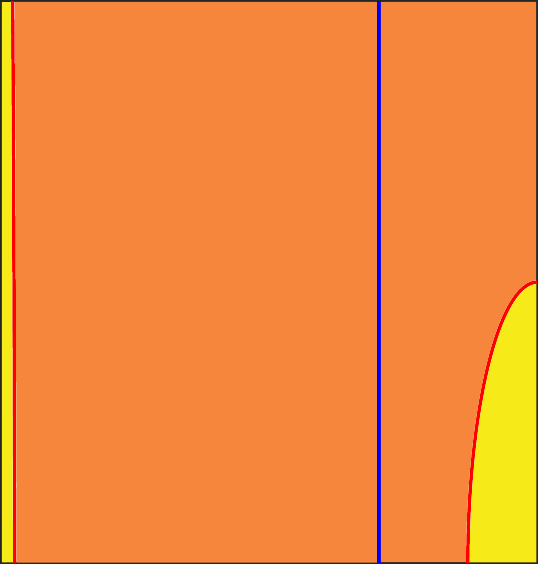}}
    \put(38,25.5){\includegraphics[width=0.12\textwidth]{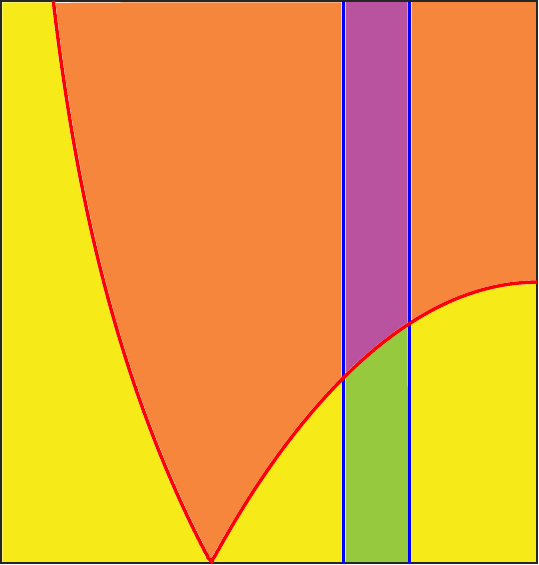}}
    \put(38,50.4){\includegraphics[width=0.12\textwidth]{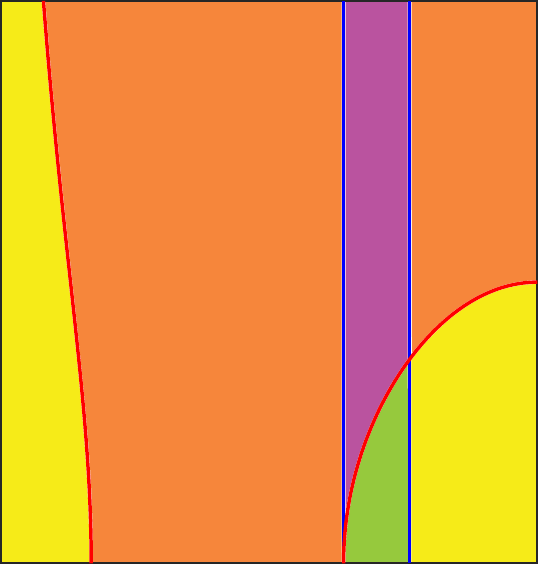}}
    \put(50.5,75.4){\includegraphics[width=0.12\textwidth]{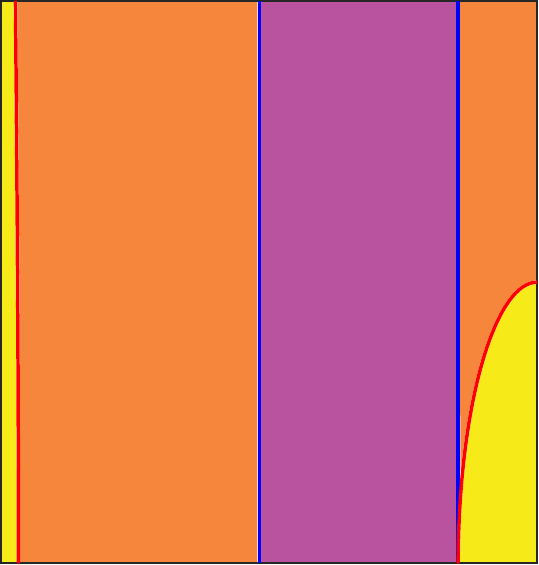}}
    \put(50.6,38){\includegraphics[width=0.12\textwidth]{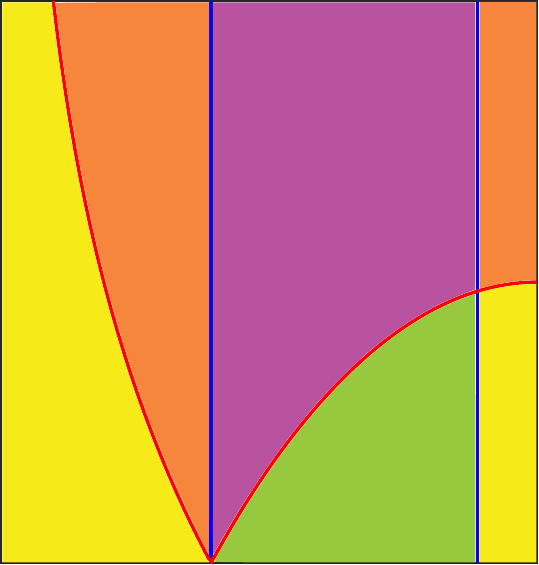}}
    \put(63,25.5){\includegraphics[width=0.12\textwidth]{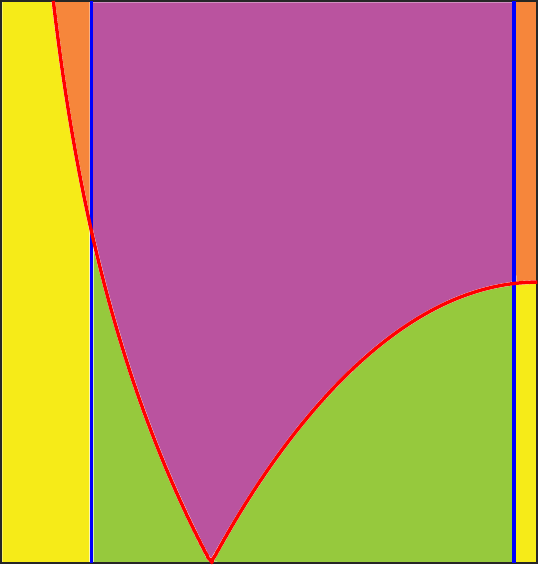}}
    \put(63,50.4){\includegraphics[width=0.12\textwidth]{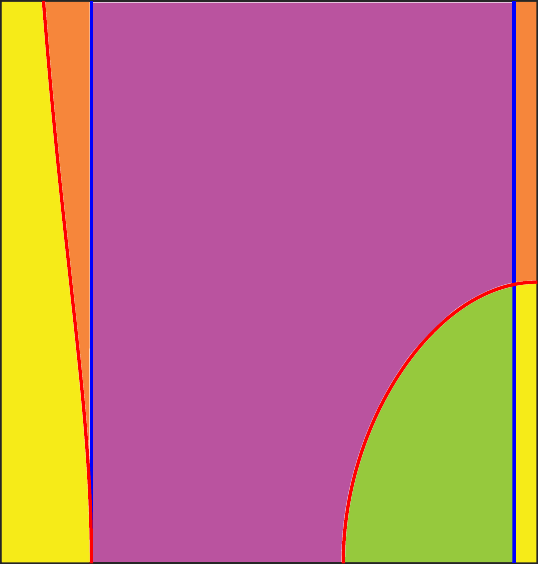}}
    \put(13.3,5){$\mu\in[0,\mu_\mathrm{P})$}
    \put(25.8,5){$\mu=\mu_\mathrm{P}$}
    \put(38,5){$\mu\in[\mu_\mathrm{P},\mu_\mathrm{L})$}
    \put(50.5,5){$\mu=\mu_\mathrm{L}$}
    \put(63,5){$\mu>\mu_\mathrm{L}$}
    \put(0,18.5){$|\Psi|\in[0,\Psi_\mathrm{L})$}
    \put(0,31.5){$|\Psi|=\Psi_\mathrm{L}$}
    \put(0,43.5){$|\Psi|\in[\Psi_\mathrm{L},\Psi_1)$}
    \put(0,55.5){$|\Psi|=\Psi_1$}
    \put(0,68.5){$|\Psi|\in[\Psi_1,\Psi_2)$}
    \put(0,80.4){$|\Psi|=\Psi_2$}
    \put(0,93.5){$|\Psi|>\Psi_2$}
    \put(13.5,14){\circled{1}}
    \put(21.3,72.5){\circled{2}}
    \put(50.8,14){\circled{3}}
    \put(40.9,39){\circled{4}}
    \put(70.5,39){\circled{5}}
    \put(52.5,72.5){\circled{6}}
    \put(51.5,88.9){\circled{7}}
    \end{overpic}
    \caption{Here we show all possible codimension-1 and codimension-2 bifurcations between cases in $(\mu,\Psi)$ space. This figure accompanies \cref{fig:mechanisms,fig:muPsicases}.}
    \label{fig:mechbifs}
\end{figure}

\bibliographystyle{siamplain}

\bibliography{library_paper}
\end{document}

%% file: ex_shared.tex
\usepackage{lipsum}

\usepackage{cite}
\usepackage{amsfonts}
\usepackage{graphicx}
\usepackage{epstopdf}
\usepackage{algorithmic}
\ifpdf
  \DeclareGraphicsExtensions{.eps,.pdf,.png,.jpg}
\else
  \DeclareGraphicsExtensions{.eps}
\fi

\usepackage{xfrac}
\usepackage{animate}
\usepackage{enumerate}
\usepackage[shortlabels]{enumitem}
\usepackage[a4paper,margin=25mm]{geometry}    
\usepackage{setspace} 
\usepackage{amsfonts,amssymb,amsmath}         
\usepackage{graphicx}                         
\usepackage{siunitx}                          
\usepackage{hyperref}						  
\usepackage{subcaption}						  
\usepackage{wrapfig}						  
\usepackage{lipsum} 
\usepackage{xcolor}
\definecolor{myblue}{rgb}{.1, .4, .8}
\definecolor{myred}{rgb}{.8, .1, .1}
\definecolor{myyellow}{RGB}{246,235, 24}
\definecolor{mylime}{RGB}{150,201, 61}
\definecolor{mypurple}{RGB}{186,83, 159}
\definecolor{myorange}{RGB}{246,134, 59}
\definecolor{matlab1}{rgb}{0, 0.4470, 0.7410}
\definecolor{matlab2}{rgb}{0.8500, 0.3250, 0.0980}
\definecolor{matlab3}{rgb}{0.9290, 0.6940, 0.1250}
\definecolor{matlab4}{rgb}{0.4940, 0.1840, 0.5560}
\usepackage[usestackEOL]{stackengine}
\usepackage[T1]{fontenc}
\usepackage{animate}
\usepackage{booktabs}
\usepackage{longtable}
\usepackage{adjustbox}
\usepackage{mdframed}
\usepackage{todonotes}
\usepackage{xargs}
\usepackage{xstring}

\newsiamremark{remark}{Remark}
\crefname{remark}{Remark}{Remark}
\crefname{Example}{Example}{Examples}
\newsiamremark{hypothesis}{Hypothesis}
\crefname{hypothesis}{Hypothesis}{Hypotheses}
\newsiamthm{claim}{Claim}
\newsiamremark{assumption}{Assumption}
\crefname{assumption}{Assumption}{Assumption}

\headers{The Geometry of the Painlev\'e paradox}{ N. Cheesman, S. J. Hogan, K. Uldall Kristiansen}

\title{The Geometry of the Painlev\'e paradox\thanks{Submitted to the editors {26/10/2021}.
\funding{{This work was funded by the UK Engineering and Physical Sciences Research Council (EPSRC) as part of the first author's PhD Studentship.}}}}

\author{
Noah Cheesman\thanks{Department of Engineering Mathematics, University of Bristol, Bristol BS8 1TW, UK
(\email{noah.cheesman@bristol.ac.uk}, \email{s.j.hogan@bristol.ac.uk}).}
 \and S.J. Hogan\footnotemark[2]
\and Kristian Uldall Kristiansen\thanks{Department of Applied Mathematics and Computer Science, Technical University of Denmark, 2800 Kgs. Lyngby, Denmark  
  (\email{krkri@dtu.dk}).}}

\usepackage{amsopn}

\newcommand{\diff}[2]{\frac{\mathrm{d}{#1}}{\mathrm{d}{#2}}}				
\newcommand{\pdiff}[2]{\frac{\partial{#1}}{\partial{#2}}}					
\renewcommand{\vec}[1]{\boldsymbol{#1}}										

\newcommand{\sign}{\mathrm{sign}}

\newcommand{\tr}{\mathrm{tr}}

\renewcommand{\vec}[1]{{#1}}
\renewcommand{\vec}[1]{\mathbf{#1}}
\newcommand{\mat}[1]{\mathbf{#1}}
\renewcommand{\mat}[1]{\underline{\vec{#1}}}
\renewcommand{\mat}[1]{{\underline{\underline{\mathrm{\smash{#1}}}}}}
\renewcommand{\vec}[1]{\underline{\smash{#1}}}

\usepackage{overpic}

\usepackage{multirow}

\usepackage{scalerel}

\usepackage{todonotes}
\newcommandx{\john}[2][1=]{\todo[linecolor=red,backgroundcolor=red!25,bordercolor=red,#1]{#2}}
\newcommandx{\johnin}[2][1=]{\todo[inline,linecolor=red,backgroundcolor=red!25,bordercolor=red,#1]{#2}}
\newcommandx{\noah}[2][1=]{\todo[linecolor=blue,backgroundcolor=blue!25,bordercolor=blue,#1]{#2}}
\newcommandx{\noahin}[2][1=]{\todo[inline,linecolor=blue,backgroundcolor=blue!25,bordercolor=blue,#1]{#2}}
\newcommandx{\kristian}[2][1=]{\todo[linecolor=green,backgroundcolor=green!25,bordercolor=green,#1]{#2}}
\newcommandx{\kristianin}[2][1=]{\todo[inline,linecolor=green,backgroundcolor=green!25,bordercolor=green,#1]{#2}}

\newlist{thmlist}{enumerate*}{1}
\setlist[thmlist]{label=\rm{(\arabic*)}}

\crefname{thmlisti}{thm}{thms}
\creflabelformat{thmlisti}{#2#1#3}

\usepackage[normalem]{ulem}

\definecolor{myblue}{rgb}{.1, .4, .8}
\definecolor{myred}{rgb}{.8, .1, .1}
\definecolor{myyellow}{rgb}{0.9647,0.9216,0.0941}
\definecolor{mylime}{rgb}{0.5882,0.7882,0.2392}
\definecolor{mypurple}{rgb}{0.7294,0.3255,0.6235}
\definecolor{myorange}{rgb}{0.9647,0.5255,0.2314}
\definecolor{md1}{rgb}{0, 0.4470, 0.7410}
\definecolor{md2}{rgb}{0.8500, 0.3250, 0.0980}
\definecolor{md3}{rgb}{0.9290, 0.6940, 0.1250}
\definecolor{md4}{rgb}{0.4940, 0.1840, 0.5560}
\definecolor{md5}{rgb}{0.4660, 0.6740, 0.1880}
\definecolor{md6}{rgb}{0.3010, 0.7450, 0.9330}
\definecolor{md7}{rgb}{0.6350, 0.0780, 0.1840}
\usepackage{multirow}

\usepackage{amssymb}

\usepackage{tikz}
\usepackage{pgfplots}
\newcommand*\circled[1]{\tikz[baseline=(char.base)]{
            \node[shape=circle,draw,inner sep=2pt] (char) {#1};}}

\usepackage{xargs}

\DeclareSymbolFont{matha}{OML}{txmi}{m}{it}
\DeclareMathSymbol{\varv}{\mathord}{matha}{118}

\usepackage{overpic}

\usepackage{todonotes}

\usepackage{subcaption}
\renewcommand{\vec}[1]{\boldsymbol{\mathbf{#1}}}
\renewcommand{\mat}[1]{\underline{\vec{#1}}}	
\renewcommand{\diff}[2]{\frac{\mathrm{d}{#1}}{\mathrm{d}{#2}}}
\renewcommand{\pdiff}[2]{\frac{\partial{#1}}{\partial{#2}}}
\renewcommand{\sign}[1]{\mathrm{sign}{#1}}

\RequirePackage[capitalize,nameinlink]{cleveref}[0.19]

\crefname{section}{section}{sections}
\crefname{subsection}{subsection}{subsections}
\Crefname{section}{Section}{Sections}
\Crefname{subsection}{Subsection}{Subsections}

\Crefname{figure}{Figure}{Figures}

\crefformat{equation}{\textup{#2(#1)#3}}
\crefrangeformat{equation}{\textup{#3(#1)#4--#5(#2)#6}}
\crefmultiformat{equation}{\textup{#2(#1)#3}}{ and \textup{#2(#1)#3}}
{, \textup{#2(#1)#3}}{, and \textup{#2(#1)#3}}
\crefrangemultiformat{equation}{\textup{#3(#1)#4--#5(#2)#6}}%
{ and \textup{#3(#1)#4--#5(#2)#6}}{, \textup{#3(#1)#4--#5(#2)#6}}{, and \textup{#3(#1)#4--#5(#2)#6}}

\Crefformat{equation}{#2Equation~\textup{(#1)}#3}
\Crefrangeformat{equation}{Equations~\textup{#3(#1)#4--#5(#2)#6}}
\Crefmultiformat{equation}{Equations~\textup{#2(#1)#3}}{ and \textup{#2(#1)#3}}
{, \textup{#2(#1)#3}}{, and \textup{#2(#1)#3}}
\Crefrangemultiformat{equation}{Equations~\textup{#3(#1)#4--#5(#2)#6}}%
{ and \textup{#3(#1)#4--#5(#2)#6}}{, \textup{#3(#1)#4--#5(#2)#6}}{, and \textup{#3(#1)#4--#5(#2)#6}}

\crefdefaultlabelformat{#2\textup{#1}#3}

\usepackage[shortlabels]{enumitem}

\newlist{mycases}{enumerate}{3}
\setlist[mycases,1]{label=\textbf{Case~\arabic*},
                   ref  =\arabic*}
\setlist[mycases,2]{label=\textbf{(\alph*)},
                   ref  =\themyenumi\textbf{(\alph*)}}
\setlist[mycases,3]{label=\bfseries(\roman*),
                   ref  =\themyenumii\textbf{.(\roman*)}}

\crefname{mycasesi}{Case}{Cases}
\crefname{mycasesii}{item}{items}
\crefname{mycasesiii}{item}{items}

\newlist{mymechs}{enumerate}{3}
\setlist[mymechs,1]{label=\textbullet,
                   ref  =\Roman*}
\setlist[mymechs,2]{label=--,
                   ref  =\themymechsi\textbf{(\alph*)}}
\setlist[mymechs,3]{label=\bfseries(\roman*),
                   ref  =\themymechsii\textbf{.(\roman*)}}

\crefname{mymechsi}{Mechanism}{Mechanisms}
\crefname{mymechsii}{item}{items}
\crefname{mymechsiii}{item}{items}
\pgfplotsset{compat=1.16}